\documentclass[12pt]{amsart}
\usepackage{amssymb}
\usepackage{amsmath}
\usepackage{graphicx}
\usepackage{fullpage}
\usepackage{subfig}
\usepackage{url}
\usepackage{enumerate}
\usepackage[latin1]{inputenc}
\usepackage{xcolor}
\usepackage{mathrsfs}  
\usepackage{mathtools}  
\usepackage{pdflscape}   
\usepackage{multirow}    
\usepackage{fdsymbol}  
\usepackage{listings}
\usepackage{bbding}
\usepackage[font=small,labelfont=bf]{caption}
\usepackage{array}

\usepackage[abs]{overpic} 

\newcommand{\R}{\mathbb{R}}
\newcommand{\Z}{\mathbb{Z}}
\newcommand{\C}{\mathbb{C}}

	\definecolor{alizarin}{rgb}{0.82, 0.1, 0.26}

\numberwithin{figure}{section}

\numberwithin{equation}{section}

\numberwithin{table}{section}

\makeatletter
\let\c@equation\c@figure
\makeatother

\makeatletter
\let\c@table\c@figure
\makeatother

\makeatletter
\let\c@algorithm\c@figure
\makeatother

\newtheorem{theorem}[equation]{Theorem}
\newtheorem{proposition}[equation]{Proposition}
\newtheorem{lemma}[equation]{Lemma}
\newtheorem{corollary}[equation]{Corollary}
\newtheorem{conjecture}[equation]{Conjecture}
\newtheorem{assumption}[equation]{Assumption}

\theoremstyle{definition}
\newtheorem{definition}[equation]{Definition}
\newtheorem{remark}[equation]{Remark}
\newtheorem{example}[equation]{Example}

\newcommand{\eqdef}{\;{:=}\;}
\newcommand{\op}{\operatorname}

 \DeclarePairedDelimiter\floor{\lfloor}{\rfloor}  

\newcommand{\inner}{\text{in}}
\newcommand{\outter}{\text{out}}
\newcommand{\vol}{\text{vol}}
\newcommand{\per}{\text{per}}
\newcommand{\ds}{\displaystyle}

\newcommand{\embeds}{{\stackrel{s}{\hookrightarrow}}}

\newcommand\bigsubset[1][1.19]{%
   \mathrel{\vcenter{\hbox{\scalebox{#1}{$\subset$}}}}}

\begin{document}
\title[Infinite staircases]
{On infinite staircases in toric symplectic  four-manifolds}

\author[Cristofaro-Gardiner]{Dan Cristofaro-Gardiner}
\address{University of Maryland}
\email{dcristof@umd.edu}

\author[Holm]{Tara S. Holm}
\address{Cornell University}
\email{tsh@math.cornell.edu}

\author[Mandini]{Alessia Mandini}
\address{Universidade Federal Fluminense}
\email{alessia\_mandini@id.uff.br}

\author[Pires]{Ana Rita Pires}
\address{University of Edinburgh}
\email{apires@ed.ac.uk}

\date{7 August 2024}

\begin{abstract}
An influential result of McDuff and Schlenk asserts that the function that encodes when a four-dimensional symplectic ellipsoid can be embedded into a four-dimensional ball has a remarkable structure: the function has infinitely many corners, determined by the odd-index Fibonacci numbers, that fit
together to form an infinite staircase.  

This work has led to considerable interest in understanding when the ellipsoid embedding function for other symplectic $4$-manifolds is partly described by an infinite staircase.   
We provide a general framework for analyzing this question for a large family of targets, called finite type convex toric domains, which we prove generalizes the class of closed toric symplectic four-manifolds.  When the target is of finite type, we prove that any infinite staircase must have a unique accumulation point $a_0$, given as the solution to an explicit quadratic equation.  Moreover, we prove that the embedding function at $a_0$ must be equal to the classical volume lower bound.  In particular, our result gives an obstruction to the existence of infinite staircases that experimentally seems strong.  Our methods apply equally well to any closed manifold which is a blowup of the complex projective plane.

In the special case of rational convex toric domains, we can say more.  We conjecture a complete answer to the question of existence of infinite staircases, in terms of six families that are distinguished by the fact that their moment 
polygon is reflexive.  We then provide a uniform proof of the existence of infinite staircases 
for our six families, using two tools.
For the first, we use recursive families of almost toric fibrations to find symplectic embeddings into closed symplectic manifolds.  
In order to establish the embeddings for convex toric domains, we prove a result of potentially independent interest:
a four-dimensional ellipsoid embeds 
into a closed toric symplectic  four-manifold if and only if it can be embedded into a corresponding
convex toric domain.
For the second tool,
we find recursive families of convex lattice paths that provide obstructions to embeddings.  
We conclude by connecting our conjecture that these are the only infinite staircases among rational convex toric domains 
to a question in number theory related to a classic work of
Hardy and Littlewood.
\end{abstract}

\maketitle

\tableofcontents

\section{Introduction}

In the past two decades,
there has been 
considerable interest in and progress on the question of whether there is an embedding
$$
(M,\omega_M) \embeds (N,\omega_N)
$$
preserving symplectic structures, or whether the existence of such a map is in some way obstructed.
On the one hand, Local Normal Form theorems and clever constructions like symplectic folding and symplectic inflation
allow us to find embeddings.  On the other hand,  
there are 
well-developed tools involving 
pseudo-holomorphic curves that provide numerous obstructions to these maps.  

A landmark result about this is the celebrated work of 
McDuff and Schlenk, 
who completely determined when a four-dimensional ellipsoid 
$$ E(a,b) = \left\lbrace (z_1,z_2)\in\mathbb{C}^2: \pi \left(\frac{|z_1|^2}{a} + \frac{|z_2|^2}{b} \right) < 1\right\rbrace$$ 	
can be symplectically embedded into a  ball $B(r)=E(r,r)$.  They found that when the ellipsoid is close to round, the answer is extremely delicate, determined by the odd-index Fibonacci numbers and the Golden Mean.  On the other hand, if the ellipsoid is sufficiently stretched then the only obstruction to an embedding is the classical volume obstruction.

Since the work of McDuff and Schlenk, many examples of infinite staircases have been found; we will survey what is known below.  There is moreover an extensive literature on
symplectic embedding problems where the domain is an ellipsoid:
\cite{CV, chls, cg ellipsoid, cgfs polydisk,  cghind, ghost, cgk, frenkelmuller, hind, hindkerman, h, qech, mcd, m2, 
m3, mcduffschlenk, kyler1, kyler2, usher}.   However, a general theory of infinite staircases does not currently exist.  
We do not know how characteristic infinite staircases are for symplectic embedding problems. 
For any fixed target, we do not know how to determine if there is an infinite staircase.  
We 
also 
do not 
know if all examples of this phenomenon can be unified in an elegant way or whether 
 their corresponding symplectic embeddings share common features.  There are other mysteries: 
for instance, in every  target known to admit an infinite staircase, 
the stairs alternate between being horizontal and being linear  with no constant term.  

Our interest here is in taking some useful first steps towards better understanding these questions.

To fix notation, let $X$ be a symplectic $4$-manifold.  We write $E\embeds X$ to mean that there is a symplectic embedding of $E$ into $X$, and define the \textbf{ellipsoid embedding function} of $X$ by
\begin{equation}\label{definition of c}
c_X(a):=\inf\left\{\lambda\  \big| \ E(1,a)\embeds\lambda X\right\}, \mbox{ for } a\geq 1,
\end{equation}
where $\lambda X$ represents the symplectic scaling $(X,\lambda\cdot \omega)$ of
$(X,\omega)$.
We could have defined the function
for $a > 0$, 
but there is a symmetry across $a=1$, making this redundant. This is a nondecreasing function; it is also continuous, although we will not prove that in full generality here. Inspired by the McDuff-Schlenk result, we make the following definition. 

\begin{definition}
We say that the ellipsoid embedding function $c_X(a)$ has an \textbf{infinite staircase}
if its graph has infinitely many non-smooth points.
\end{definition}

For a general $X$ with no further assumptions, computing the function $c_X(a)$ exactly or even determining whether or not it has an infinite staircase is presumably impossible due to the subtle mix of obstructions and constructions alluded to above.  For example, when $X$ is star-shaped domain, it is known that the periods of certain Reeb orbits on the boundary give obstructions to the existence of an embedding. 
Moreover one has the impression that a certain amount of additional structure on $X$ is 
required for an infinite staircase to exist.  So, one should make some restrictions on $X$.  The starting point for our investigations is the following question.  Let $M$ be a {\bf closed toric symplectic four-manifold} associated to a {\bf Delzant moment polygon} $\Omega$ in $\R^2$; this is a rich source of interesting examples.  Can we determine from $\Omega$ whether or not $c_M$ has an infinite staircase?  The answers we have found 
are governed by beautiful combinatorics and
number theory.

\begin{center}
  \begin{figure}[ht]
\begin{overpic}[
scale=1,unit=1mm]{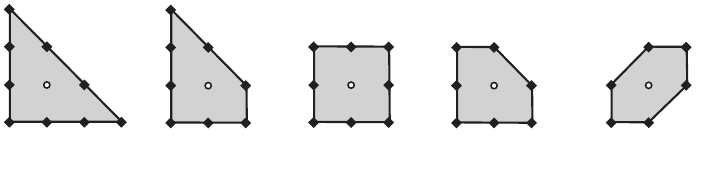}
   \put(7,3){(a)}
   \put(32,3){(b)}
   \put(57,3){(c)}
   \put(81,3){(d)}
   \put(107,3){(e)}
\end{overpic}
\caption{This illustrates several regions in $\R^2_{\geq 0}$ that are Delzant polygons (with bottom left corner at the origin).
These polygons correspond to the closed manifolds
(a) $\C P^2$; (b) $\C P^2 \# \overline{\C P}^2$; (c) $\C P^1\times{\C P}^1$; (d) $\C P^2 \# 2\overline{\C P}^2$;
and (e) $\C P^2 \# 3\overline{\C P}^2$.
As we will see, the ellipsoid embedding functions for these examples
admit infinite staircases. }

\label{fig:delzant}
\end{figure}
\end{center}

In fact, most of our results should hold for more general closed four-manifolds $M$, for example rational $M$ or Hamiltonian $S^1$-manifolds, or even uniruled $M$, see \S\ref{sec:other}, but that is not our focus here.

To state our results, we now introduce the related notion of toric domains.
A 4-dimensional \textbf{toric domain} $X$ is the preimage of a domain $\Omega\subset \mathbb{R}_{\geq 0}^2$ under the moment map 
$$\mu:\mathbb{C}^2\to\mathbb{R}^2, \qquad (z_1,z_2)\mapsto(\pi |z_1|^2,\pi |z_2|^2).$$ 
For example, if $\Omega$ is the hypotenuse-less triangle with vertices $(0, 0)$, $(a, 0)$, and $(0, b)$, then $X_\Omega$ is the open ellipsoid $E(a, b)$.
Following the notation set forth in \cite[Definition~1.1]{dan}, a \textbf{convex toric domain} is the preimage under $\mu$ of a closed 
region $\Omega\subset \mathbb{R}_{\geq 0}^2$ that is convex, connected, and contains the origin in its interior (in the
subspace topology on $\mathbb{R}_{\geq 0}^2\subset \R$).
We denote this $X_\Omega=\mu^{-1}(\Omega)$
and call the region $\Omega$ the {\bf moment polygon} of 
$X_{\Omega},$ in 
analogy with the case of closed toric symplectic manifolds, 
reviewed in \S\ref{subsec:toric}.
By contrast, we note that $X_\Omega$ is a manifold with boundary,
while the closed toric symplectic manifold is a compact
manifold (without boundary).  We illustrate
five examples of moment polygons for convex toric domains in 
Figure~\ref{fig:delzant}.  For example, the convex toric domain
corresponding to the region Figure~\ref{fig:delzant}(a)
is the closed ball $\overline{B(3)}$ and the toric symplectic manifold corresponding to this Delzant polygon is $\C P^2$ obtained by collapsing the boundary of $\overline{B(3)}$ by the Hopf fibration.

One motivation for studying convex toric domains comes from 
the following result that we prove, 
which ties together the ellipsoid embedding functions for
closed toric manifolds and convex toric domains.   
This result features essentially in our proof of 
Theorem~\ref{thm:main} as well.

\begin{theorem}\label{prop:mainproposition}
Let $\Omega\subset \mathbb{R}_{\geq0}^2$ be a convex region 
that is also a Delzant polygon for a closed toric symplectic four-manifold $M$.  
Then there exists a symplectic embedding
\begin{equation}\label{eqn:desiredembedding}
E(d,e) \embeds M
\end{equation}
if and only if there exists a symplectic embedding
\begin{equation}\label{eqn:desiredembedding2}
E(d,e) \embeds X_\Omega.
\end{equation}
\end{theorem}

\begin{remark}
Delzant proved that toric symplectic manifolds are in one-to-one
correspondence with their moment map images, which must be 
Delzant polytopes \cite{delzant}.
Generalizing this to the orbifold case, Lerman and Tolman
find a one-to-one correspondence between toric symplectic orbifolds
and labeled simple, rational polytopes \cite{lerman-tolman}.
Theorem~\ref{prop:mainproposition} likely extends to some 
class of orbifolds.
Proofs for specific examples
may be straightforward, but even the starting point
for a more general statement in the orbifold setting is beyond
the scope of the current paper.
\end{remark}

Thus, from the point of view of the function $c_X$, convex toric domains significantly generalize closed toric manifolds.  In fact, we can relate embeddings into convex toric domains to embeddings into closed manifolds in a slightly more general context, 
including some well-known examples, for example  equilateral pentagon space: see
Corollary~\ref{prop:mainproposition2} and Lemma~\ref{rem:polygon}.

For a general convex toric domain $X$, the embedding function $c_X(a)$ has an interesting qualitative  structure.  
For a fixed $X$, the \textbf{volume curve} is the curve $y=\sqrt{\frac{a}{\vol}}$
and the {constraint} $$c_X(a)\geq \sqrt{\frac{a}{\vol}}$$  
holds because
\begin{eqnarray*}
    E(1,a)\embeds\lambda X & \Rightarrow & \text{volume}(E(1,a))\leq\text{volume}(\lambda X)\\
& \Leftrightarrow & a\leq \lambda^2 \text{volume}(X)\\
& \Leftrightarrow & \lambda\geq \sqrt{\frac{a}{\vol}}.
\end{eqnarray*}
We will show in Proposition~\ref{prop:properties} that $c_X(a)$ 
is continuous and non-decreasing, but not generally $C^1$. For sufficiently large $a$, we also show that the function 
$c_X(a)$ remains equal to the volume curve: this is the phenomenon known as {\bf packing\footnote{\phantom{.}``Packing" refers to
fact that  embedding  $E(1,k)$ is equivalent to embedding $k$ disjoint balls $B(1)$.} stability}.
Moreover, the function $c_X(a)$ is piecewise linear 
when not equal to the volume curve, except at points that are limit points of singular\footnote{\phantom{.}We call a non-smooth point of $c_X$ a singular point, and we use these terms interchangeably.} points of $c_X$.      
We call these limit points  \textbf{accumulation points} and they are an important focus of this paper.

\begin{remark}
In \cite[Definition 1.1]{CV}, Casals and Vianna work with a different concept, a {\bf sharp} infinite staircase.
This is an infinite staircase where infinitely many of the 
non-smooth points must be on the volume curve. That notion therefore excludes the $J=3$ cases below (cf.\ 
Remark~\ref{rem: all cases} and Figure~\ref{fig:4pics}(b)).
\end{remark}

All known examples of infinite staircases are associated to convex toric domains and we now survey what is known. 
In the paper  \cite{cgk}, Cristofaro-Gardiner and Kleinman studied the ellipsoid embedding function
of an ellipsoid $c_{E(1,b)}(a)$ 
and found infinite staircases when $b=2$ and
$b=\frac32$.  Frenkel and M\"uller found an infinite staircase in the ellipsoid embedding function
for a polydisc $P(1,1)$  \cite{frenkelmuller}, where the function is governed by the
Pell numbers.  
Cristofaro-Gardiner, Frenkel, and Schlenk have shown that
the only infinite staircase in the ellipsoid embedding function for 
polydisks $P(1,b)$  with $b\in \mathbb{N}$ is when $b=1$ \cite{cgfs polydisk}.
By contrast,
Usher studied ellipsoid embedding functions for 
irrational polydisks $P(1,b)$  \cite{usher} and found the first infinite families of infinite
staircases.  Usher's families all have $b$ quadratic irrationalities of a special form.  
 
Further work, in part making use of our Theorem~\ref{satisfies quadratic equation} below,  by Bertozzi, Holm, Maw, Mwakyoma, McDuff, Pires, and Weiler \cite{project9}, Magill and McDuff \cite{mm}, and Magill, McDuff and Weiler \cite{mmw} studies ellipsoid embedding functions for one point blowups of $\mathbb{C}P^2$, varying the symplectic size of the blowup.
Like Usher, 
they establish infinite families of infinite staircases; 
they also identify infinitely many 
subintervals of $[0,1]$ such that for any $b$ in one of the subintervals,
there is no infinite staircase for the ellipsoid embedding function of
the symplectic blowup $\C P^2_1\#\overline{\C P}^2_b$.
We say more about their work at the end of Section \ref{sec:reflex} below.

\medskip

\subsection{Accumulation points for infinite staircases}
\label{sec:acc}

Our first main result is aimed at rooting out the ``germs" of infinite staircases.  
If the function $c_X(a)$ has an infinite staircase, then the singular points must accumulate at 
some set of points, otherwise this would contradict packing stability.  We show that if an infinite 
staircase exists, there is in fact a unique such accumulation point; moreover, it can be 
characterized as the solution to an explicit quadratic equation 
determined by $\Omega$.

We now make this precise, following \cite[\S2.2]{dan}.  To a convex toric domain $X_{\Omega}$ 
we can associate a \textbf{negative weight expansion}  
$(b;b_1,b_2,\ldots)$. To do so, first we need to 
define a $b$-triangle to be the triangle with vertices $(0,0)$, $(b,0)$ and 
$(0,b)$ or any $AGL(2,\mathbb{Z})$ transformation of it\footnote{\phantom{.}By 
$AGL(2,\mathbb{Z})$ transformation we mean a $GL(2,\mathbb{Z})$ 
transformation followed by an affine translation by a (real) vector.}. We proceed inductively: 
let $b$ be the smallest number such that $\Omega$ is contained in a $b$-triangle. 
If $\Omega$ 
equals that triangle, we are done. Otherwise, let $b_1>0$ be the largest value
such that $\Omega$ 
is contained in the original $b$-triangle minus a $b_1$-triangle that is removed at a 
corner of the $b$-triangle. If $\Omega$ equals this quadrilateral, we are done. 
Otherwise, let $b_2>0$ be the largest value such that $\Omega$ is contained in the previous 
quadrilateral minus a $b_2$-triangle that is removed at one of its corners. 
The removing of the $b_i$-triangles is reminiscent of what is done to the moment 
polytope when performing equivariant symplectic blowups at fixed points.
We note that when 
$\Omega$ is a lattice polygon, this procedure is finite.

We note that two different convex toric domains can have the same  
negative weight expansion, see Figure \ref{fig:twelve} for examples. We will  see  in 
Remark~\ref{rmk:dependsonly}
that  the relevant feature of a convex toric domain in the context of this paper is its 
negative weight expansion, and not the actual shape of $\Omega$.  
\begin{definition}\label{def:finiterational}
We say that a negative weight expansion 
$$(b;b_1,b_2,\ldots)$$ 
is \textbf{finite} if there are finitely 
many non-zero $b_i$'s, and we say that such a convex toric domain $X_{\Omega}$ has {\bf finite type}.
Given a finite type convex toric domain $X$ 
with negative weight expansion $(b;b_1,\ldots,b_n)$, we define:
\begin{equation*}\begin{split}
\per & = 3b-\sum_{i=1}^n b_i \\
\vol & = b^2-\sum_{i=1}^n b_i^2.
\end{split}\end{equation*}
We say that $X_\Omega$ is a \textbf{rational convex toric domain} if $b$, $b_1$, $\ldots$, 
and $b_n$ are all rational numbers.
\end{definition}

\begin{figure}[h!]
\centering
\includegraphics[width=0.9\textwidth]{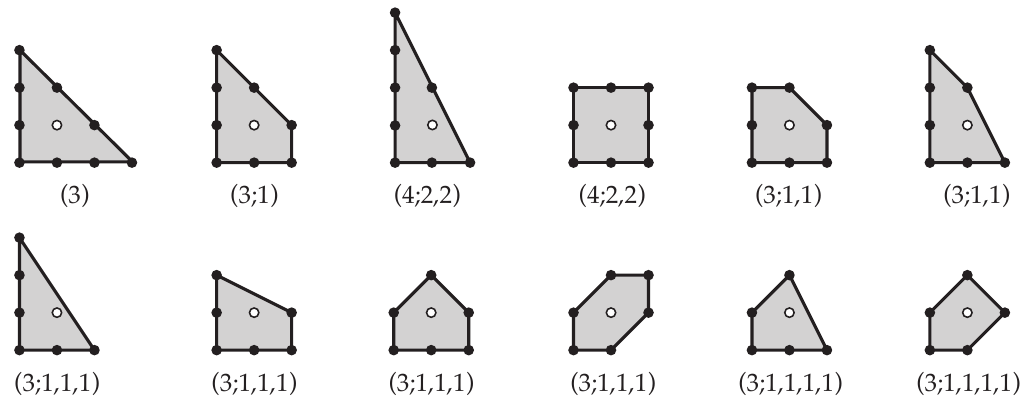}
\caption{Regions corresponding to finite type rational convex toric domains, and their negative weight expansions $(b;b_1,\ldots,b_n)$.
These negative weight expansions are precisely the ones that
feature in Theorem~\ref{thm:main}.
The negative weight expansions $(3;1)$ and $(3;1,1)$ correspond to the $J=3$ case and all others are $J=2$; cf.\
Table~\ref{table:recurrence} and Remark~\ref{rem: all cases}. Note that all of these polygons are
reflexive: each is a lattice polygon with precisely one interior lattice point.}
\label{fig:twelve}
\end{figure}

\newpage

\begin{remark}\label{rem:newest}
Finite type convex toric domains are a generalization of closed toric symplectic four-manifolds as targets for ellipsoid embedding functions 
in the following sense. Theorem \ref{prop:mainproposition} makes a correspondence between closed toric symplectic 
manifolds and convex toric domains that have the same moment image in $\mathbb{R}^2$. The proof of 
Theorem~\ref{prop:mainproposition} explains why the convex toric domains that arise in this way are always of finite 
type, while Lemma~\ref{rem:polygon} discusses an example of a finite type convex toric domain that does not arise 
from a closed toric symplectic four-manifold.

We note that all convex toric domains in Theorem \ref{thm:main} are rational.  The significance of this property is also highlighted by Conjecture \ref{conj:reflex}.
\end{remark}

\begin{remark}\label{rem:invt}
The quantities $\per$ and $\vol$ are, respectively, the affine 
perimeter\footnote{\phantom{.} We recall for the convenience of the reader that the affine perimeter of a polygon is the sum of the affine lengths of its sides.  The affine length $\ell_A$ of a vector $v \in \mathbb{R}^2$ is the unique $SL_2(\mathbb{Z})$-invariant function that vanishes on vectors of irrational slope, satisfies $\ell_A( (1,0) ) = 1$, and satisfies $\ell_A( c\cdot v) = c\cdot \ell_A(v)$ for all vectors $v$ and scalars $c$: to compute this, by the scaling axiom it suffices to compute 
for integral vectors, and then by invariance this is one less than the number of integer lattice points that the integral vector contains. This is sometimes called the {\bf rational length}, for example in \cite[\S 2.4]{kkp}.}
and twice the area of the region in $\mathbb{R}^2$ representing 
$X$. They are well-defined  invariants of this region.
In particular, $\vol$ is the {\bf symplectic volume} of $X$.
Note also that $\frac{\per^2}{\vol}$ is invariant under 
scaling of (the region representing) $X$.
\end{remark}


\noindent We can now state precisely our theorem about finding accumulation points.

\begin{theorem}\label{satisfies quadratic equation}
Let $X$ be a finite type convex toric domain. 
If the ellipsoid embedding function $c_X(a)$ has an infinite staircase 
then it accumulates at $a_0$,  a real solution\footnote{
\phantom{.}Note that if equation \eqref{eqn:theequation} has two distinct real solutions, then there is a unique solution greater than 1.  Thus, when $a_0$ exists as in the statement of the theorem, it is unique on the domain of $c_X$.}  
of the quadratic equation
\begin{equation}
\label{eqn:theequation}
a^2-\left(\frac{\per^2}{\vol}-2\right)a+1=0.
\end{equation}
Furthermore, at $a_0$ the ellipsoid embedding function touches the volume curve:
$$c_X(a_0)=\sqrt{\frac{a_0}{\vol}} \phantom{..}.$$
\end{theorem}

We emphasize that, for any particular finite type target $X_{\Omega}$, Theorem~\ref{satisfies quadratic equation} leads 
to the following procedure for approaching the question of whether or not $X_{\Omega}$ has an infinite staircase.  
We compute the quantity $c_{X_\Omega}(a_0) - \sqrt{\frac{a_0}{\vol}} \ge 0$, which we call the {\bf staircase obstruction}.  
The staircase obstruction seems to give a strong indication about the possibility of an infinite staircase.
In our experience, it is almost always sharp.

If the staircase obstruction is positive, then an infinite 
staircase cannot exist.  For example, the staircase obstruction is positive in Figure~\ref{fig:4pics}(c) below, where we can see 
clearly that the ellipsoid embedding
function is obstructed at $a_0$, and so a toric domain with negative weight expansion $(4;2,1)$ 
cannot have an infinite staircase.  
Developing this idea, in the sequence of works \cite{project9, many, m, mm, mmw},  
Theorem~\ref{satisfies quadratic equation} is used
to identify intervals of $b$-values for which the ellipsoid embedding function of a toric domain with
negative weight expansion $(1;b)$ does not have an infinite staircase.

When the staircase obstruction vanishes, then by Theorem~\ref{satisfies quadratic equation}, 
if there is an infinite staircase, 
infinitely many of the steps must occur in a neighborhood of $a_0$.
In this case, 
one can attempt to explore the question numerically near $a_0$
to see if there is any evidence of an infinite staircase.
In all known examples, 
numerical exploration gives a very strong indication of whether or not an infinite 
staircase exists, and then further 
ideas can 
be used to rigorously establish the existence or nonexistence of an infinite staircase.  
We explain some instances of this analysis in Appendix~\ref{appendix}.

\begin{remark}  There are only a handful of known families of examples where the staircase obstruction
vanishes but there is no staircase.   
One such example is shown in Figure~\ref{fig:4pics}(d).  The negative weight
expansion $(4;1,1,1,1)$ corresponds to $X_\Omega=E(3,4)$, as shown in Figure~\ref{fig:4-1-1-1-1-is-ellipsoid}. 
For this example, the first named author
has rigorously computed a precise formula for the ellipsoid embedding
function in a neighborhood of $a_0$ and thereby proved that 
it does not have an infinite staircase \cite[Sec. 2.5]{cg ellipsoid}.  At the time that the current 
paper first appeared, this was the only known example where the staircase obstruction vanishes but 
there is no infinite staircase.  Subsequent work \cite{project9, many, m, mm, mmw} has since identified 
other families of examples of this phenomenon, corresponding to one point blowups of $\mathbb{C}P^2$.  
These papers classify precisely which one-point blowups of $\C P^2$ have an infinite staircase and in them,
a similar phenomenon holds: in all the cases where the staircase obstruction vanishes but there is no infinite 
staircase, one can rigorously compute a precise formula for the embedding function to verify that an 
infinite staircase does not exist.  We note for the interested reader that the methods for analyzing these 
cases where the staircase obstruction vanishes differ between these works: in \cite{project9, cg ellipsoid}, 
the nonexistence of a staircase is proved by using the fact 
that ECH capacities are sharp, in connection with some new lattice point estimates; 
in \cite{many, m, mm, mmw} , it is established by an analysis of the reduction algorithm.  
For more about the reduction algorithm, we refer the reader for example to the discussion 
in \cite{cgfs polydisk}.
\end{remark}

\begin{remark}
Theorem~\ref{satisfies quadratic equation} has an interesting interpretation in terms of the asymptotics of ECH capacities.  
We review the theory of ECH capacities in Section \ref{sec:ech}.   Theorem~\ref{satisfies quadratic equation} 
can be interpreted as saying that at an accumulation point of singular points, the leading and subleading 
asymptotics of the ECH capacities of the domain and target agree.  For more about this, see Remark~\ref{rmk:asymptotics}.

It might also be interesting to note that in the case of a rational convex toric domain, 
the coefficients of \eqref{eqn:theequation} are rational.  
In particular, $a_0$ is a quadratic surd for these examples. This is potentially useful for ruling out infinite 
staircases, as we further explore in Section \ref{sec:conj}.  
\end{remark}

\subsection{Reflexive polygons and infinite staircases}
\label{sec:reflex}

Having explained in the previous section where infinite staircases must accumulate, 
we now turn our attention to finding and describing them. We begin by showing in Figure~\ref{fig:4pics}
the types of graphs we can produce of embedding functions using Mathematica.
These types of plots, analyzed via Theorem~\ref{thm:main}, were essential in our early investigations of infinite staircases.
This is discussed further in Appendix~\ref{appendix}.

\newpage 

\begin{figure}[htp] 
\centering
\subfloat[ $(3;1,1,1,1)$]{%
\includegraphics[width=0.45\textwidth]{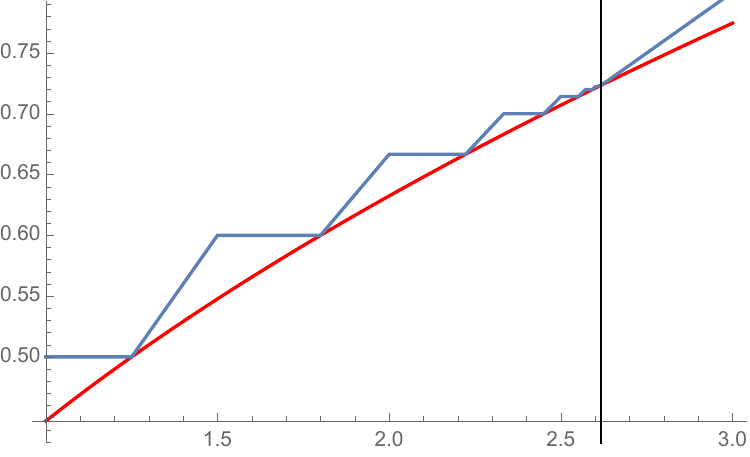}%
}\hfil
\subfloat[ $(3;1)$]{%
\includegraphics[width=0.45\textwidth]{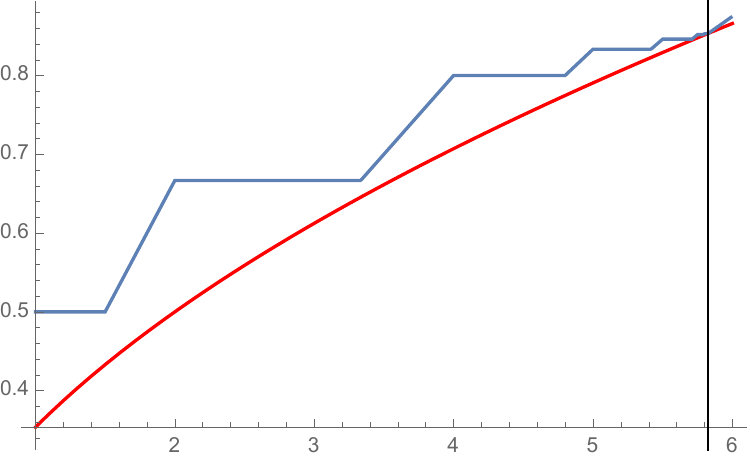}%
}

\subfloat[$(4;2,1)$]{%
\includegraphics[width=0.45\textwidth]{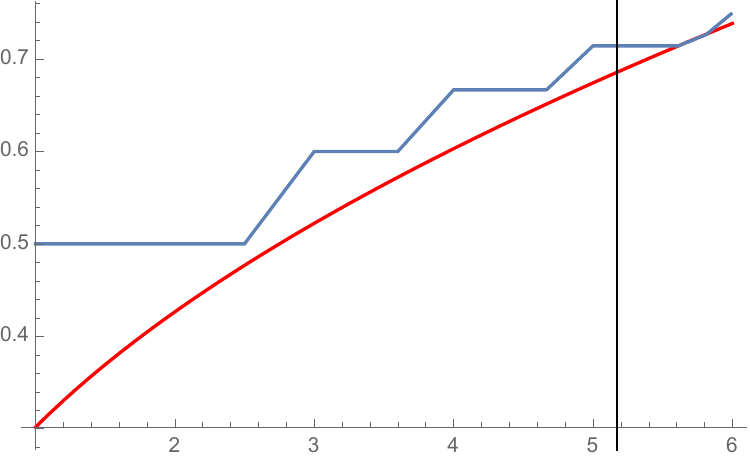}%
}\hfil
\subfloat[$(4;1,1,1,1)$]{%
\includegraphics[width=0.45\textwidth]{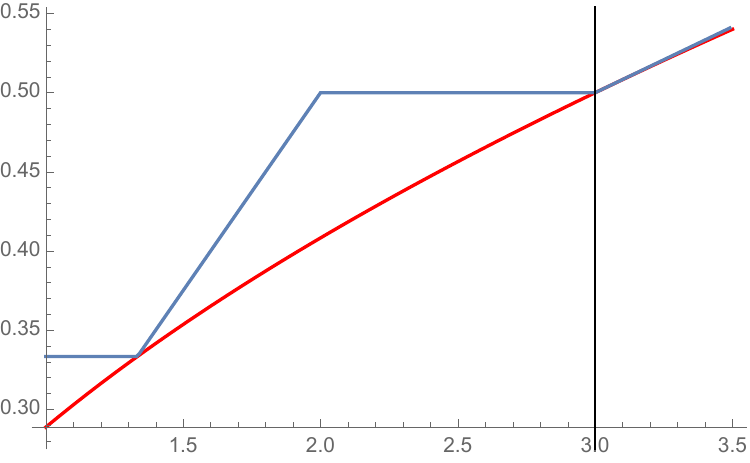}%
}
\caption{{\footnotesize Plots of ellipsoid embedding functions for different domains, 
labeled by their negative weight expansion.
The red curves are the volume curves and the vertical lines indicate where the 
accumulation points would necessarily be located, if a staircase existed, per
Theorem~\ref{satisfies quadratic equation}.
The top two plots have infinite staircases: in (a) we have a $J=2$ case, 
where the inner corners touch the volume curve (a sharp infinite staircase); and in (b)
we have a $J=3$ case, where the inner corners approach but do not touch the volume curve. 
The plots (c) and (d) do not have infinite staircases:
 (c) has non-zero staircase obstruction, while showing that (d) does not have a staircase is intricate \cite{cg ellipsoid}.}}
\label{fig:4pics}
\end{figure}

Our next result identifies infinite staircases for the 
ellipsoid embedding functions of twelve convex toric domains, 
including the already known ball, polydisk $P(1,1)$, and $E(1,\frac32)$. 
Our proof of Theorem~\ref{thm:main} provides a uniform approach to prove the existence
of all twelve in one fell swoop.
The graphs of these 
functions are related to certain recurrence sequences, which are given in Table \ref{table:recurrence}.  
We remark that the scale invariant quantity $\frac{\text{per}^2}{\text{vol}}=K+2$, together with the order of recurrence $2J$, determine the recurrence relation.\footnote{Each order $2J$ recurrence sequence actually consists of $J$ intertwined order 2 recurrence sequences, all with the same relation but different seeds. The sequences corresponding to $(4;2,2)$ and $(3;1)$ have  one of these subsequences in common, the one with $K=6$ and seeds $1,1$.}

\newpage

\renewcommand{\arraystretch}{1.5}\begin{table}[htbp]
\centering 
\tiny
\begin{tabular}{| c || c | c | c  | c | c |}
\hline

Neg.\ wt.\                      & Recurrence relation                       & Seeds  $g(0)$     &  $K=$ & $J$   & $a_0$     \\ 
expansion   &  $g(n+2J)=Kg(n+J)-g(n)$ & to $g(2J-1)$ & $\frac{\per^2}{\vol}-2$  &  & \\ \hline \hline
$(3)$                         & $g(n+4)=7g(n+2)-g(n)$                     & $2, 1, 1, 2$               & 7 & 2   & $\frac{7+3\sqrt{5}}{2}  $       \\  \hline
$(4;2,2)$                        & $g(n+4)=6g(n+2)-g(n)$                     & $1, 1, 1, 3$     & 6 & 2   &$3+2\sqrt{2}$ \\   \hline
$(3;1,1,1)$     &  $g(n+4)=4g(n+2)-g(n)$    &   $1, 1, 1, 2$   & 4 & 2 &$2+\sqrt{3}$  \\   \hline
$(3;1,1,1,1)$    & $g(n+4)=3g(n+2)-g(n)$   &  $1,2, 1, 3$  & 3 & 2 &  $\frac{3+\sqrt{5}}{2}  $ \\ \hline
$(3;1)$                      & $g(n+6)=6g(n+3)-g(n)$                     & $1,1, 1, 1, 2, 4$   & 6 & 3   & $3+2\sqrt{2}$    \\  \hline
$(3;1,1)$    & $g(n+6)=5g(n+3)-g(n)$ & $1, 1, 1, 1, 2, 3$    & 5 & 3 &  $\frac{5+\sqrt{21}}{2}  $  \\  \hline
\end{tabular}
\caption{The key recurrence relations.}
\label{table:recurrence}
\end{table}

\begin{theorem}\label{thm:main}
Let $X$ be a convex toric domain with negative weight expansion $(b;b_1,\ldots,b_n)$ 
equal to 
$$
(3)\ , \ (3;1)\ , \ (3;1,1)\ , \ (3;1,1,1)\ , \ (3;1,1,1,1)\ , \mbox{or } (4;2,2).
$$
Then the ellipsoid embedding function $c_X(a)$ has an infinite staircase
which alternates between horizontal lines and lines through the origin 
 connecting inner and outer corners 
$$(x^\inner_0,y^\inner_0),(x^\outter_1,y^\outter_1),(x^\inner_1,y^\inner_1),(x^\outter_2,y^\outter_2),\ldots
$$
with coordinates
$(x^\inner_n,y^\inner_n)$ given by
$$\left( \frac{g(n+J)\left( g(n+1)+g(n+1+J) \right)}{\left(g(n)+g(n+J)\right) g(n+1)},  \frac{g(n+J)}{g(n)+g(n+J)}  \right),$$
and 
$(x^\outter_n,y^\outter_n)$ given by
$$ \left( \frac{g(n+J)}{g(n)},\frac{g(n+J)}{g(n)+g(n+J)} \right).$$ 
\end{theorem} 

\begin{remark}
The recurrence relations that appear in Table \ref{table:recurrence} do not immediately appear to be the 
ones previously associated to infinite staircases.  But a quick computation shows that for $(3)$, this
does recover the odd-index Fibonacci numbers McDuff and Schlenk found in \cite{mcduffschlenk}; for $(4;2,2)$ it recovers 
Pell and half-companion
Pell numbers as found by Frenkel and M\"uller \cite{frenkelmuller}; and for $(3;1,1,1)$ the sequences of Cristofaro-Gardiner 
and Kleinman \cite{cgk}.  Writing them in this 
uniform way simplifies the statement of Theorem~\ref{thm:main}.
\end{remark}

\begin{remark}
Combining Theorem~\ref{prop:mainproposition}, Lemma~\ref{rem:polygon}, and Theorem~\ref{thm:main}, 
we conclude that the ellipsoid embedding function 
$c_X(a)$ has an infinite staircase for 
symplectic forms on the compact symplectic manifolds
$\C P^1_2 \times \C P^1_2$ and $\C P^2_3\# k\overline{\C P}^2_1$ for $k=0,1,2,3,4$.  
The smooth polygons in Figure~\ref{fig:twelve} are Delzant polygons: they 
are the moment polygons of compact toric symplectic manifolds, namely
with underlying smooth manifold 
$\C P^1 \times \C P^1$ and the $k$-fold blowup $\C P^2\# k\overline{\C P}^2$ for $k=0,1,2,3$.
The only negative weight expansion from the list that does not have a smooth Delzant polygon representative
is $(3;1,1,1,1)$.  This manifold is well known not to admit a Hamiltonian circle or
$2$-torus action \cite{hausmann knutson}.  
We may identify this manifold as {equilateral pentagon space} and as such,
it is well known to admit a completely integrable system from {bending flows}
whose image is shown in the bottom right picture in Figure~\ref{fig:twelve}.  Lemma~\ref{rem:polygon}
allows us to make the same conclusions about embeddings into $\C P^2_3\# 4\overline{\C P}^2_1$.
\end{remark}

\begin{remark}\label{rem: all cases}
For each convex toric domain, the accumulation point of the infinite staircase is on the volume  curve. However, 
two fairly distinct behaviors can be observed, related to the order of the recurrence relation in Table~\ref{table:recurrence}. 
In the $J=2$ cases, the inner corners of the infinite staircase are on the volume  curve, whereas in the $J=3$ cases, 
they approach the volume  curve but never touch it. 
Examples are shown in Figure~\ref{fig:4pics}(a) and (b).
Wherever the staircase hits the
volume curve, it indicates that there is a full filling of the target by the ellipsoid.
The behavior when $J=3$ has not previously been observed for rational convex toric domains.

These two different behaviours can be seen explicitly in the Proof of Proposition \ref{prop:inner}, which following a beautiful idea of Casals and Vianna uses sequences of almost toric fibrations to construct symplectic embeddings corresponding to the inner corners of the staircase.
In the $J=2$ case, the base diagrams of the almost toric fibrations are triangles, which give full filling ellipsoids. In the $J=3$
case, the base diagrams are quadrilaterals and the embeddings are determined by the biggest triangle contained in each 
quadrilateral, and therefore do not constitute a full filling. See also \cite{CV} and the note at the end of the introduction of this manuscript.
\end{remark}

We complete the introduction with a conjecture that the list in Theorem~\ref{thm:main} is in a suitable sense exhaustive.

Recall that a convex lattice polygon is {\bf reflexive} if it has exactly one interior lattice point; this is equivalent to requiring that its dual polygon is also a lattice polygon.  
Pick's Theorem allows us to conclude that any lattice domain with one of the negative weight expansions listed in Theorem~\ref{thm:main} must have exactly one interior lattice point and thus is a reflexive lattice polygon. 
These are classified up to $AGL_2(\Z)$, there are sixteen such polygons: the twelve shown in Figure~\ref{fig:twelve} and the four in Figure~\ref{fig:the others}. These last four do not have infinite staircases as part of their
ellipsoid embedding function, see Remark \ref{rmk:the other four}.

\begin{conjecture}
\label{conj:reflex}
If the ellipsoid embedding function of a rational convex toric domain has an infinite staircase, then 
its moment polygon is a scaling of a reflexive polygon.  
\end{conjecture}

In particular, if Conjecture~\ref{conj:reflex} holds, we will see that the only rational convex toric domains whose 
ellipsoid embedding function has an infinite staircase are indeed the ones from Theorem~\ref{thm:main}
or any scaling of those, by ruling out the remaining four reflexive polygons.  
We will use Theorem~\ref{satisfies quadratic equation} to give some evidence supporting Conjecture~\ref{conj:reflex} in this paper.  As further evidence, Cristofaro-Gardiner's paper \cite{cg ellipsoid} applies Theorem~\ref{satisfies quadratic equation} to prove Conjecture~\ref{conj:reflex} in the special case of ellipsoids.  After the first version of this paper appeared, Conjecture~\ref{conj:reflex} was also verified in the special case of one-point blowups of $\mathbb{C}P^2$ in \cite{inprep}.  

In light of the recent work \cite{project9,usher} about infinite staircases for irrational targets, 
it is crucial in Conjecture~\ref{conj:reflex} that the toric domain be rational.   That said, it would be compelling to try to fit the  
examples from  \cite{project9,usher} into a more general framework in the spirit of Conjecture~\ref{conj:reflex}. For some very interesting work related to this in the case of one-point blowups, we refer the reader to \cite{project9,mm,mmw}. 
It also would be interesting to explore the almost toric fibration methods in this context; for example, 
perhaps the embeddings required in the irrational examples can also be constructed via polytope mutation; see also \cite{mm, mmw, mtalk}.

\subsection{Other closed targets}
\label{sec:other}

We have chosen to focus our attentions here on the toric case, but the methods should work equally well for some other natural targets.  For instance, recall that a closed symplectic $4$-manifold is {\bf rational} if it is a product of spheres or a blowup of $\mathbb{C}P^2$. The case of the product of spheres is already toric.  As for a blowup of $\mathbb{C}P^2$, Theorem~\ref{satisfies quadratic equation} still holds, by essentially the same argument, where in the definition of $\text{per}$ and $\text{vol}$, we use the size of the blowups instead of the negative weight expansion.

More generally, recall that a symplectic manifold is called {\bf uniruled} if there is a nonzero genus zero Gromov-Witten invariant with a point constraint.  For example, any Hamiltonian $S^1$-manifold is uniruled \cite{m4} as is any rationally connected manifold.  By a fundamental result of McDuff \cite{m1}, the only closed uniruled symplectic 4-manifolds are rational or {\bf ruled}, which means that the manifold is a blowup of an $S^2$ bundle over a Riemann surface.  We discussed the rational case above; as for a manifold which is ruled but not rational (which is called {\bf irrationally ruled}), one in fact should be able to show that infinite staircases do not ever exist.  Indeed, the ball-packing problem is known to have a straightforward answer for these targets, first by the work of Biran \cite[Cor. 5.C]{biran} in the case of products and then by Holm and Kessler in the general case, see \cite[Lem. 4.8]{hk}; and, by a straightforward modification of the argument in \cite{mcd}, it should be possible to show that an ellipsoid embeds into an irrationally ruled manifold if and only if a corresponding ball-packing problem, analogous to \eqref{eqn:embeddingequation} below, has a solution.  

We should also remark that for many targets, for example linear tori, the ellipsoid embedding problem should be completely unobstructed; see \cite{lms, ev}.  In fact, the embedding obstructions in our work here and in all known examples of infinite staircases are encoded in exceptional spheres, and these are most interesting in the irrational and ruled case by for example \cite[Cor. 3]{li}.

\subsection*{Organization of the paper} 
We begin in Section~\ref{preliminaries} by reviewing the basic properties of ellipsoid embedding functions, ECH capacities, 
convex lattice paths, obstructive classes, toric manifolds, and almost toric fibrations. Next, we explore the relationship
between convex toric domains and compact toric manifolds in Section~\ref{sec:passing to closed}, proving 
Theorem~\ref{prop:mainproposition}.  In Section~\ref{sec:pinpoint}, we turn to the proof of 
Theorem~\ref{satisfies quadratic equation}.  We are then able give our unified proof of the existence of the infinite
staircases (Theorem~\ref{thm:main}) in Section~\ref{section:existence}.  We conclude by describing evidence supporting our Conjecture, 
in Section~\ref{sec:conj},  that  the six examples described here are the only examples among rational convex toric domains.

The paper also includes three appendices: the first, Appendix~\ref{ap:convex lattice paths}, draws together some combinatorial 
data used to define families of convex lattice paths $\Lambda_n$ needed to find obstructions for the proof of Theorem~\ref{thm:main}.
The second, Appendix~\ref{ap:atfs}, describes seeds for the families of almost toric fibrations needed
to provide embeddings in the proof of Theorem~\ref{thm:main}.  Finally, in Appendix~\ref{appendix}, we recall the very
beginning of the project, including a surprise connection to the numbered stops on a Philadelphia subway line.  This appendix
also contains the Mathematica code we used to estimate ellipsoid embedding functions and search for infinite staircases.

\vskip 0.1in

\noindent {\bf Acknowledgements}.  
We are deeply grateful for the support of the Institute for Advanced Study,
without which we would not have been in the same place at the same time (with lovely lunches) to begin this project, nor
continue our conjectural work during an engaging Summer Collaborators visit.  
We were encouraged and helped along by conversations with:
Roger Casals, 
Helmut Hofer, 
Yael Karshon,
Alex Kontorovich, 
Nicole Magill, 
Emily Maw, 
Dusa McDuff, 
Peter Sarnak, 
Felix Schlenk, 
Ivan Smith,  
Margaret Symington, 
Renato Vianna,
and 
Morgan Weiler.
We would particularly like to thank Dusa McDuff for her help improving our exposition of Theorem \ref{satisfies quadratic equation}.  Finally, we would like to thank the anonymous referees for their helpful comments and corrections.

\vskip 0.1in

\noindent DCG was supported by NSF grant DMS 1711976 and the Minerva Research Foundation.

\noindent TSH was supported by NSF grant DMS 1711317 and the Simons Foundation.

\noindent AM was supported by FCT/Portugal through project 
PTDC / MAT-PUR/29447/2017.

\noindent ARP was supported by a Simons Foundation Collaboration Grant for Mathematicians.

\vskip 0.1in

\noindent {\bf Relation to \cite{CV}}.  
This article has been posted simultaneously with that of Roger Casals and Renato Vianna \cite{CV}.
Their Theorem~1.2 coincides with our Proposition~\ref{prop:inner} for the negative weight expansions
$(3)$, $(3;1,1,1)$, $(3;1,1,1,1)$, and $(4;2,2)$.  Both collaborations have benefitted from our
exchanges of ideas.  
Indeed, our initial proof  relied solely on ECH capacities, but required
 additional technical details and  guaranteed existence of an infinite staircase without completely
computing the embedding capacity function.

When Pires gave a talk on this topic at a 2017 KCL/UCL Geometry Seminar, Casals 
shared his beautiful idea: that mutation sequences of ATFs provided explicit symplectic embeddings for the Fibonacci staircase and should do the same whenever the target region $\Omega$ is a triangle, that is, corresponding to 
the negative weight expansions $(3)$, $(4;2,2)$ and $(3;1,1,1)$. 
Following this suggestion, we were then able to implement these ATFs explicitly and uniformly for all of our target regions, including the
non-triangular ones. 
This  
clarified and simplified our work and allowed us to pin down the embedding capacity function entirely, rather 
than just providing an existence proof for infinite staircases. 
Independently and without mutual knowledge, Casals and Vianna went on to explore the embeddings 
arising from mutation sequences of ATFs, also studying connections to tropical geometry and cluster algebras. They use tropical techniques to go from the base diagrams to the existence of embeddings. 
We tackle the same issues using two symplectic techniques.
First, we use Local Normal Form results  for toric actions on 
non-compact manifolds \cite[Theorem B.3]{KL:noncpt toric} to understand the relationship
between the base diagram and coordinate neighborhoods on the corresponding manifold.  This allows
us to establish Proposition~\ref{prop:triangleintothing}, which is a key ingredient
in the proof of Theorem~\ref{thm:main}.  Second, we
rely heavily on symplectic inflation type machinery pioneered by Lalonde-McDuff, 
applied here via \cite{dan} to state
Proposition~\ref{prop:embedding} and via \cite{mcduffschlenk} to state Proposition~\ref{either}. 
These results are key ingredients in the proofs of Theorem~\ref{prop:mainproposition},
Proposition~\ref{prop:properties}(4) and (5),
and Corollary~\ref{prop:mainproposition2}.

\section{Preliminaries and tools}\label{preliminaries}

\subsection{Properties of the ellipsoid embedding function}

\begin{proposition}
\label{prop:properties}
Let $X$ be a convex toric domain with finite negative weight expansion.  The ellipsoid embedding function $c_X(a)$
\begin{enumerate}
\item is non-decreasing;
\item has the following scaling property: $c_X(t\cdot a)\leq t\cdot c_X(a)$ for $t\geq 1$;
\item is continuous; 
\item is equal to the volume curve for sufficiently large values of $a$; \label{it:stab} 
\item is piecewise linear, when not on the volume curve, or at the limit of singular points. \label{it:piecewise} 
\end{enumerate}
\end{proposition}

\begin{proof}

We prove only the first three points here, delaying the proof of the fourth to Section~\ref{sec:passing to closed}, and the fifth to Section~\ref{sec:pinpoint}, because the methods used to prove it are similar to the methods used to prove the results there.  The first three properties actually hold for general symplectic
$4$-manifolds $X$.

\begin{enumerate}
\item Let $a_1<a_2$. For all $\lambda$ such that $E(1,a_2)\embeds \lambda X$ we have $E(1,a_1)\embeds E(1,a_2)\embeds \lambda X$, so $c_X(a_1)\leq\lambda$.  We conclude that  $c_X(a_1)\leq c_X(a_2)$.

\item Let $t\geq 1$. For all $\lambda$ such that $E(1,a)\embeds\lambda X$ we have a sequence of embeddings
$E(1,t a)\embeds E(t,ta)\embeds t\lambda X$, which implies $c_X(t a)\leq t\lambda$. Therefore $c_X(t a)\leq t c_X(a)$.

\item Let $(a_i)_{i\in\mathbb{N}}$ be an increasing sequence converging to $a$, and define $t_i:=\frac{a}{a_i}>1$.  Using properties (2) and (1) we have
$$c_X(a)=c_X(t_i a_i)\leq t_i c_X(a_i)\leq t_i c_X(a).$$
Dividing through by $t_i$ and letting $i\to\infty$, we conclude that $\lim_{i\to\infty}c_X(a_i)=c_X(a)$.

Now let $(a_i)_{i\in\mathbb{N}}$ be a decreasing sequence converging to $a$, and define $t_i:=\frac{a_i}{a}>1$. Using properties (1) and (2) we have
$$c_X(a)\leq c_X(a_i)=c_X(t_i a)\leq t_i c_X(a),$$
in particular

\[  c_X(a) \le c_X(a_i) \le t_i c_X(a).    \]

Then, as above, letting $i\to\infty$ and using the fact that $\lim_{i \to \infty} t_i =1 $ implies that $\lim_{i\to\infty}c_X(a_i)=c_X(a)$. Therefore $c_X$ is continuous at $a$.
\end{enumerate}
\end{proof}

\subsection{ECH capacities}
\label{sec:ech}

Let $c_\text{ECH}(X)=(c_0(X),c_1(X),c_2(X),\ldots)$  represent the non-decr\-easing sequence of ECH capacities of the toric domain $X$,
as defined in \cite{qech}. The sequence inequality 
$$c_{ECH}(X) \le c_{ECH}(Y)$$
means that $c_k(X)\leq c_k(Y)$ for all $k\in\mathbb{N}_0$.

The sequence of ECH capacities for an ellipsoid $E(a,b)$ is the sequence (indexed starting at $0$) $N(a,b)$, where for $k\geq 0$, the term $N(a,b)_k$ is the $(k+1)^\text{st}$ smallest entry in the array 
$$
(am+bn)_{m,n\in\mathbb{N}_0},
$$
counted with repetitions \cite{m2}.  For example,  
$$N(1,1) = (0, 1, 1, 2, 2, 2, \ldots ).$$

\begin{proposition}\label{facts about sequences}
Fix positive integers $a$ and $b$.  There are at most $\frac{(a+1)(b+1)}{2}-1$ terms in the sequence $N(a,b)$ that are strictly less than $ab$. 
\end{proposition}

\begin{proof}
The number of terms in $N(a,b)$ with value strictly less than $ab$ is the same as the number $S$ of integer lattice points in the triangle in the first quadrant bounded by the axes and the line $ax + by < ab.$  We will estimate this using Pick's theorem.  More precisely, let $i$ denote the number of interior lattice points in the closure of this triangle, so that 
\begin{equation}
\label{eqn:defnS}
S = i + a + b - 1.
\end{equation}  
This is a lattice triangle, so we are justified in applying Pick'   s theorem, which gives that
\[ i = \frac{ab}{2} + 1 - \frac{q}{2},\]
where $q$ is the number of boundary lattice points.  Since $q \ge a + b + 1$, we therefore obtain in view of \eqref{eqn:defnS} that
\[ S \le \frac{ab}{2} + 1 + (a + b  - 1) - \frac{a+b+1}{2} = \frac{ (a+1)(b+1)}{2} - 1,\]
as desired.
\end{proof}

\noindent We now turn to some algebraic operations on ECH capacities.

\begin{definition}\label{def:seqsum}
Let $(S_k)_{k\geq 0}$ and $(T_k)_{k\geq0}$  be the sequences of ECH capacities of two convex toric domains $X$ and $Y$. We define the \textbf{sequence sum} and \textbf{sequence subtraction} as:
\begin{align*}
(S\#T)_k&= \max_{m+n=k}\left( S_m+T_n \right) \\
(S-T)_k&=\inf_{m\geq 0} \left(S_{k+m}-T_m\right).
\end{align*}
\end{definition}

\begin{remark}
In the definition of sequence subtraction above we require that $T\leq S$, i.e. that term-by-term this inequality holds.  If additionally 
$$\lim_{k\to\infty}S_k-T_k=\infty,$$
 then $\inf$ can be replaced by $\min$. This will happen in all instances in this paper, because $\text{volume}(X)>\text{volume}(Y)$. See \cite[Remark A.2]{dan} and \cite[Theorem 1.1]{asymptotics} for more details.
\end{remark}

\begin{remark}
The operations $\#$ and $-$ are not inverses.  Rather, as first observed by Hutchings, all one can conclude in general is that $(S-T) \# T \le S \le (S \# T) - T.$ 
\end{remark}

By \cite[Theorem A.1]{dan}, the sequence of ECH capacities of the convex toric domain $X$ with
negative weight expansion $(b;b_1,\ldots,b_n)$ is
obtained by the sequence subtraction
\begin{equation}\label{eq:ECH of toric convex domain}
\begin{array}{rcl}
c_\text{ECH}(X) & =&  c_\text{ECH} \left(B(b)\right) - c_\text{ECH} \left(\bigsqcup_{i=1}^nB(b_i)\right) \\
& =& c_\text{ECH} \left(B(b)\right) - \#_{i}c_\text{ECH}(B(b_i)).
\end{array}
\end{equation}
Since $E(1,a)$ is a concave toric domain in the sense of \cite[\S1.1]{dan} and $X$ is a convex
toric domain, the main result \cite[Theorem 1.2]{dan} implies the following. 

\begin{proposition}\label{prop:embedding}
A symplectic embedding $E(1,a) \stackrel{s}{\hookrightarrow} \lambda X$ exists if and only if
$$c_{ECH}(E(1,a)) \le c_{ECH}( \lambda X).$$
\end{proposition}  

\begin{remark}\label{rmk:dependsonly}
The existence of a symplectic embedding is equivalent to an inequality 
of ECH capacities, which are determined by the negative weight expansion.
Thus, the function $c_X(a)$ depends only on the negative weight expansion $(b;b_1,\ldots,b_n)$, 
not on any particular shape of a region in $\mathbb{R}_{\geq0}^2$ with that negative weight expansion.
\end{remark}

Combining this with the definition \eqref{definition of c} of the ellipsoid embedding function $c_X(a)$, we have
\begin{equation}\label{def:csup}
c_X(a)=\sup_k\frac{c_k(E(1,a))}{c_k(X)}. 
\end{equation}

An equivalent way to compute ECH capacities for convex toric domains uses the combinatorics of convex lattice paths.
The definitions below are based on \cite[Definitions A.6, A.7, A.8]{dan} and can be found there in more detail.

\begin{definition}
A \textbf{convex lattice path} is a piecewise linear path  $\Lambda:[0,c]\to\mathbb{R}^2$ such that all its vertices, including the first $(0,y(\Lambda))$ and last $(x(\Lambda),0)$, are lattice points and the region enclosed by $\Lambda$ and the axes is convex; we assume that $y(\Lambda)$ and $x(\Lambda)$ are both nonnegative. An \textbf{edge} of $\Lambda$ is a vector $\nu$ from one vertex of $\Lambda$ to the next.
The \textbf{lattice point counting function} $\mathcal{L}(\Lambda)$ counts the number of lattice points in the region bounded by a convex lattice path $\Lambda$ and the axes, including those on the boundary.

Let $\Omega\subset\mathbb{R}_{\geq0}^2$ be a convex region in the first quadrant. 
We define the $\Omega$\textbf{-length} of a convex lattice path $\Lambda$ to be
\begin{equation}\label{eq:Omega length}
\ell_\Omega(\Lambda)=\sum_{\nu \in\text{Edges}(\Lambda)}\det \left[ \nu\,\, p_{\Omega,\nu}\right],
\end{equation}
where for each edge $\nu$ we pick an auxiliary point $p_{\Omega,\nu}$ on the boundary of $\Omega$ such that $\Omega$ lies entirely to the right of the line through $p_{\Omega,\nu}$ with direction $\nu$.  To give some geometric intuition for the definition of $p_{\Omega,\nu}$, we note that if $p_{\Omega,\nu}$ happens to be a smooth point on the boundary of $\Omega$, then $\nu$ is parallel to the tangent line to $\Omega$ at $p_{\Omega,\nu}$.
\end{definition}

\noindent This type of calculation is illustrated
in the following figure.
\begin{center}
  \begin{figure}[ht]
\begin{overpic}[
scale=0.4,unit=1mm]{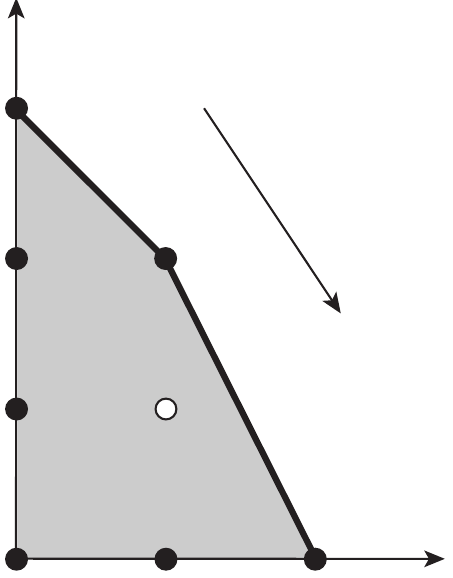}
   \put(10,25){$p_{\Omega,\nu}$}
   \put(5,11){\Large $\Omega$}
   \put(18,28){$\nu$}
\end{overpic}
\caption{This figure illustrates how to compute $\ell_\Omega(\Lambda)$.  The direction $\nu$ 
contributes one of the terms in the summation \eqref{eq:Omega length}.  In this case,
$\nu = (2,-3)$ and $p_{\Omega,\nu}= (1,2)$, and so the contribution is 
$\det \left(\begin{smallmatrix}
2&1\\-3&2
\end{smallmatrix}\right) = 4+3 = 7$.}
\label{fig:omega-length}
\end{figure}
\end{center}

\begin{remark}
Geometrically,
one can check that $\ell_{\Omega}$ is precisely the length measured with respect to the dual norm, for any (not necessarily symmetric) norm such that $\Omega$ is its unit ball.  For more related to this point of view, see \cite[\S 1.4]{qech}    
\end{remark}

Convex lattice paths provide a combinatorial way of computing ECH capacities of a convex toric domain, which we will use to prove Proposition \ref{prop:outer}.

\begin{theorem}\cite[Corollary A.5]{dan}\label{thm:lattice path}
Let $X$ be the toric domain corresponding to the region $\Omega$. Then its $k^\text{th}$ ECH capacity $c_k(X)$
is given by:
$$\min\left\{ \ell_\Omega(\Lambda): \Lambda  \text{ is a convex lattice path with } \mathcal{L}(\Lambda)=k+1\right\}.$$
\end{theorem}

 Hutchings indicates \cite[Ex.\ 4.16(a)]{h3} that the minimum can be taken over those lattice paths $\Lambda$ that
have edges parallel to edges of the region $\Omega$.  This simplifies the search for obstructing paths $\Lambda$ and 
explains why the lattice paths in
Figure~\ref{fig:paths} look similar to some of the domains in Figure~\ref{fig:twelve}.

\subsection{Obstructive classes}\label{sec:obstructive}

To find classes that obstruct the ellipsoid embedding question for $E(1,a)$, we must introduce
the \textbf{weight expansion} $(a_1,\ldots,a_n)$ of the rational number $a\geq1$. 
The definition is recursive, given geometrically by subdividing a $1\times a$ rectangle into squares
as follows. We begin by marking off as many $1\times 1$ squares as possible, denoting the number of 
such as $\ell_0$.  Note that $\ell_0 = \lfloor a\rfloor$.  This leaves a $1\times (a-\ell_0)$ rectangle,
where $x_1=a-\ell_0<1$. We now proceed by marking off as many
$x_1 \times x_1$ squares as fit in this remaining rectangle. 
The number of such squares is denoted $\ell_1$.  We continue recursively, at each stage marking of $\ell_j$
squares of dimensions $x_j\times x_j$.  Because $a$ is rational, this procedure will end with $\ell_N$ squares of
size $x_N\times x_N$ for some $N$.  This will 
result in a tiling of
the $1\times a$ rectangle by squares with rational side lengths.  The weight expansion for $a$ is the list
of these side lengths $x_i$, with multiplicities $\ell_i$.  That is, 
$$
(a_1,\dots,a_n) = (\underbrace{1,\dots,1}_{\ell_0},\underbrace{x_1,\dots,x_1}_{\ell_1},\dots,\underbrace{x_N,\dots,x_N}_{\ell_N}).
$$
This procedure is illustrated
in the following figure, which shows that the weight expansion of $\frac83$ is $(1,1,\frac23,\frac13,\frac13)$.
\begin{center}
  \begin{figure}[ht]
\begin{overpic}[
scale=2,unit=1mm]{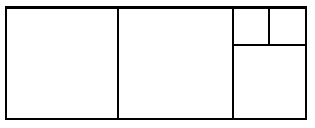}
   \put(-2,21){$1$}
   \put(21,5){$1$}
   \put(60,5){$1$}
   \put(90,6){$\frac23$}
   \put(105,15){$\frac23$}
   \put(105,33){$\frac13$}
   \put(84,36){$\frac13$}
   \put(97,36){$\frac13$}
\end{overpic}
\vskip 0.1in

\caption{This figure illustrates how to compute the weight expansion for $a=\frac83$.  We must 
divide a $1\times \frac83$ rectangle into squares.  The process begins with $\ell_0=2$  squares of size
$1\times 1$, then
$\ell_1=1$  square of size $\frac23 \times \frac23$, and finally $\ell_2=2$  squares of size $\frac13\times\frac13$.
}
\label{fig:weight expansion}
\end{figure}
\end{center}

\noindent 
The equivalent numerical definition of the weight expansion can be found 
in \cite[Definition 1.2.5]{mcduffschlenk}.  
When $a$ is irrational, we may still define the weight expansion in the same recursive way.  The geometric
procedure in this case will not end, and correspondingly the weight expansion has infinite length.

We recall here the essential properties of weight expansions that we use later in Section \ref{sec:pinpoint}.

\begin{lemma}{\cite[Lemma 1.2.6]{mcduffschlenk}}\label{lemma:mcduff schlenk}
Let $(a_1,\ldots,a_n)$ be the weight expansion of $a=\frac{p}{q}\geq1$, where $a$ is expressed in lowest terms. Then:
\begin{enumerate}
\item\label{item:small} $\ds{a_n=\frac{1}{q}\phantom{\sum_i^n}}$;
\item $\ds{\sum_{i=1}^n a_i^2=a}$; and
\item\label{item:weight expansion} $\ds{\sum_{i=1}^n a_i=a+1-\frac{1}{q}}$.
\end{enumerate}
\end{lemma}

Let $(b;b_1,\ldots,b_N)$ be the negative weight expansion of the convex toric domain $X$. 
The weight expansion of $a$ is related to the problem of embedding the ellipsoid $E(1,a)$ into $\lambda X$ in the following way, following \cite[Theorem 2.1]{dan}:
\begin{equation}
\label{eqn:embeddingequation}
\begin{array}{rl}
E(1,a)\embeds \lambda X  \iff & \bigsqcup_{j=1}^n B(a_j)\embeds\lambda X\\
\iff & \bigsqcup_{j=1}^n B(a_j)\sqcup \bigsqcup_{i=1}^N B(\lambda b_i)\embeds B(\lambda b).
\end{array}
\end{equation}

Equation \eqref{eqn:embeddingequation} highlights how the problem of embedding an ellipsoid $E(1,a)$ into a scaling of a convex toric domain $X$ is similar to the problem of embedding it into a scaling of a ball studied in \cite{mcduffschlenk}: both boil down to the problem of embedding a disjoint union of balls into another ball. Therefore it is no surprise that our proof uses similar tools to those in \cite{mcduffschlenk}, adapted to this more general case. In particular we use classes $(d;\mathbf{m})$, which for us are tuples of non-negative integers of the form 
\begin{equation}\label{eq:class}
(d;\mathbf{m})=(d;\widetilde{m}_1,\ldots,\widetilde{m}_N,m_1,\ldots,m_n)
\end{equation}
that satisfy the following Diophantine equations (cf.\ \cite[Proposition 1.2.12(i)]{mcduffschlenk}):
\begin{eqnarray}
\sum_{i=1}^N\widetilde m_i+\sum_{j=1}^n m_j=3d-1 \label{cond1} \\
\sum_{i=1}^N \widetilde m_i^2 + \sum_{j=1}^n m_j^2=d^2+1.\label{cond2}
\end{eqnarray}

Fix a convex toric domain $X$ and its negative weight expansion $(b;b_1,\ldots,b_N)$. Each class $(d;\textbf m)$ 
determines a function $\mu_{(d;\textbf m)}$ as follows. First, pad the tuple $(d;\mathbf{m})$ with zeros on the right in order to make it infinitely long. Then, for $a\in\mathbb{Q}$ with weight expansion $(a_1,\ldots,a_n)$ we define:
\begin{equation}\label{eq:defmu}
\mu_{(d;\textbf m)}(a):=\frac{\sum_{j=1}^n m_j a_j }{d\, b-\sum_{i=1}^N \widetilde m_i b_i}.
\end{equation}
Formula \eqref{eq:defmu} also makes sense for irrational values of $a$; as above, these have weight expansions of infinite length.

The following is analogous to  \cite[Corollary 1.2.3]{mcduffschlenk}:

\begin{proposition}\label{either}
Let $(a_1,\ldots,a_n)$ be the weight expansion of $a\in\mathbb{Q}$ and $(b;b_1,\ldots,b_N)$ be the negative weight expansion of $X$. 

If the ellipsoid $E(1,a)$ embeds symplectically into $X$, then either  
$$c_X(a)=\sqrt{\frac{a}{\vol}}$$
or there exists a class $(d;\mathbf{m})$ satisfying conditions \eqref{cond1} and \eqref{cond2} such that
\begin{equation} \label{cond3}
\mu_{(d;\mathbf m)}(a)>\sqrt{\frac{a}{\vol}}.
\end{equation}
In the latter case, $\displaystyle{c_X(a)=\sup_{(d;\mathbf{m})} \left\{\mu_{(d;\mathbf m)}(a)\right\}}$.
\end{proposition}

\begin{remark}
The supremum in Proposition~\ref{either} is actually a maximum: $$c_X(a)=\max_{(d;\mathbf{m})} \left\{\mu_{(d;\mathbf m)}(a)\right\}.
$$
This follows immediately from \eqref{dbound} below.  
\end{remark}

\noindent A class $(d;\mathbf m)$ satisfying \eqref{cond3} (in addition to \eqref{cond1} and \eqref{cond2}) is called 
an \textbf{obstructive class} and the corresponding function $\mu_{(d;\mathbf m)}$ is called an \textbf{obstruction}.

\begin{proof}
Hutchings' survey article \cite{h} gives a nice overview of these ideas.
By \cite[Theorem~1.2.2, Proposition~1.2.12]{mcduffschlenk}, 
an embedding as in the rightmost side of \eqref{eqn:embeddingequation} exists if and only if the volume constraint is satisfied (i.e. there is no obstruction from the volume) and 
\begin{equation}
\label{eqn:geomine}
\sum_{i=1}^N  \widetilde{m}_i  \lambda b_i + \sum_{j=1}^n m_j a_j < d \lambda b
\end{equation}
for all obstructive classes $(d; {\bf m})$.  
Though the following interpretation is not used in our proof, we remark for the reader seeking geometric intuition that the above inequality is motivated by consideration of symplectically embedded $(-1)$-spheres, which must in particular have positive symplectic area.  More precisely, the inequality \eqref{eqn:geomine} guarantees that the homology classes that could possibly support such spheres must have positive pairing with the symplectic form. 
It then turns out that this is the only constraint beyond the volume constraint for the existence of the desired embedding. For more details, we refer the reader to \cite{h}.

The inequality \eqref{eqn:geomine} is equivalent to the condition that
\begin{equation}\label{eq:geomine2} \lambda > \frac{ \sum_{j=1}^n m_j a_j}{db - \sum_{i=1}^N \widetilde{m}_i b_i}=\mu_{(d;\mathbf m)}(a).
\end{equation}
The last line of the proposition now follows because $c_X(a)$ is the infimum $\lambda$ value for which the embedding in the leftmost side of \eqref{eqn:embeddingequation} exists and thus equivalently \eqref{eq:geomine2} holds.
\end{proof}

The \textbf{length} $\ell(\mathbf{m})$ of the class 
$$(d;\mathbf{m})=(d;\widetilde{m}_1,\ldots,\widetilde{m}_N,m_1,\ldots,m_n)$$
is the number of nonzero $m_j$'s (not the $\widetilde m_i$'s).
For a rational number $a\in \mathbb{Q}$,  its \textbf{length} $\ell(a)$ is the number $n$ of 
entries in the weight expansion $(a_1,\ldots,a_n)$ of $a$.

\begin{lemma}\label{lemma:lengths}
Let $X$ be a convex toric domain and $(b;b_1\ldots,b_N)$ its negative weight expansion. Then,
\begin{enumerate}
\item If $\ell(a)< \ell(\mathbf m)$ then $\mu_{(d;\mathbf m)}(a)\leq \sqrt\frac{a}{\vol}$.
\item For all $a$ for which the right hand side is defined,
\begin{equation}
\label{dbound}
\mu_{(d;\mathbf m)}(a) \le  \sqrt{\frac{a}{vol}} \left( \frac{ \sqrt{b^2 - \sum b^2_i}}{\sqrt{ b^2 \frac{d^2}{d^2+1} - \sum b^2_i}}\right).
\end{equation}
\end{enumerate}
\end{lemma}

\begin{proof}
First assume that  $\ell(a)<\ell(\mathbf m)$, and therefore not all $m_j$'s appear in the sum $\sum m_j\, a_j$. 
Then we have
\begin{eqnarray*}
\mu_{(d;\mathbf m)}(a) & =& \frac{\sum m_j \,a_j }{d\, b-\sum \widetilde m_i \,b_i}\\
& \leq &\frac{\sqrt{\sum_{a_j\neq 0} m_j^2}\,\sqrt{\sum a_j^2}}{d\,b-\sum \widetilde m_i \,b_i} \quad \mbox{ (by Cauchy-Schwarz)}\\
& \leq & \frac{\sqrt{\sum m_j^2-1}\,\sqrt{\sum a_j^2}}{d\,b-\sum \widetilde m_i \,b_i}\quad \left(\mbox{ \begin{tabular}{l}
because at least one\\ $m_j\in\Z$ was excluded\end{tabular}}\right)\\
& =& \frac{\sqrt{d^2-\sum\widetilde{m}_i^2}\,\sqrt{a}}{d\,b-\sum \widetilde m_i \,b_i} \quad\quad \mbox{ (by \eqref{cond2})}
\end{eqnarray*}
The light-cone inequality (an analogue of the Cauchy-Schwarz inequality for the Lorentz product, \cite[Problem 4.5]{lightcone}) guarantees that
$$\sqrt{d^2-\sum\widetilde{m}_i^2}\,\sqrt{b^2-\sum b_i^2}\leq d\,b-\sum \widetilde{m_i}\,b_i,$$
so we obtain the desired inequality:
$$\mu_{(d;\mathbf m)}(a)\leq \frac{ \sqrt{a}}{\sqrt{b^2-\sum b_i^2}}=\sqrt\frac{a}{\vol}.$$

To prove \eqref{dbound} for general $a$ we repeat the argument above -- minus the line where we used the fact that at least one $m_j$ was excluded -- and conclude that
\[ \mu_{(d;\mathbf m)}(a)  \le \frac{\sqrt{1+d^2-\sum\widetilde{m}_i^2}\,\sqrt{a}}{d\,b-\sum \widetilde m_i \,b_i} .\]
The light-cone inequality from above guarantees that
\[ \sqrt{1+d^2-\sum\widetilde{m}_i^2}\,\sqrt{b^2\frac{d^2}{d^2+1}-\sum b_i^2}\leq d\,b-\sum \widetilde{m_i}\,b_i,\]
which gives us the desired bound
\[\mu_{(d;\mathbf m)}(a) \leq \frac{ \sqrt{a} }{\sqrt{b^2\frac{d^2}{d^2+1}-\sum b_i^2}}.\]
\end{proof}

\begin{proposition}\label{prop:there can be only one}
On each maximal interval $I$ where we have $\mu_{(d;\mathbf m)}(a)>\sqrt\frac{a}{\vol}$, the obstruction function $\mu_{(d;\mathbf m)}$ is continuous and piecewise linear.  Moreover, it has a unique non-differentiable point, which is at a value $\mathfrak{s}$ such that $\ell(\mathfrak{s})=\ell(\mathbf m)$.
\end{proposition}

\begin{center}
  \begin{figure}[ht]
\begin{overpic}[
scale=0.75,unit=1mm]{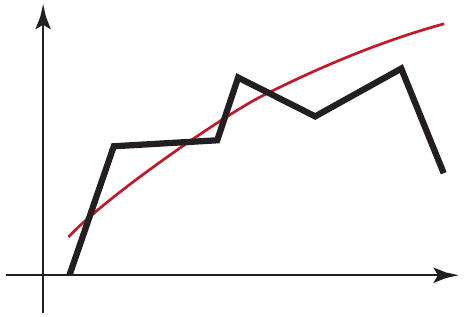}
   \put(57,35){\color{alizarin} $\displaystyle{\sqrt{\frac{a}{\vol}}}$}
   \put(57,18){$\mu_{(d;\mathbf m)}(a)$}
   \put(59,4){$a$}
   \put(25,31){\FiveStarOpenCircled}
   \put(8.5,21){\FiveStarOpenCircled}
\end{overpic}
\vskip 0.1in

\caption{The graph of an obstruction function $\mu_{(d;\mathbf m)}(a)$, together with the
volume curve {\color{alizarin} ${\sqrt{\frac{a}{\vol}}}$} in red.  
The marked \FiveStarOpenCircled s represent the unique singular points
guaranteed in Proposition~\ref{prop:there can be only one}.
}
\label{fig:obstruction}
\end{figure}
\end{center}

\begin{proof}
Let $I$ be a maximal interval where $\mu_{(d;\mathbf m)}(a)>\sqrt\frac{a}{\vol}$. Then by Lemma~\ref{lemma:lengths}, we have that $\ell(a)\geq \ell(\mathbf m)$ for all $a\in I$.
We note that discontinuities of the function
$$
\mu_{(d;\mathbf m)}(a) = \frac{ \sum_{j=1}^n m_j a_j}{db - \sum_{i=1}^N \widetilde{m}_i b_i}
$$
occur only when $\ell(a)$ drops strictly below $\ell(\mathbf m)$.  Thus,
$\mu_{(d;\mathbf m)}(a)$ is continuous on the interval $I$.
Now, assume towards a contradiction that $\ell(a) > \ell(\mathbf m)$ for all $a$ in $I$. Then in particular $1\notin I$ because $\ell(1)=1$. 

As in \cite[Lemma 2.1.3]{mcduffschlenk}, the $i^{th}$ weight in the weight expansion of $a$, considered as a function of $a$, is linear on any open interval that does not contain a point whose weight expansion has length less than or equal to $i$.  Thus, with $\ell(a) > \ell(\mathbf m)$ for all $a$ in $I$, the function $\mu_{(d;\mathbf m)}(a)$ would be linear on $I$.  But this is impossible: the volume curve is concave  and the interval $I$ is necessarily bounded above (and below by 1), as the graph of $c_X(a)$ is equal to the volume curve for sufficiently large $a$. Thus, there is some point $\mathfrak{s}$ with $\ell(\mathfrak{s}) = \ell(\mathbf m)$.  

The uniqueness follows from Lemma~\ref{lemma:lengths} together with the following basic fact about weight expansions, proved in  \cite[Proof of Lemma 2.1.3]{mcduffschlenk}: if $b > a$ are two rational numbers and $\ell(a)=\ell(b)$, then there must be some number $y \in (a,b)$ with $\ell(y) < \ell(a)=\ell(b)$.  

We conclude that $\mu_{(d;\mathbf m)}(a)$ is piecewise linear on $I$, with $\mathfrak{s}$ the unique singular point.  To see that  $\mu_{(d;\mathbf{m})}(a)$ is continuous at $\mathfrak{s}$, note first that it is equivalent to show that $\sum m_j a_j$ is continuous at $\mathfrak{s}$.  That $\sum m_j a_j$ is continuous at $\mathfrak{s}$ is proved in  
\cite[Lem. 2.3.3]{mcduffschlenk}.
\end{proof}

\begin{corollary} 
\label{it:one}
Any obstructive class is obstructive on finitely many intervals, on which it is linear.
\end{corollary}
\begin{proof}
For any obstructive class $(d; {\bf m})$, there are only finitely many values $a$ for which $\ell(a) = \ell(\mathbf m)$, and by Proposition~\ref{prop:there can be only one} any such interval must have such a point.
\end{proof}

\vskip 0.2in

\subsection{Toric manifolds and almost toric fibrations}\label{subsec:toric}
A {\bf  toric symplectic manifold} is a symplectic manifold $M$ equipped with
an effective\footnote{\phantom{.}An action is effective if no 
nontrivial element acts trivially.} Hamiltonian $T$ action satisfying
$
\dim(T)= \frac{1}{2}\dim(M).
$
Delzant established a one-to-one correspondence between compact 
toric symplectic manifolds (up to equivariant symplectomorphism) and Delzant polytopes (up to
$AGL_n(\Z)$ equivalence).

A convex polytope $\Delta$ in $\R^n$ may be defined as the convex hull of finitely many points, 
or alternatively as a (bounded) intersection of a finite number of half-spaces in $\R^n$.  
We say $\Delta$ is {\bf simple} if there are $n$ { edges} adjacent to each { vertex}, 
and it is {\bf rational} if the edges have rational slope relative to a choice of lattice
$\Z^n\subset \R^n$. For a vector with rational slope, the {\bf primitive vector} 
with that slope is the shortest positive multiple of the vector that is in the lattice 
$\Z^n\subseteq \R^n$. A  simple polytope is {\bf smooth} at a vertex if the $n$ 
primitive edge vectors emanating from the vertex span the lattice 
$\Z^n\subseteq\R^n$ over $\Z$. It is { smooth} if it is smooth at each vertex.  
A {\bf Delzant polytope} is a simple, rational, smooth, convex polytope.

To each compact toric symplectic  manifold, the polytope we associate to it is 
its moment polytope.  
There is a more complicated version of this classification theorem for toric symplectic
manifolds without boundary which are not necessarily compact.
In this case, {\em polytopes} are replaced by {\em orbit spaces}, which are possibly unbounded.  
Given such an orbit space, the manifold $M$ is
not unique but determined by a choice of cohomology class in 
$$
H^2(M/T; \Z^n\times \R).
$$
For further details, the reader should consult 
\cite[Chapter VII]{audin} and \cite[Theorem 1.3]{KL:noncpt toric}.

\begin{remark}\label{rem:toric unique}
Note that when $M/T$ is contractible, the cohomology group above is trivial and 
the corresponding $T$-space is unique.
For example, Euclidean space $\C^n$ equipped with the coordinate $T^n$ action
$$
(t_1,\dots,t_n)\cdot (z_1,\dots,z_n) = (t_1\cdot z_1,\dots,t_n\cdot z_n)
$$
is a toric symplectic manifold.  The moment map is 
$$
\mu:\mathbb{C}^n\to\mathbb{R}^n, \qquad (z_1,\dots,z_n)\mapsto(\pi |z_1|^2,\dots,\pi |z_n|^2),
$$ 
with image the positive orthant in $\R^n$.  Note that $\C^n/T^n$ is equal to this positive orthant, 
which is contractible and so by \cite[Theorem 1.3]{KL:noncpt toric}, this is the unique toric symplectic
manifold with this moment map image.

More generally, for any relatively open subset $\Omega\subset \R^n_{\geq 0}$, the toric domain 
$X_{\Omega}=\mu^{-1}(\Omega)$ inherits
a linear symplectic form and Hamiltonian torus action from $\C^n$.
Thus endowed,  $X_\Omega$ is a (non-compact) toric symplectic manifold with 
$X_\Omega/T = \Omega$.
When $\Omega$ is contractible, for example, the cohomology group $H^2(X_\Omega/T; \Z^n\times \R) = 0$ and 
in this case, $X_\Omega$ is the unique toric symplectic manifold with  moment map image $\Omega$.
\end{remark}

The moment map on a toric symplectic manifold $M$ is a completely integrable system
with elliptic singularities.  We now focus on four-dimensional manifolds.
An {\bf almost toric fibration} or {\bf ATF} is a completely integrable system on a four-manifold $M$ 
with elliptic and focus-focus singularities.
An {\bf almost toric manifold} is a symplectic manifold equipped with an almost toric fibration. 
These were introduced by Symington \cite{symington}, building on work of Zung \cite{zung}. 
Almost toric fibrations on compact four-manifolds without boundary
were classified by Leung and Symington in \cite{leung symington} in terms of the
 {\bf base diagram}, which includes
the image of the Hamiltonians 
with decorations to indicate the locations of  focus-focus singularities.
Evans gives a particularly nice exposition of these ideas \cite{evans}.

For a toric symplectic $M$, we can identify the singular points of the Hamiltonians in terms of the
moment map image.  In the four dimensional case, the preimage of each vertex in the moment polygon 
is a single point for which the moment map has an elliptic singularity of corank two.  
The preimage of a point on the interior of an edge is a circle, for each point of which 
the moment map has an elliptic singularity of corank one.  The preimage of a point on the interior of the polygon 
is a $2$-torus, of which each point is a regular point. 
Thus, in Figure~\ref{fig:mutation procedure}(a), there are three  corank two elliptic singularities,
three open intervals' worth of circles of corank one elliptic singularities, and a disc's worth of tori of regular points.

\begin{center}
  \begin{figure}[ht]
\begin{overpic}[
scale=0.3,unit=1mm]{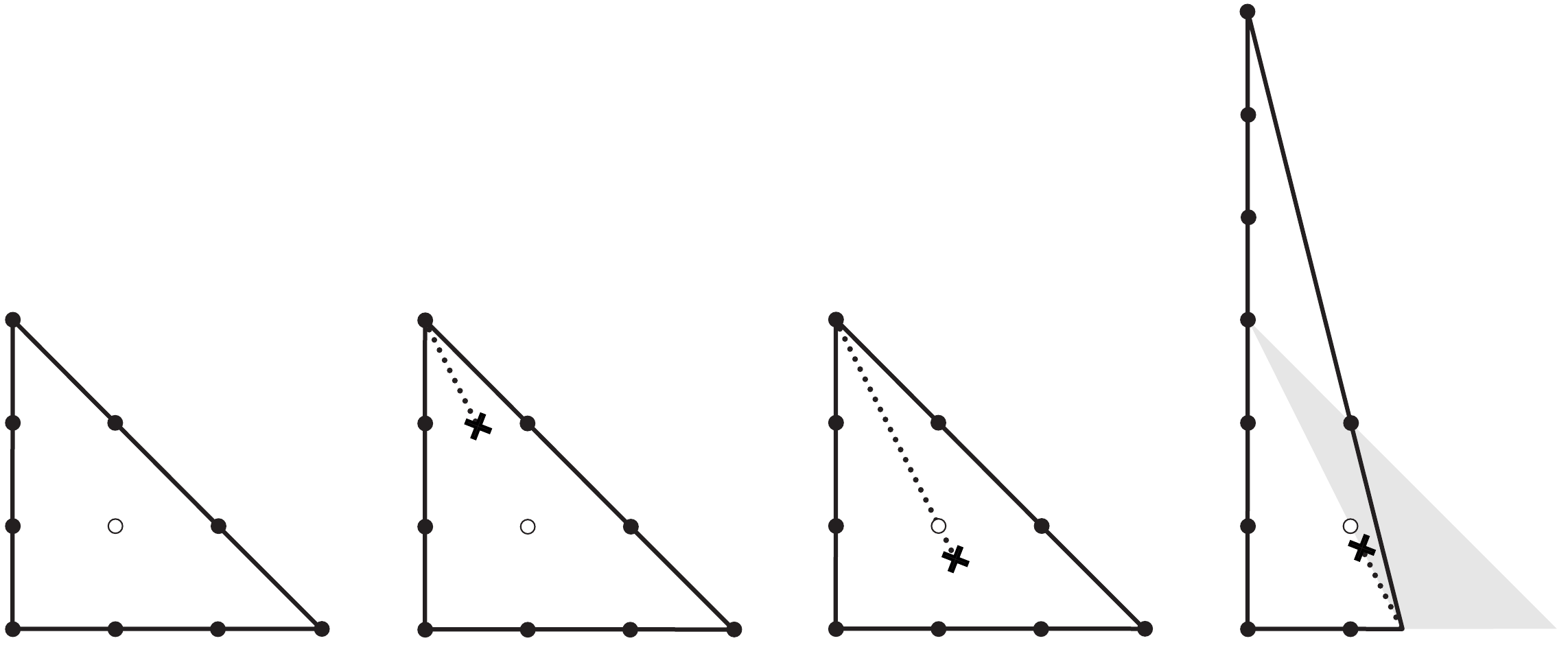}
   \put(9,-4){(a)}
   \put(40,-4){(b)}
   \put(70,-4){(c)}
   \put(98,-4){(d)}
   \put(60,25){\bf v}
   \put(105,0){\bf w}
\end{overpic}
\vskip 0.1in

\caption{Figure (a) is the Delzant polygon for the standard $T$-action on
$\C P^2_3$ (where the line has symplectic area 3).  From (a) to (b), we perform a nodal trade at the top vertex.  From (b) to (c),
we perform a nodal slide.  From (c) to (d), we perform a mutation on the base diagram at the anchor vertex
labeled {\bf v}; the mutation matrix is $\left(\begin{smallmatrix}
-1&-1\\ \phantom{-}4&\phantom{-}3
\end{smallmatrix}\right)$.
In (d), the base diagram is outlined in black and the new anchor vertex is labeled {\bf w}. The light gray part in (d) is there just to indicate the portion of the base diagram in (c) that changed as a result of the mutation.
Thus, each of these figures represents an almost toric
fibration on $\C P^2_3$.
Note that the last figure allows us to find an embedding from $E(\frac32,6)$ into $\C P^2_3$, which gives
$E(1,4)\embeds B(2)$.  This embedding 
is only as explicit as the diffeomorphisms described here pictorially (which is to say, not explicit!).
}
\label{fig:mutation procedure}
\end{figure}
\end{center}

There are three important operations on the base diagram of an almost toric manifold that fix the symplectomorphism 
type of the manifold (cf.\ \cite{evans, leung symington, symington, vianna}). 
The first is a {\bf nodal trade}.  Geometrically, this involves excising the neighborhood of a fixed
point and gluing in a local model of a focus-focus singularity.  This does not change the underlying manifold,
but it does change the Hamiltonian functions.  The effect on the base diagram is that we must 
insert a ray with a mark for the focus-focus singularity thereon.  In Figure~\ref{fig:mutation procedure}, such a ray has
appeared in (b).  The singularities of the Hamiltonian function are still recorded in the base diagram.
Above the marked point on the ray, there is a pinched torus.  The pinch point is a focus-focus singularity for the
new Hamiltonians; the other points on the pinched torus are regular.  Everything else is as before except
for the vertex that anchors the ray.  This has been transformed into 
a circle, for each point of which 
the new Hamiltonians have an elliptic singularity of corank one.

The second operation is a {\bf nodal slide}.  The local model for a focus-focus singularity has one degree of freedom.
A shift in that degree of freedom moves the focus-focus singularity further or closer to the preimage of the corner where
the ray is anchored.  In the base diagram, the marked point moves along the ray.
Such a slide is occurring in Figure~\ref{fig:mutation procedure} from (b) to (c).
The singularities remain as they were.

The third operation is a {\bf mutation} with respect to a nodal ray of the base diagram, at the corresponding anchor vertex.  
This changes the shape of the base diagram as follows.
The base diagram is sliced in two by the nodal ray.  One piece remains unchanged and the other
is acted on by an affine linear transformation in $ASL_2(\Z)$ that
\begin{itemize}
\item fixes the anchor vertex;
\item fixes the nodal ray; and
\item aligns the two edges emanating from the anchor vertex.
\end{itemize}
The operation creates a new (anchor) vertex and nodal ray (in the opposite direction from before) 
in the base diagram.  This result is shown in Figure~\ref{fig:mutation procedure} from (c) to (d).  As before, 
the preimage of the anchor vertex is
a circle, for each point of which 
the new Hamiltonians have an elliptic singularity of corank one.  The old anchor vertex is now
in the interior of an edge, and its preimage remains a circle of corank one elliptic singularities.

It is important to note that a mutation is only allowed when the nodal ray hits
\begin{itemize}
\item the interior of an edge; or
\item a vertex which is the anchor of a nodal ray in the opposite direction.
\end{itemize}
In the latter case, the marked points accumulate on the nodal ray.
See, for example, the sequence of mutations described in Figure~\ref{fig:3-1-1-1-pregame}
where many nodes have accumulated.

\begin{proposition}\label{prop:triangleintothing}
Suppose that a symplectic manifold 
$M$ is equipped with an almost toric fibration with base diagram $\Delta_M$ that consists of
a closed region in $\mathbb{R}_{\geq 0}^2$ that is bounded by the axes and a 
convex (piecewise-linear) curve from $(a, 0)$ to $(0, b)$, for $a,b\in\mathbb{R}^+$. Suppose in addition that there
is no nodal ray emanating from $(0,0)$.
Then there exists a symplectic embedding of the  ellipsoid
$(1-\varepsilon){E(a,b)}$ into $M$ for any $0<\varepsilon<1$.
\end{proposition}

\begin{proof}
The region $\Delta_M$ resembles Figure~\ref{fig:triangle fits}(a).  
We slide all nodes so that they are contained in small
neighborhoods of the vertices from which their rays emanate.
The neighborhoods should be sufficiently small so that
they are disjoint from
the triangle with vertices
$(0,0)$, $\big( (1-\varepsilon)\cdot a,0\big)$, and  $\big( 0,(1-\varepsilon)\cdot b\big)$.
The result now resembles
Figure~\ref{fig:triangle fits}(b).

\begin{center}
  \begin{figure}[ht]
\begin{overpic}[
scale=0.75,unit=1mm]{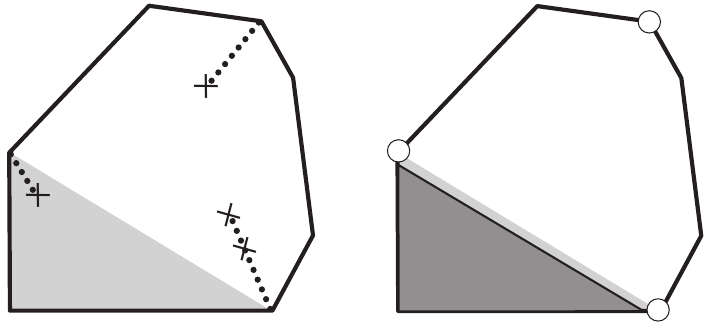}
   \put(13,-4){(a)}
   \put(68,-4){(b)}
   \put(36,0){$(a,0)$}
   \put(-8.5,23){$(0,b)$}
\end{overpic}
\vskip 0.1in

\caption{Figure (a) is a base diagram satisfying the hypotheses of  Proposition~\ref{prop:triangleintothing}.  
Figure (b) shows the new base diagram after nodal slides.  The nodal rays are 
contained in the small disks indicated at the corresponding vertices.
}
\label{fig:triangle fits}
\end{figure}
\end{center}

We now remove the small disks from the base diagram to produce a non-compact region $\Omega$.
Because there is no nodal ray emanating from $(0,0)$, the origin is still included in the region $\Omega$.
We also remove the corresponding neighborhoods from $M$ to produce a non-compact symplectic
manifold $M_\Omega\subset M$ with a pair of Poisson-commuting Hamiltonian functions that have only elliptic singularities.
Thus, $M_\Omega$ is actually a toric symplectic manifold.  Because $\Omega$ is contractible, 
following Remark~\ref{rem:toric unique}, $M_\Omega$ is the unique toric symplectic manifold
with this moment map image. The preimage of the origin is a fixed point.
Because $\Omega$ contains  the dark, closed triangle in
Figure~\ref{fig:triangle fits}(b)  with vertices
$$(0,0)\ ,\ \big( (1-\varepsilon)\cdot a,0\big)\ , \mbox{ and } \big( 0,(1-\varepsilon)\cdot b\big),
$$
the Local Normal Form theorem  \cite[Theorem B.3]{KL:noncpt toric} now guarantees that for the fixed point above $(0,0)$,
there is an equivariant neighborhood that is symplectomorphic to the closed
ellipsoid $(1-\varepsilon)\cdot \overline{E(a,b)}$.
This guarantees that for any $\varepsilon>0$, there is a symplectic embedding 
$(1-\varepsilon)\cdot {E(a,b)}\embeds M_\Omega\subset M$ (centered at the fixed point), as desired. 
\end{proof}

\section{Passing to closed symplectic manifolds}\label{sec:passing to closed}

We will see in this section how 
ellipsoid embeddings into compact target spaces, including $\mathbb{C}P^2$ blown up $0$ to $4$ times and 
$\mathbb{C}P^1\times\mathbb{C}P^1$, are equivalent to ellipsoid embeddings into appropriate convex toric
domains.
We begin with the proof of our first main theorem, which
relates the ellipsoid embedding function of a convex toric domain
$X_\Omega$ (a subset of $\C^2$ that is a manifold with boundary) and 
the ellipsoid embedding function of a compact toric symplectic
manifold $M$, when these two manifolds have the isomorphic
image $\Omega$ under their respective moment maps.
This will include the examples shown in Figure~\ref{fig:delzant},
which will produce the result for $0\leq n\leq 3$ blowups needed for Theorem~\ref{thm:main}.
Then we will prove Lemma~\ref{rem:polygon}, which will produce
the desired result for $4$ blowups of $\C P^2$

\begin{proof}[Proof of Theorem \ref{prop:mainproposition}]
($\Longleftarrow$)
First suppose we have an embedding $E(d,e) \embeds X_\Omega$. 
The ellipsoid $E(d,e)$ is an open ellipsoid, so the image of the symplectic embedding is contained in $\op{int}(X_\Omega)$.
Because the Delzant polygon for $M$ coincides with $\Omega$, we have an inclusion $\op{int}(X_\Omega)\embeds M$.
Indeed, this is a symplectic embedding, 
so we may simply compose $E(d,e) \embeds \op{int}(X_\Omega)\embeds M$ to get
 an embedding \eqref{eqn:desiredembedding}.
 
 \vskip 0.1in

\noindent ($\Longrightarrow$)
For the other direction, suppose that $M$ is a toric symplectic manifold whose moment map image is a Delzant polygon
(five examples are shown in Figure~\ref{fig:delzant}).
Assume there is an embedding $E(d,e) \embeds M$. 
To show that there is an embedding $E(d,e)\embeds X_\Omega$, 
\cite[Corollary~1.6]{dan} establishes that it is sufficient 
to produce embeddings of closed ellipsoids  $(1-\varepsilon) \overline{E(d,e)} \embeds X_{\Omega}$ for any  $0< \varepsilon <1$.  
Given such an $\varepsilon$, we first choose $d', e'$ so that $e'/d'$ is rational and
$$(1-\varepsilon) \overline{E(d,e)} \subset E(d',e')\subset E(d,e).
$$
In particular, because $E(d,e)\embeds M$, there is also a symplectic 
embedding of the closed ellipsoid $\overline{E(d',e')} \to M$. 

It is a general fact that a closed toric symplectic four-manifold $M$ is either a product of two symplectic two-spheres, 
or can be obtained from $\mathbb{C}P^2$ by a series of equivariant blowups (see \cite[Corollary 2.21]{kkp} for a careful exposition).  
For example, 
in Figure~\ref{fig:delzant}, the square in (c) corresponds to $S^2\times S^2$ with symplectic form that has
area $2$ on each $S^2$.  The triangle in Figure~\ref{fig:delzant}(a) corresponds to $\C P^2$ and
the polygons in Figure~\ref{fig:delzant} (b), (d), and (e) 
correspond to equivariant blowups of $\C P^2$.
We consider general blowups of $\C P^2$ and the product $S^2\times S^2$ in turn.

\vspace{3 mm}

\noindent {\bf Case 1: Blowups of $\mathbb{C}P^2$.}

\vspace{3 mm}

 Assume first that $M$ is obtained from $\mathbb{C}P^2$ by a series of 
equivariant symplectic blowups; as in the proof of \cite[Corollary 2.21]{kkp}, these equivariant blowups 
correspond to corner chops on the polygon, resulting finally in $\Omega$.  As has been our convention, we 
may assume that we choose the negative weight expansion for $\Omega$ with $b$ as small as possible and the $b_i$ 
as large as possible at each step, resulting in the negative weight expansion $(b;b_1,\dots,b_n)$. 
  
We have $\overline{E(d',e')}\subset M$.  Because $e'/d'$ is rational, this ellipsoid has a finite weight expansion
$(a_1,\dots, a_m)$.  We may use this weight expansion to blow up along that
closed ellipsoid (as in \cite[\S 2.1]{dan} or \cite{mcd}).  Together with the negative weight expansion
for $\Omega$, this sequence of blowups yields a symplectic form on
\begin{equation}
\label{eqn:blownupmanifold}
\mathbb{C}P^2 \# n\overline{\mathbb{C}P}^2  \# m\overline{\mathbb{C}P}^2.
\end{equation}
Specifically, we think of the first $n$ $\overline{\mathbb{C}P}^2$ factors as corresponding to the $n$ blowups 
required to produce $M$, and we think of the remaining $m$ factors as those required to blowup $\overline{E(d',e')}$;
The symplectic form on \eqref{eqn:blownupmanifold}
satisfies
$$
PD[\omega] = bL - \sum_{i=1}^n b_i E_i - \sum_{j=1}^m a_j E_j
$$
and
$$
PD(-c_1(TM)) = -3L + \sum_{i=1}^n E_i + \sum_{j=1}^m E_j.
$$
These two equations are analogues of equations (6) and (7) in \cite{h} 
(where we have normalized the line to have symplectic
area $b$).

By \cite[Proposition 6]{h}, having such a blowup symplectic form is equivalent to a symplectic embedding
\[\bigsqcup^{m}_{i=1} \overline{B( a_i) } \sqcup \bigsqcup_{i=1}^{n} \overline{ B(b_i) } \embeds B(b).\]
This immediately implies that the open balls embed
\[\bigsqcup^{m}_{i=1} {B( a_i) } \sqcup \bigsqcup_{i=1}^{n} { B(b_i) } \embeds B(b),\]
which allows us to use \cite[Theorem 2.1]{dan} to deduce that there is a symplectic embedding
\[ E(d',e') \embeds X_{\Omega},\]
and hence  the desired embedding $(1-\varepsilon) \overline{E(d,e)} \embeds X_{\Omega}$ exists.

\vspace{3 mm}

\noindent {\bf Case 2:  $M=S^2\times S^2 = \mathbb{C}P^1 \times \mathbb{C}P^1$.}

\vspace{3 mm}

If $M$ is a product of two symplectic two-spheres, we use the trick that after performing a single (arbitrarily 
small) blowup, we are back in Case 1.
Using the same notation as before,
we first find a small embedded $\overline{B(\delta)}$ disjoint from the image of $\overline{E(d',e')}$. 
Blow up along this ball, let $F$ denote the homology class of the exceptional fiber, and let $S_1$ and $S_2$ denote 
the homology classes of the spheres.  There is a diffeomorphism from the resulting manifold $\widehat{M}$ 
to $\mathbb{C}P^2 \# 2 \overline{\mathbb{C}P}^2$ mapping
\[F \mapsto L - E_1 - E_2\ , \quad S_1 \mapsto L - E_1\ , \quad S_2 \mapsto L - E_2.\]
This is described, for example, in \cite{frenkelmuller}. 
The canonical class gets mapped 
$$-c_1(TM) \mapsto -3L + E_1 + E_2$$ 
and there is an embedding 
$\overline{E(d',c')} \to \widehat{M}$, and so we can repeat the argument from Case $1$ above.  
More precisely, if the spheres have areas $b_1$ and $b_2$, respectively, then under this diffeomorphism 
the symplectic form on $\widehat{M}$ induces a symplectic form on $\mathbb{C}P^2 \# 2 \overline{\mathbb{C}P^2}$ 
that is obtained from $\mathbb{C}P^2$, normalized so that the line class has area $b_1 + b_2 - \delta$, by 
blowups of size $b_1 - \delta$ and $b_2 - \delta$.  The triple $(b_1 + b_2 - \delta; b_1 - \delta, b_2 - \delta)$ 
is the negative weight expansion for a rectangle of side lengths $b_1$ and $b_2$ with its top right 
corner removed, so the argument from Case $1$ gives an embedding of $E(d,e)$ into this toric domain, 
which in turn embeds into the toric domain associated to a rectangle of side lengths $b_1$ and $b_2$.
\end{proof}

The second half of the argument in the proof of Theorem~\ref{prop:mainproposition} also guarantees the following
Corollary.  As that argument has two
cases, so too does the Corollary.

\begin{corollary}
\label{prop:mainproposition2}
\phantom{let}

\begin{enumerate}
    \item 
Suppose $(b;b_1,b_2,\ldots,b_n)$ is a vector of non-negative integers that  both 
represents a blowup symplectic form on an $n$-fold blowup of projective space,
$M=\C P^2_b\ \#_i\  \overline{\C P}^2_{b_i}$, and also is
the negative weight expansion of 
a convex toric domain $X_{\Omega}$.  Then
$$
E(c,d) \embeds M\ \Longrightarrow  E(c,d) \embeds X_{\Omega}.
$$

\item Suppose $(a+b; a,b,c_1,c_2,\ldots)$ is a vector of non-negative integers that  both 
represents a symplectic form on  $\C P^1_a\times \C P^1_b$ blown
up symplectically
by sizes $(c_1,c_2,\ldots)$ and also is
the negative weight expansion of 
a convex toric domain $X_{\Omega}$.  Then
$$
E(c,d) \embeds M\ \Longrightarrow  E(c,d) \embeds X_{\Omega}.
$$
\end{enumerate}
\end{corollary}

We want to use Theorem~\ref{prop:mainproposition} to deduce that the embedding capacity function
for a closed manifold is equal to the embedding capacity function for a related toric domain.  
Inspecting Figure~\ref{fig:delzant}, the square in (c) is a Delzant polygon
with blowup vector $(4;2,2)$.
The triangle in Figure~\ref{fig:delzant}(a) corresponds to 
blowup vector $(3)$.
The polygons in Figure~\ref{fig:delzant} (b), (d), and (e) 
correspond to blowup vectors $(3;1)$, $(3;1,1)$, and $(3;1,1,1)$ respectively.
That leaves blowup vector $(3;1,1,1,1)$, which does not have a 
Delzant representative.
Nevertheless, we can recognize one of the toric domains with blowup vector $(3;1,1,1,1)$, 
shown in Figure~\ref{fig:eq-pol}, as 
the image of an integrable system on a smooth,
compact manifold $\mathscr{P}o\ell(1,1,1,1,1)$ that is not toric.
  Indeed, $\mathscr{P}o\ell(1,1,1,1,1)$ is known not to
admit any Hamiltonian circle action \cite[Theorem~3.2]{hausmann knutson}, even though
$\mathscr{P}o\ell(1-\delta,1+\delta,1,1-\delta,1+\delta)$ is a toric symplectic manifold for any $0<\delta<1$.
We will make use of this fact
in the following Lemma.


\begin{lemma}\label{rem:polygon}
Suppose that $(3;1,1,1,1)$ is the
negative weight expansion of a convex toric domain $X_{\Omega}$ and
let $M=\mathscr{P}o\ell(1,1,1,1,1)$ denote equilateral
pentagon space. Then 
$$
E(c,d) \embeds M\ \Longleftrightarrow  E(c,d) \embeds X_{\Omega}.
$$
\end{lemma}

\begin{proof}
For equilateral pentagon space $\mathscr{P}o\ell(1,1,1,1,1)$, the integrable system called the {bending flow}  \cite{kapovich-millson} has image as shown in the following figure.

  \begin{figure}[ht]
\centering
\includegraphics{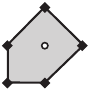}
\caption{The image of the bending flow integrable system on equilateral pentagon space $\mathscr{P}o\ell(1,1,1,1,1)
= \C P^2_3\# 4\overline{\C P}^2_1$.  
The bottom left vertex is at the origin.
Note that this is a region $\Omega$ that appears in Figure~\ref{fig:twelve} that has
blowup vector $(3;1,1,1,1)$.}
\label{fig:eq-pol}
\end{figure}

This integrable system extends a toric action on an open dense subset of $\mathscr{P}o\ell(1,1,1,1,1)$: we must simply remove two Lagrangian $S^2$s
that map to the points $(2,1)$ and $(1,2)$ in the Figure~\ref{fig:eq-pol}.  
These Lagrangian $S^2$s are the loci of points where the ``bending diagonals" vanish.
The dense subset has moment image the polytope in
Figure~\ref{fig:eq-pol} with the two vertices $(2,1)$ and $(1,2)$ removed.  The Local Normal Form theorem  \cite[Theorem B.3]{KL:noncpt toric} 
now guarantees that the relevant $\op{int}(X_\Omega)$ is in fact a subset of $\mathscr{P}o\ell(1,1,1,1,1)$.  This allows
us to conclude that if an ellipsoid $E(d,e)\embeds X_\Omega$, it must also embed in $\mathscr{P}o\ell(1,1,1,1,1)$.

On the other hand, Proposition~\ref{prop:mainproposition2}(1)
guarantees that if 
$$E(d,e)\embeds \mathscr{P}o\ell(1,1,1,1,1),$$ it also embeds into $X_\Omega$.
This completes the proof.
\end{proof}

We now also give the promised proof of the fourth item in Proposition~\ref{prop:properties}, which uses some of the 
same ideas in the of proof of Theorem~\ref{prop:mainproposition} .

\begin{proof}[Proof of Proposition~\ref{prop:properties}(4)]
Recall from \eqref{eqn:embeddingequation} that finding an embedding $E(1,a) \embeds X_{\Omega}$ 
is equivalent to finding a  ball-packing
\begin{equation}
\label{eqn:rearranged}
\bigsqcup_{i=1}^n B\left(\frac{a_i}{\lambda}\right)\sqcup \bigsqcup_{j=1}^N B(b_j)\embeds B(b),
\end{equation}
where $\lambda = \sqrt{a/vol}$, the $a_i$ are the weights of $E(1,a)$, and the vector $(b;b_1,\ldots,b_N)$ is the negative weight expansion for 
$X_{\Omega}$.  As in the proof of Theorem \ref{prop:mainproposition} and by the argument\footnote{\phantom{.}The result \cite[Corollary 1.6]{dan} is stated for a single domain, but as was already observed by Gutt-Usher \cite[\S 3]{gu} the proof works just as well for disjoint unions.} for \cite[Corollary 1.6]{dan}, in order to find an embedding  \eqref{eqn:rearranged}, it suffices to find, for any $0<\varepsilon <1$, an embedding
\begin{equation}
\label{eqn:rearranged2}
\bigsqcup_{i=1}^n \overline{B\left((1-\varepsilon) \frac{a_i}{\lambda}\right)}\sqcup \bigsqcup_{j=1}^N \overline{B((1-\varepsilon) b_j)}\embeds B(b).
\end{equation}

We will find this embedding by looking at the closed symplectic manifold $(M,\omega)$ which is the $N$-fold blowup
of $\mathbb{C}P^2_b$  with blowups of sizes $(1-\varepsilon)b_j$.  By the strong packing stability property 
\cite[Theorem 1]{bho}, there is some number $\delta$ associated to $M$ such that the only obstruction to embedding any number of (open) balls of parameter less than $\delta$ is given by the volume constraint.  Now choose $a$ sufficiently large, so that each $\frac{a_i}{\lambda}$ above is smaller than $\delta$; we can do this, because each $a_i$ is bounded above by $1$, while $\lambda=\sqrt{\frac{a}{\vol}}$ limits to $+\infty$ with $a$.  Then, strong packing stability applies to find an embedding of $\bigsqcup_{i=1}^n B\left(\frac{a_i}{\lambda}\right)$ into $M$.  
It therefore follows that we can find an embedding of the disjoint union of the closed balls $\overline{B\left((1-\varepsilon) \frac{a_i}{\lambda}\right)}$ into $M$ as well.   As in the proof of Theorem \ref{prop:mainproposition} above, we can then blow down to get an embedding of the desired
form \eqref{eqn:rearranged2}.
\end{proof}

\section{Pinpointing the location of the accumulation point}\label{sec:pinpoint}

The purpose of this section is to prove Theorem \ref{satisfies quadratic equation}.  We first give a roadmap of the proof.  The proof naturally breaks up into two parts. Here and below, let $a_0 \ge 1$ denote the solution to \eqref{eqn:theequation}. 

The first part of the proof, corresponding to the part of the proof up through Step 2, is concerned with establishing a certain very important estimate, which we now explain.  Recall that any obstructive class has a unique singular 
point $\mathfrak{s}$, guaranteed by Proposition~\ref{prop:there can be only one}. 
We call this unique point the {\bf breakpoint}.  The desired bound, which is inspired by 
\cite[Lem. 2.1.3]{mcduffschlenk}, relates $d$ and $|\mathfrak{s} - a_0|$ and
is stated precisely in \eqref{eqn:thebound}.
The rough idea is that a lower bound on $|\mathfrak{s} - a_0 |$ translates to an upper bound on $d$,
which then bounds the number of possible obstructive classes. 

As we will explain, it is almost an immediate consequence of \eqref{eqn:thebound} that if there are infinitely many distinct breakpoints of obstructive classes, then they must accumulate at $a_0$. 

To complete the proof, one must then relate singular points of the function $c_X$ to the breakpoints of obstructive classes;
this is the content of the second part of the proof.  The two are closely related, but the complication here is that a singular point of $c_X$ in principle need not arise from a breakpoint of an obstructive class, see Figure~\ref{fig:extra-singular-pts}.  Nevertheless, we are able to show that infinitely many singular points of $c_X$ correspond to infinitely many breakpoints, completing the argument.
We now give the details.

\begin{proof}[Proof of Theorem \ref{satisfies quadratic equation}]
To begin, it is convenient to introduce further notation. We let $a = p/q \geq1$ be a rational number with weight expansion $(a_1,\ldots,a_n)$ and $X$ be a convex toric domain with negative weight expansion $(b;b_1,\ldots,b_N)$. Let also $\lambda_a = \sqrt{\frac{a}{\vol}}$. We introduce the vector
$$\mathbf w=(\lambda_a b_1,\ldots,\lambda_a b_N,a_1,\ldots,a_n)$$
and use it to define the {\bf error vector} $\pmb{\epsilon}$ following \cite[(2.1.1)]{mcduffschlenk} by
\begin{equation}
\label{eqn:errorvector}
\mathbf m = \frac{d}{\lambda_a b} \mathbf w + \pmb{\epsilon}, 
\end{equation}
where $(d;\mathbf m)$ is a class as in 
equation \eqref{eq:class} satisfying \eqref{cond1} and \eqref{cond2}; we emphasize for the reader that $\mathbf{w}, \pmb{\epsilon}$ and $\lambda_a$ depend on $a$, though $\mathbf m$ does not. (Note that here and below we mean componentwise addition, in contrast to the sequence sum in Definition~\ref{def:seqsum}.  We will always denote the sequence sum by $\#$.)      It can be checked that $(d;\mathbf m)$ satisfies \eqref{cond3} and is thus called an obstructive class if and only if the inner product
\begin{equation}
\label{eqn:positivityofintersection}
\pmb{\epsilon} \cdot \mathbf w > 0.
\end{equation}

We now derive a key equality, see \eqref{eqn:coolequation} below.
We know that
\begin{equation}\label{eq:above}
 - \pmb{\epsilon} = \frac{d}{\lambda_a b} \mathbf w - \mathbf m.
 \end{equation}
Let $(d;\mathbf m)$ be an obstructive class and let $s_i$ denote the entries in $\mathbf w$.  Then combining 
equation~\eqref{eq:above} with \eqref{cond1} gives
\begin{align*}
- \sum_i {\epsilon}_i = \frac{d}{\lambda_a b} \left(\sum_i s_i\right) - (3d - 1) \\
 = 1 + \frac{d}{\lambda_a b} \left(\left(\sum_i s_i\right) - 3 \lambda_a b\right).
\end{align*}
Using Lemma~\ref{lemma:mcduff schlenk}\eqref{item:weight expansion}
and taking the absolute value of both sides, we can further rewrite the above as
\begin{equation}
\label{eqn:coolequation}
 \left| - \sum_i  {\epsilon}_i \right| = \left|1 + \frac{d}{\lambda_a b}\left(a + 1 + \left( \sum_i \lambda_a b_i \right) - 3 \lambda_a b - \frac{1}{q}\right)\right|.
\end{equation}

Before proceeding, let us say a few words about the 
significance of \eqref{eqn:coolequation}.  Recall from the ``roadmap" discussion before the proof that we seek a bound on $d$, in order to bound the number of obstructive classes.  The equation  \eqref{eqn:coolequation} gives such a bound given a bound on  $| \sum_i  {\epsilon}_i |$. Establishing a bound on $| \sum_i  {\epsilon}_i |$ is the content of the next three steps.

\vskip 0.1in

\noindent
{\bf Step 0: Ordering the class; capacity function at accumulation equals volume.}
We now assume here and below that the entries of $\mathbf m$ satisfy $\widetilde{m}_i \ge \widetilde{m}_j$ 
and $m_i \ge m_j$, whenever $i \le j$.  In other words, we will only analyze classes $\mathbf m$ 
for which this property holds; we call such an $\mathbf m$ {\bf ordered}.  The reason we can do this
is that if we have an arbitrary $\mathbf m$, and we permute its entries to make it ordered, then 
the left hand side of \eqref{cond3} for the permuted $\mathbf m$ will be at least as much as the 
value for the original $\mathbf m$.   Hence, in computing $\mu_{(d;\mathbf m)}(a) $, we can 
restrict to ordered $\mathbf m$.

It now follows that if $z_\infty$ is a limit of singular points $z_i$, then $c_X(z_{\infty})$ lies on the volume curve.  Otherwise, by continuity, there is some neighborhood $I$ of $z_{\infty}$ for which the distance between $c_X(a)$ and the volume curve is uniformly bounded away from zero
for all $a \in I$. 
However, this cannot occur: in this neighborhood, any obstructive class whose obstruction gives $c_X(a)$ for any $a \in I$ must have a uniform bound on $d$, using \eqref{dbound}, and then since our classes are ordered there are only a finite number of them, which is a contradiction in view of Corollary~\ref{it:one} since we have infinitely many singular points in $I$.

\vskip 0.1in

\noindent {\bf Step 1: A preliminary estimate.}  The purpose of this step is to prove a basic, but very important, 
estimate on any obstructive class, namely \eqref{eqn:criticalbound} below.

Let $\mathbf m$ be an ordered obstructive class, and let $\mathfrak{s}$ be a corresponding breakpoint; recall that this is the point from 
Proposition~\ref{prop:there can be only one} where $\ell(\mathfrak{s})=\ell( \mathbf m)$.  
Write $\mathfrak{s}=p/q$, where $p$ and $q$ are in lowest terms.  Assume that $\mathfrak{s} \ne 1$.  
We know from Lemma~\ref{lemma:mcduff schlenk}\eqref{item:small} 
that the smallest weight of $\mathfrak{s}$ must be $1/q$.  Moreover, we know that $E(1,\mathfrak{s})$ is not a ball.  
Hence, the smallest weight of $\mathfrak{s}$ must repeat at least twice.  We now claim that we must have 
\begin{equation}
\label{eqn:criticalbound}
\frac{d}{q\lambda_\mathfrak{s} b} > 1/4.
\end{equation}
(In fact, the constant $1/4$ here is not optimal, but suffices for our purposes.) To see why, first note that by condition~\eqref{cond2} and equation~\eqref{eqn:errorvector},
we have 
\begin{eqnarray*}
d^2+1 & = & \mathbf m\cdot \mathbf m \\
& = & \left( \frac{d}{\lambda_\mathfrak{s} b} \mathbf w + \pmb{\epsilon} \right) \cdot  \left( \frac{d}{\lambda_\mathfrak{s} b} \mathbf w + \pmb{\epsilon} \right) \\
& = & \frac{d^2}{\lambda_\mathfrak{s}^2b^2}\mathbf  w\cdot\mathbf  w + 2 \frac{d}{\lambda_\mathfrak{s} b} \mathbf w\cdot \pmb{\epsilon} + \pmb{\epsilon}\cdot \pmb{\epsilon}.
\end{eqnarray*}
We know that $\mathbf w\cdot \mathbf w = \mathfrak{s} + \lambda_\mathfrak{s}^2(b^2-\vol)$.  Noting that $\lambda_\mathfrak{s}^2 = \frac{\mathfrak{s}}{\vol}$, this simplifies to $\mathbf w\cdot\mathbf  w=\lambda_\mathfrak{s}^2b^2$,
and so 
\[
d^2+1 = d^2 + 2 \frac{d}{\lambda_\mathfrak{s} b} \mathbf w\cdot \pmb{\epsilon} + \pmb{\epsilon}\cdot \pmb{\epsilon}.
\]
Now recalling that \eqref{eqn:positivityofintersection}  says $\mathbf{w}\cdot \pmb{\epsilon}>0$, we conclude that
\begin{equation}
\label{eqn:randomepsilonbound}
\sum_i  {\epsilon}_i^2 < 1.
\end{equation}

Hence, in particular, each ${\epsilon}_i$ must be less than $1$.  
Remember now that we have
$$\begin{array}{rcll}
\mathbf m& = & \big(  \,  \begin{array}{|lcl|}\hline \widetilde{m}_1\ , & \dots\ , & \widetilde{m}_N\\ \hline \end{array} \ , &
 \begin{array}{|lcl|}\hline {m}_1\ , & \dots\ , & {m}_n\\ \hline \end{array}  \ \big) \\
\mathbf w& = & \big(  \,  \begin{array}{|lcl|}\hline \lambda_a b_1\ , & \dots\ , & \lambda_a b_N\\ \hline \end{array} \ , &
 \begin{array}{|lcl|}\hline {a}_1\ , & \dots\ , & {a}_n\\ \hline \end{array}  \ \big) \\
 \end{array}$$ 
where the entries in each box are in decreasing order,  the $m_i$ are positive integers, and the $a_i$ are the weight
expansion for $\mathfrak{s}$. In particular, we must have $a_{n-1}=a_n=\frac1q$.
Thus, the last two entries of the vector $\pmb{\epsilon} = \mathbf m-\frac{d}{\lambda_\mathfrak{s} b}\mathbf w$ are $m_{n-1}-\frac{d}{\lambda_\mathfrak{s} bq}$ and
$m_{n}-\frac{d}{\lambda_\mathfrak{s} bq}$.  In proving \eqref{eqn:criticalbound}, we can assume that $\frac{d}{\lambda_\mathfrak{s} bq} < 1,$ or else \eqref{eqn:criticalbound} already holds; thus, because ${\epsilon}_i<1$, we must have $m_{n-1}=m_n=1$.  
If contrary to the assumption \eqref{eqn:criticalbound} we had $\frac{d}{q\lambda_\mathfrak{s} b} \leq \frac14$,
then each of these last two terms would be at least $\frac34$, and so we would conclude that
\[ \sum_i {\epsilon}_i^2 \geq 9/16 + 9/16 > 1,\]
contradicting \eqref{eqn:randomepsilonbound}.     

\vskip 0.1in

\noindent {\bf Step 2.  Estimating $d$.}  We can now prove a strong estimate around $d$, namely the promised estimate \eqref{eqn:thebound}.  The rough idea behind this estimate is to combine the estimate \eqref{eqn:criticalbound} from Step 1  with the key \eqref{eqn:coolequation}.

The details are as follows.
  By the proof of 
  \eqref{eqn:coolequation}, we have that 
\[- \sum_i  {\epsilon}_i =  1+\frac{d}{\lambda_\mathfrak{s} b} \left(\mathfrak{s}+1 + \left(\sum_i \lambda_\mathfrak{s} b_i \right) - 3 \lambda_\mathfrak{s} b - \frac1q\right).\]
Let $L$ be the length of the weight expansion of $\mathfrak{s}$, plus a finite number 
$N$ of terms 
corresponding to the number of $b_i$.  
Applying Cauchy-Schwarz to $\pmb{\epsilon}$ and the vector $(1,\dots, 1)$ of length $L$, and using 
\eqref{eqn:randomepsilonbound}, we know that
\[ \left|\sum_i - {\epsilon}_i \right|< \sqrt{L}.\] 
The triangle inequality guarantees that 
$$
\left| -1-\sum_i{\epsilon}_i\right| \le 1 + \left| -\sum_i{\epsilon}_i\right|.
$$
We therefore get that
\[ \left| -1-\sum_i{\epsilon}_i\right| = \frac{d}{\lambda_\mathfrak{s} b} \left(\left| \mathfrak{s} + 1 + \left(\sum_i \lambda_\mathfrak{s} b_i \right) - 3 \lambda_\mathfrak{s} b - \frac1q \right|\right) \le 1+ \sqrt{L} .\]
We now want to bound $L$.
It is a basic fact about weight expansions, see \cite[Lemma 5.1.1]{mcduffschlenk}, that the length of the weight expansion for $\mathfrak{s}$ is bounded from above by $\lfloor \mathfrak{s} \rfloor + q$, where we recall that $\mathfrak{s}=\frac{p}{q}$ in lowest terms.
  To simplify the notation, define
\begin{equation}\label{eq:f}
\begin{array}{rcl}
f(a) & = & a+ 1 +\left( \sum_i \lambda_a b_i \right) - 3 \lambda_a b \\
& = & (a+1) - \left( \lambda_a\cdot \left(3b-\sum_i b_i\right)\right) \\
& = & a+1-\sqrt{a\cdot\frac{ \per^2}{\vol}}.
\end{array}
\end{equation}
We note that $f(a) = 0$ has the same solutions  as \eqref{eqn:theequation}, as can be seen by multiplying both sides of the equation $f(a) = 0$ by $\left(a+1+\sqrt{a\cdot\frac{ \per^2}{\vol}}\right)$; we will use this fact below.

In view of Proposition~\ref{prop:properties}(4), there is a constant $M$ such that $\mathfrak{s} \le M$ and we thus get
\[  \frac{d}{\lambda_\mathfrak{s} b} \left(\left| f(\mathfrak{s}) - \frac{1}{q}\right|\right) \le 1+\sqrt{N+M+q}.\]
Rearranging \eqref{eqn:criticalbound}, we have 
\[ q < \frac{4d}{\lambda_\mathfrak{s} b}.\]
Hence, we get
\begin{equation}
\label{eqn:thebound}
\frac{d}{\lambda_\mathfrak{s} b} \left(\left|f(\mathfrak{s}) - \frac{1}{q}\right|\right) \le 1+\sqrt{N+M + 4\frac{d}{\lambda_\mathfrak{s} b}}.
\end{equation} 
The importance of this bound will already be evident in the next step.

\vspace{2 mm}

\noindent{\bf Step 3.  Infinitely many distinct breakpoints must accumulate at $a_0$.}  Using \eqref{eqn:thebound} we now give the promised proof that infinitely many distinct breakpoints of obstructive classes must accumulate at $a_0$.  Assume that there are infinitely many distinct breakpoints $a_i = p_i/q_i$ corresponding to ordered obstructive classes $(d_i;{\bf m}_i)$.  If $a$ is sufficiently large, then $c_X(a)$ lies on the volume curve by Proposition~\ref{prop:properties}\eqref{it:stab}.  Hence, the $a_i$ must accumulate at some value $s_\infty$.  If $s_\infty \ne a_0$, then 
\[ |f(s_\infty)| > 0.\]  
Assuming this, then there are infinitely many $a_i$ with the property that $| f(a_i) - 1/q_i |$ is uniformly bounded away from $0$, since the $q_i$ must limit to infinity.  Hence, by \eqref{eqn:thebound} the corresponding values $d_i$ are uniformly bounded from above.  However, this is a contradiction, since there are only finitely many ordered obstructive classes $(d; \mathbf{m})$ with prescribed $d$, and a single obstructive class has only finitely many breakpoints by Corollary~\ref{it:one}.

\vspace{2 mm}

\noindent {\bf Step 4.  Accumulation point must be $a_0$.}  
With the key estimate \eqref{eqn:thebound}, we can complete the proof of Theorem \ref{satisfies quadratic equation}.  The main remaining challenge in putting everything together is that \eqref{eqn:thebound}
only holds at the unique point guaranteed by Proposition~\ref{prop:there can be only one} --- recall that we are calling this unique point the {\em breakpoint} --- for the obstructive class $(d;\mathbf{m})$; however, even when $c_X(a) = \mu_{d;\mathbf{m}}(a)$ we cannot guarantee that $a$ is the breakpoint for $(d;\mathbf{m}).$
The example in Figure~\ref{fig:extra-singular-pts} illustrates how a singular point need not
be the breakpoint of an obstruction.
This is not fatal to the argument, but requires a little bit of care.


\begin{center}
  \begin{figure}[ht]
\begin{overpic}[
scale=0.75,unit=1mm]{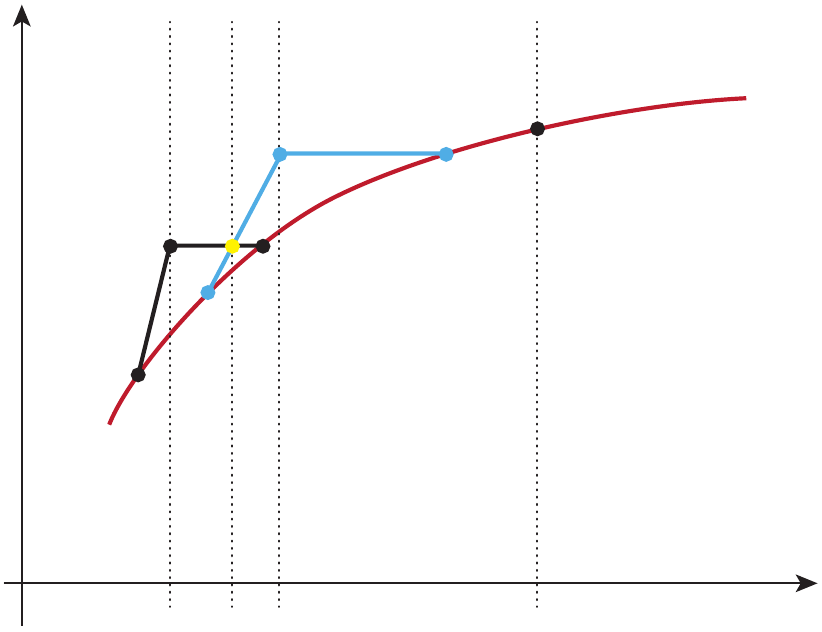}
   \put(20,-1){\small{$a_i$}}
   \put(28,-1){\small{$z_i$}}
   \put(35,-1){\small{$a_{i+1}$}}
   \put(68,-1){\small{$z_\infty$}}
\end{overpic}
 \vskip 0.3in
\caption{The singular point at $z_i$ is not
the breakpoint of a single obstruction. Instead, it arises where two obstructions cross.
The red curve represents the volume lower bound.}
\label{fig:extra-singular-pts}
\end{figure}
\end{center}

The details are as follows.
We assume first that the $z_i$ are converging to $z_{\infty}$ from the left; the argument in the case where the
$z_i$ are converging from the right is completely analogous and so we omit it for brevity. 

We pass to a subsequence of $z_i$ that increase to $z_\infty$.
We now want to take a sequence of obstructive classes.
In fact, we cannot assume that there are any obstructive classes at all at $z_i$, since $c_X(z_i)$ could lie on the volume curve.  However, since the $z_i$ are singular points, we can assume that the $(d_i; {\bf m}_i)$ are obstructive at points $z'_i$ arbitrarily close to any $z_i$; it will be convenient to fix such points $z'_i$ to be within distance $\kappa_i := \frac{|z_i - z_{\infty} |}{2}$ of $z_i$.  
By passing to a subsequence, we can also assume that the $(d_i;{\bf m}_i)$ are distinct in view of Corollary~\ref{it:one}.


\begin{center}
  \begin{figure}[ht]
\begin{overpic}[
scale=0.5,unit=1mm]{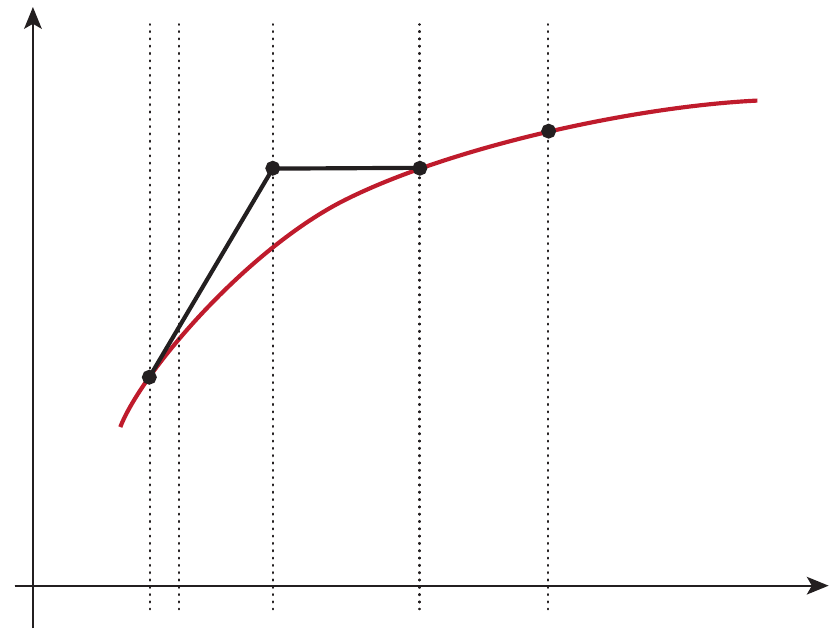}
   \put(10,-1){\small{$z_i$}}
   \put(15,-1){\small{$z_i'$}}
   \put(22,-1){\small{$a_{i}$}}
   \put(34,-1){\small{$z_{i}^*$}}
   \put(45,-1){\small{$z_\infty$}}
   \put(36,-7){(a)}
\end{overpic}\hskip 0.2in
\begin{overpic}[
scale=0.5,unit=1mm]{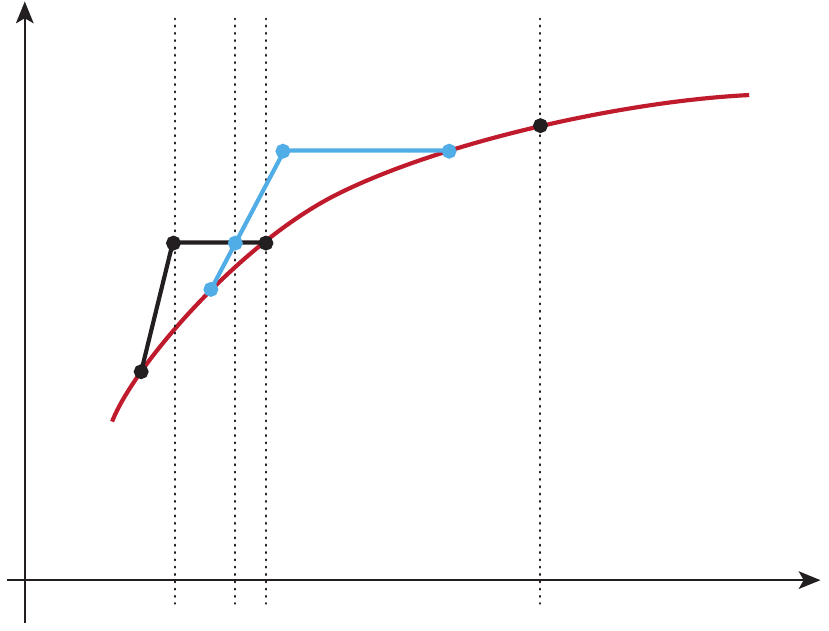}
   \put(14,-1){\small{$a_i$}}
   \put(19,-1){\small{$z_i$}}
   \put(22,-1){\small{$z_{i}^*$}}
   \put(46,-1){\small{$z_\infty$}}
   \put(36,-7){(b)}
\end{overpic}
 \vskip 0.3in
\caption{(a) The point $z_i$ could be a singular point but not a value where the capacity
function is obstructed.  Thus we may need to replace it with a nearby $z_i'$ where the capacity function
is obstructed. We will also refer to the endpoint
$z_i^*$ of the interval obstructed by $(d_i; {\bf m}_i)$.
(b) The singular point $z_i$ could occur at a non-breakpoint 
(exactly as in Figure~\ref{fig:extra-singular-pts} above).  In this case, we can take
$z_i=z_i'$ and still denote by $z_i^*$ the endpoint of the interval obstructed by $(d_i; {\bf m}_i)$. 
In both cases, the red curve represents the volume lower bound.}
\label{fig:z-prime-star}
\end{figure}
\end{center}

Now, let $a_i$ be the breakpoint corresponding to $z'_i$. 
We must have $a_i < z_{\infty}$, since $c_X(a)$ is obstructed on  the interval with endpoints $z'_i$ and $a_i$, but  $c_X(z_{\infty})$ lies on the volume curve by Step 0.
If infinitely many of the $a_i$ satisfy $a_i \ge z'_i$, then necessarily these $a_i$ must be distinct, so by Step 3 they must accumulate at $a_0$, and so since the $z_i$ are accumulating at $z_\infty$, we must have $a_0 = z_\infty$ as claimed.

 Thus, we can assume that
 infinitely many of the $a_i$ satisfy $a_i < z'_i$. Under this assumption, 
 pass to a subsequence so that all $a_i$ have this property.  If infinitely many $a_i$ are converging to $z_\infty$, then we have $a_0 = z _\infty$ by Step 3, and so we are done.  Thus, we can assume that the $a_i$ are uniformly bounded away from $z_\infty$.  
Because $a_i < z'_i$, the function $\mu_{(d_i,{\bf m}_i)}(a)$  is linear on the maximal interval 
$[a_i, z^*_i)$ on which $(d_i, {\bf m}_i)$ is obstructive.  As above, these points $z^*_i$ satisfy 
$z^*_i \le z_{\infty}$ 
since $c_X(z_{\infty})$ lies on the volume curve. 
As the $z_i$ are converging to $z_{\infty}$, it then follows that the $z^*_i$ must be as well; it 
therefore follows that the slope of the volume curve at $z^*_i$ is converging to the slope of the 
volume curve at $z_{\infty}$.  
Now the line segment from  $(a_i,\mu_{ (d_i, \bf{m}_i ) }(a_i))$  to $\left(z_i^*, \sqrt{ \frac{z^*_i}{vol}}\right)$ lies
above the volume curve.  So this line segment
must also be above the tangent line to the volume curve at $z_i^*$ on the interval $(a_i,z_i^*)$, because the volume curve is concave.   
Thus, we must have a uniform upper
bound on $d_i$ across the $(d_i, {\bf m}_i)$, by \eqref{dbound}: the length of the 
interval $(a_i,z^*_i)$ is uniformly bounded away from zero, so 
$\left|\mu_{ (d_i, \bf{m}_i ) }(a_i) - \sqrt{ \frac{a_i}{vol}}\right|$ is uniformly bounded away from zero,  
as a consequence of the upper bound on the slope of the line segment.  The  uniform bounds on the $d_i$ mean we have only finitely many obstructive classes, but we are assuming that the $(d_i, {\bf m}_i)$ are distinct, which is a contradiction.
\end{proof}

\begin{remark}
It would be interesting to understand whether an analogue of Theorem \ref{satisfies quadratic equation} still holds, without the assumption of finitely many $b_i$; this could be useful for understanding embeddings into an irrational ellipsoid, for example.  Most of the above argument should go through, except that now the number $N$ in \eqref{eqn:thebound} would be infinite. 
It is nevertheless plausible that there is a way around this.
\end{remark}

\begin{remark}
\label{rmk:asymptotics}
We could alternatively think about Theorem \ref{satisfies quadratic equation} from the point of view of the ECH capacities reviewed in Section \ref{sec:ech}.  This works as follows.

Normalize the domain to have the same volume as the target; in other words, consider the problem of embedding an ellipsoid $E\left(\sqrt{\frac{ \op{vol} }{a}}, \sqrt{a \op{vol}}\right)$ into $X$.  If we assume that $a$ is irrational, and set the perimeter of the domain and the target equal to each other, we get the equation
\begin{equation}
\label{eqn:perimetersequal}
\sqrt{\frac{ \op{vol}}{a}} + \sqrt{a \op{vol}} = \op{per},
\end{equation}
which can be rearranged to \eqref{eqn:theequation}.

There is in turn a heuristic for why \eqref{eqn:perimetersequal} is natural to consider in view of the question of 
finding infinite staircases for this problem from the point of view of ECH capacities.  The justification for normalizing 
the volumes to be equal is as follows: by packing stability (Proposition~\ref{prop:properties}\eqref{it:stab}),
an infinite staircase must accumulate at {\em some} 
point $s_0$.  It is not too hard to show in addition that the embedding function at $s_0$ must
lie on the volume curve, as in Step~3 of the Proof of Theorem~\ref{satisfies quadratic equation} above.  

Now, it has been shown \cite[Theorem~1.1]{asymptotics} that for the manifolds we consider here,  asymptotically 
ECH capacities recover the volume; moreover, the subleading asymptotics have recently been studied, see for 
example \cite[Theorem~3]{cgs}, 
and in the present situation these next order asymptotics are well-understood as well.  These can be interpreted as recovering 
the perimeter (see \cite[Proposition~16]{cgs}). These asymptotics dominate when we normalize the leading asymptotics, 
which are the volume.  

With all of this understood, here is the promised heuristic: if the subleading asymptotics of the domain are 
larger than the subleading asymptotics of the target (which happens when $s_0$ irrational is smaller than the 
solution to \eqref{eqn:perimetersequal}), then no volume preserving embedding can exist.  On the other hand, 
if the subleading asymptotics of the domain are smaller than the subleading asymptotics of the target (which 
happens when $s_0$ irrational is larger than the solution to \eqref{eqn:perimetersequal}), then only finitely 
many ECH capacities can give an obstruction at $s_0$; one might then hope that the same is true in a neighborhood of $s_0$, and then argue that these finitely many capacities are not enough to generate an infinite staircase.  
Assuming all this, at least when $s_0$ is irrational the only possibility would be that the accumulation point is actually given by the relevant solution to \eqref{eqn:perimetersequal}.

Note, however, that this is quite different than the proof we give above for Theorem \ref{satisfies quadratic equation}.  
It is easy to make the heuristic above rigorous concerning the case where the subleading asymptotics of the domain 
are smaller than the subleading asymptotics of the target; and, in the other case, it is easy to make rigorous that {\em at} $s_0$, there are only finitely many ECH capacities that are obstructive; however, establishing the needed result in a neighborhood of $s_0$ is much more subtle.    
Of course, another issue is that we assumed $s_0$ irrational in our heuristic discussions.  If $a = s_0$ is rational instead of irrational, then the perimeter of 
the domain is different than what is said above, so \eqref{eqn:perimetersequal} does not hold.  
All of this is why we give a rather 
different argument, inspired by the work of McDuff and Schlenk in \cite{mcduffschlenk}.
\end{remark}

We now also give the promised proof of the fifth item in Proposition~\ref{prop:properties}, which uses some of the same ideas as in the proof of Theorem~\ref{satisfies quadratic equation}.   

\begin{proof}[Proof of Proposition~\ref{prop:properties}\eqref{it:piecewise}]
Let $\widetilde{a}$ be a point which is not a limit of singular points.  
Then, there is some open interval $I=(m,n)$ containing $\widetilde{a}$ on which the only possible singular point of $c_X(a)$ is $\widetilde{a}$ itself.  If $c_X(a)$ is equal to the volume obstruction on $I$, then the conclusion of the proposition holds near $\widetilde{a}$.
Thus we can assume there is some point $y$ in $I$ on which $c_X(y)$ is strictly greater than the 
volume obstruction; without loss of generality, assume that $y < \widetilde{a}$.  

There is now some subinterval $I' \subset I$, containing $y$,  on which $c_X(a)$ is the supremum of finitely many obstructive classes, each of which is piecewise linear on $I'$, with at most one singular point.  It follows that $c_X$ is piecewise linear on $I'$; since $\widetilde{a}$ is the only possible singular point of $c_X(a)$ on $I$, it follows that in fact $c_X(a)$ is linear on $(m,\widetilde{a}].$  We now apply the same argument to the interval $(\widetilde{a},n)$.  Namely, if $c_X(a)$ is the volume on $(\widetilde{a},n)$, then the conclusion of the proposition holds near $\widetilde{a}$, so we are done. 
Otherwise, we can assume there is some point $y'$ in $(\widetilde{a},n)$ such that $c_X(y')$ is strictly greater the volume obstruction.  
Then, as in the $y < \widetilde{a} $ case, $c_X(a)$ is linear on $[\widetilde{a},n)$, as desired.
\end{proof}

\section{The existence of the Fano staircases}\label{section:existence}

\noindent To prove Theorem \ref{thm:main}, we begin by showing that the purported $x$-coordinates 
intertwine: $x^\outter_n < x^\inner_n < x^\outter_{n+1}$.
We then take a limit as  $n\to\infty$, verifying that the $x$-coordinates  $x^\outter_n$ tend to  $a_0$  (and therefore also $x^\inner_n$ 
tend to $a_0$ as well) and $y^\outter_n=y^\inner_n$ tend to $\sqrt{\frac{a_0}{\vol}}$.
Next we show that $y^\outter_n\leq c_X(x^\outter_n)$ and that  $c_X(x^\inner_n)\leq y^\inner_n$. For the first inequality, 
we find an obstruction, and for the second, we produce an explicit embedding.
Finally, we use the fact that $c_X(a)$ is continuous, non-decreasing, and has the scaling property to 
conclude that the graph of the function must consist of  line segments alternately joining points of the 
two sequences $(x^\inner_n,y^\inner_n)$ and $(x^\outter_n,y^\outter_n)$, and that these line segments 
alternate: some are horizontal and the others, when extended to be lines, pass through the origin.  This is illustrated in Figure~\ref{fig:proofpic}.

\newpage

\begin{figure}[h!]
\centering
\includegraphics[width=0.55\textwidth]{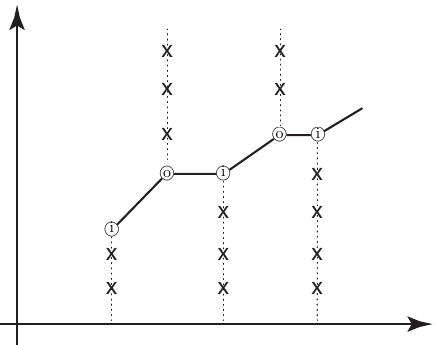}
\caption{The inner corners are marked \textcircled{i} and the outer corners are
marked \textcircled{o}. The exed out lines represent the inequalities $y^\outter_n\leq c_X(x^\outter_n)$ and $c_X(x^\inner_n)\leq y^\inner_n$. The properties of the embedding capacity function then imply that its graph consists of line segments between the corners.}
\label{fig:proofpic}
\end{figure}

Before we begin, it will be convenient to catalogue certain combinatorial identities 
that hold for our sequences.  These will be essential for the inductive proofs that follow.

\begin{lemma}\label{lemma:identities}
Let $X$ be a convex toric domain with negative weight expansion $(b;b_1,\ldots,b_n)$ 
equal to 
$$
(3)\ , \ (3;1)\ , \ (3;1,1)\ , \ (3;1,1,1)\ , \ (3;1,1,1,1)\ , \mbox{or } (4;2,2).
$$
When $J=2$, that is, for the sequences with recurrence relation $g(n+4)=Kg(n+2)-g(n),$ the following identities hold:
\begin{equation*}g(n)+g(n+2)=\beta_{n+1}g(n+1) \tag{$\clubsuit$}\end{equation*}
\begin{equation*}g(n)^2+g(n+2)^2-K g(n)g(n+2)=- \alpha \beta_{n+1}, \tag{$\vardiamondsuit$}\end{equation*}
\begin{equation*}g(n)g(n+3)=g(n+1)g(n+2)+\alpha \tag{$\varheartsuit$}\end{equation*}
where $K=\vol-2$, the sequence seeds $g(0),\ldots, g(3)$, and the parameters $\alpha$ and $\beta_n$ are:


\renewcommand{\arraystretch}{1.5}\begin{table}[htbp]
\centering 
\scriptsize
\begin{tabular}{| c | c | c | c | c  |}
\hline

Negative weight expansion & $K$ & Seeds & $\alpha$    & $\beta_n$    
\\ \hline \hline
$(3)$     & $7$ & $2,1,1,2$ & $3$  & $3$  
\\  \hline
$(4;2,2)$  & $6$ & $1,1,1,3$ & $2$  & $ \left\{\begin{array}{ll}2,& n \text{ odd}\\ 4,& n \text{ even} \end{array}\right.$ 
 \\   \hline
$(3;1,1,1)$  & $4$ & $1,1,1,2$  & $1$  & $\left\{\begin{array}{ll}2,& n \text{ odd}\\ 3,& n \text{ even} \end{array}\right.$
\\   \hline
$(3;1,1,1,1)$   & $3$ & $1,2,1,3$ & $1$  & $\left\{\begin{array}{ll}1,& n \text{ odd}\\ 5,& n \text{ even} \end{array}\right.$
\\  \hline
\end{tabular}
\end{table}

When $J=3$, that is, for the sequences with recurrence relation $g(n+6)=Kg(n+3)-g(n)$, the following identities hold:

\begin{equation*} g(n)+g(n+3)=
\left\{\begin{array}{ll}
g(n+1)+g(n+2),& n\equiv0 \mod 3\\   
2g(n+1)+g(n+2),& n\equiv1 \mod 3\\  
g(n+1)+2g(n+2),& n\equiv2 \mod 3\\ 
\end{array}\right. \tag{$\clubsuit$ for (3;1)}\end{equation*}

\begin{equation*} g(n)+g(n+3)=
\left\{\begin{array}{ll}
g(n+1)+g(n+2),& n\equiv0 \mod 3\\   
g(n+1)+2g(n+2),& n\equiv1 \mod 3\\  
2g(n+1)+g(n+2),& n\equiv2 \mod 3\\ 
\end{array}\right. \tag{$\clubsuit$ for (3;1,1)}\end{equation*}

\begin{equation*}g(n)^2+g(n+3)^2-K g(n)g(n+3)=- \beta_{n+1}, \tag{$\vardiamondsuit$}\end{equation*}
\begin{equation*}g(n)g(n+4)=g(n+1)g(n+3)+\delta_n \tag{$\varheartsuit.1$}\end{equation*}
\begin{equation*}g(n)g(n+5)=g(n+2)g(n+3)+\mu_n \tag{$\varheartsuit.2$}\end{equation*}
where $K=\vol-2$, the sequence seeds $g(0),\ldots, g(5)$, and the parameters $\beta_n$, $\delta_n$ and $\mu_n$ are,
with all congruences \hskip -0.1in $\mod 3$:

\renewcommand{\arraystretch}{1.5}\begin{table}[htbp]
\centering 
\scriptsize
\begin{tabular}{| c | c | c | c | c | c |}
\hline
  &  $K$ & Seeds & $\beta_n$    &  $\delta_n$ & $\mu_n$ 
\\ \hline \hline
$(3;1)$   & $6$ & $1,1,1,1,2,4$ &  $\left\{\begin{array}{ll}4,& n\equiv 1 \\ 
7,& n\equiv 0,2 \end{array}\right.$  
& $ \left\{\begin{array}{ll}1,& n\equiv 0,2 \\ 2,& n\equiv 1  \end{array}\right.$    & $3$ 
\\  \hline
$(3;1,1)$  & $5$ & $1,1,1,1,2,3$  &  $ \left\{\begin{array}{ll}3,& n\equiv 1 \\ 5,& n\equiv 0,2 
\end{array}\right.$  & $1$ &  $\left\{\begin{array}{ll}2,& n\equiv 0,1 \\ 3,& n\equiv 2  \end{array}\right.$  
\\  \hline
\end{tabular}
\end{table}

\end{lemma}

\begin{proof}
We prove the identities for $J=2$, the ones for $J=3$ are proved similarly.
For the base case, it is straightforward to show that each of the identities holds for $n=0$ and $1$.

Next, we prove the induction step for $\clubsuit$. For this purpose, we assume that the identity holds for $n=m$ and $n=m+1$ and aim to show that it also holds for $n=m+2$:
{\small
$$
\begin{array}{lr}
  g(m+2)+g(m+4) &  \\
\phantom{..} =  (K+1)g(m+2)-g(m)  & \phantom{..} \text{{\footnotesize (1)}}\\
\phantom{..}= (K+2)g(m+2) -\left(g(m+2)+g(m)\right)&  \text{{\footnotesize (2)}}\\
\phantom{..}= (K+2)g(m+2)-\beta_{m+1}g(m+1) &\text{{\footnotesize (3)}}\\
\phantom{..}= (K+2)g(m+2)-\beta_{m+1}\left(g(m+1)\right) +g(m+3)  & \\
 \phantom{=(K+2)} +\beta_{m+1}g(m+3) &\text{{\footnotesize (4)}}\\
\phantom{..}= (K+2)g(m+2)-\beta_{m+1}\beta_{m+2}g(m+2) +\beta_{m+1}g(m+3)  &\text{{\footnotesize (5)}}\\
\phantom{..}=\beta_{m+3}g(m+3). &\text{{\footnotesize (6)}} 
\end{array}
$$
}

In line (1) above we used the recurrence relation formula, in lines (3) and (5) we used the induction hypothesis, and in line (6) we used the facts that $\beta_n$ only depends on the parity of $n$ and that $\beta_{n} \beta_{n+1}=K+2$.

Now we move on to $\vardiamondsuit$. We show that the left-hand side of the identity with $n=m+2$ equals the left-hand side with $n=m$. Since the right-hand side only depends on the parity of $n$, this will imply that the identity holds for all $n$.
Using the recurrence relation in line (2), we have:
{\small
\begin{align*}
g(m+4)^2-g(n)^2&=\left(g(m+4)+g(n)\right)\left(g(m+4)-g(n)\right) & \text{{\footnotesize (1)}}\\
&=K g(m+2)\left(g(m+4)-g(n)\right)& \text{{\footnotesize (2)}}\\
&=K g(m+2) g(m+4)- K g(m) g(m+2).&\text{{\footnotesize (3)}}
\end{align*}
}
Rearranging and then adding $g(m+2)^2$ on both sides, we get as desired:
$$
\begin{array}{c}
g(m+4)^2+g(m+2)^2-K g(m+2)g(m+4) = \phantom{BOOMBOO} \\
\phantom{BOOMBO} =g(m+2)^2+g(m)^2-K g(m)g(m+2).
\end{array}
$$

Finally, we tackle $\varheartsuit$. For this purpose, we assume that the identity holds for $n=m$ and aim to show that it also holds for $n=m+2$:
{\small
\begin{align*}
g(m+2) g(m+5)&=g(m+2)\left(K g(m+3)-g(m+1)\right) & \text{{\footnotesize (1)}}\\
&=K g(m+2) g(m+3)-g(m+1) g(m+2)& \text{{\footnotesize (2)}}\\
&=K g(m+2) g(m+3)- g(m) g(m+3) + \alpha &\text{{\footnotesize (3)}}\\
&=g(m+3)\left(K g(m+2)-g(m)\right)+\alpha & \text{{\footnotesize (4)}}\\
&=g(m+3)g(m+4) + \alpha.&\text{{\footnotesize (5)}}
\end{align*}
}

\noindent Here we have used the recurrence relation formula in lines (1) and (5), and the induction hypothesis in line (3).
\end{proof}

We now use the identities in Lemma~\ref{lemma:identities} to establish the relationships among
the $x$- and $y$-coordinates of purported corners of the ellipsoid embedding functions.

\begin{proposition}\label{prop:3things}
The recurrence relations above define inner and outer corners
respectively with coordinates:
$$
\textstyle{(x^\inner_n,y^\inner_n)=\left( \frac{g(n+J)\left( g(n+1)+g(n+1+J) \right)}{\left(g(n)+g(n+J)\right) g(n+1)},  \frac{g(n+J)}{g(n)+g(n+J)}  \right),}$$
$$
\textstyle{(x^\outter_n,y^\outter_n)= \left( \frac{g(n+J)}{g(n)},\frac{g(n+J)}{g(n)+g(n+J)} \right).}$$ 
These coordinates satisfy:
\begin{enumerate}
\item $\displaystyle{x^\outter_n < x^\inner_n < x^\outter_{n+1}}$ ; 
\item $\displaystyle{\lim_{n\to\infty} x^\outter_n=\lim_{n\to\infty} x^\inner_n=a_0}$ ; and
\item $\displaystyle{\lim_{n\to\infty} y^\outter_n=\lim_{n\to\infty} y^\inner_n=\sqrt{\frac{a_0}{\vol}}}$ .
\end{enumerate}
\end{proposition}

\begin{proof}
For (1):  Both inequalities boil down to showing that 
$$g(n+1)  g(n+J)  < g(n) g(n+J+1),$$
which follows immediately from the identities ($\varheartsuit$) in Lemma \ref{lemma:identities}.

For (2): In view of (1), it suffices to show that $\displaystyle{\lim_{n\to\infty} x^\outter_n=a_0}$.
The linear recurrence relation 
$$g(n+2J)=K g(n+J)-g(n)$$ 
is of order $2J$ but can be replaced by $J$ linear recurrence relations of order 2, one for each of the $J$ subsequences of $g(n)$ with $n\equiv j$ (mod $J$), for $j=0,1,\ldots, J-1$. Each of these subsequences has the recurrence relation 
\begin{equation}\label{eq:smaller recurrence}
g_j(n+2)=K g_j(n+1)-g_j(n),
\end{equation}
where $g_j(n)=g(J n+j)$.

We can get a closed form for $g_j(n)$ by solving the polynomial equation $\lambda^2=K\lambda-1$. Let $\lambda_1, \lambda_2$ be the roots of this equation.  We note that in each of the cases we are considering we have $\lambda_1>1>\lambda_2>0$. Then for appropriate coefficients $D_j,E_j$ depending on the seed of the sequences, 
\begin{equation}\label{eq:closed form}
g_j(n)=D_j \lambda_1^n+E_j \lambda_2^n.
\end{equation}

For each $j=0,1,\ldots,J-1$ we have
$$\lim_{n\to\infty}\frac{g_j(n+1)}{g_j(n)}=\lim_{n\to\infty}\frac{D_j \lambda_1^{n+1}+E_j \lambda_2^{n+1}}{D_j \lambda_1^n+E_j \lambda_2^n} =\lambda_1.$$
Noting that $a_0$ is exactly $\lambda_1$, the larger solution of $\lambda^2-K\lambda+1=0$, we conclude as desired that $\lim_{n\to\infty}x_n=\lambda_1=a_0.$

Finally, for (3): In view of the fact that  $y^\outter_n = y^\inner_n$, it suffices to show that 
$$\lim_{n\to\infty} y^\outter_n=\sqrt\frac{a_0}{\vol}.$$ 
Indeed we have
$$\lim_{n\to\infty}\frac{1}{y^\outter_n}=\lim_{n\to\infty}\left(\frac{1}{x^\outter_n}+1\right)=\frac{1}{a_0}+1=\sqrt\frac{\vol}{a_0},$$
the last equality uses the facts that $a_0^2-K a_0+1=0$ and $\vol=K+2$.
This completes the proof.
\end{proof}

Next, we show that at the outer corners, the ellipsoid embedding function is indeed obstructed in all of our
six examples.

\begin{proposition}\label{prop:outer}
Let $X$ be a convex toric domain with negative weight expansion $(b;b_1,\ldots,b_n)$ is
$$\textstyle{
(3)\ , \ (3;1)\ , \ (3;1,1)\ , \ (3;1,1,1)\ , \ (3;1,1,1,1)\ , \mbox{or } (4;2,2).
}
$$
For each $(x^\outter_n,y^\outter_n)= \left( \frac{g(n+J)}{g(n)},\frac{g(n+J)}{g(n)+g(n+J)} \right)$, we have the inequality $y^\outter_n\leq c_X(x^\outter_n)$.
\end{proposition}

\begin{proof}
Recall that $N(a,b)_k$ denotes the $k^{\mathrm{th}}$ ECH capacity of $E(a,b)$ and $c_k(X)$ denotes the 
 $k^{\mathrm{th}}$ ECH capacity of $X$.
Define $k_n:=\frac{(g(n)+1)(g(n+J)+1)}{2}-1$. If we prove that 
\begin{equation}\label{eq: things work out with the right kn}
g(n+J)\leq N(1,x^\outter_n)_{k_n} \,\,\,\,\mbox{ and }\,\,\,\, g(n)+g(n+J)\geq c_{k_n}(X),
\end{equation}
then we have the desired inequality:
$$
\textstyle{y^\outter_n=\frac{g(n+J)}{g(n)+g(n+J)}\leq \frac{N(1,x^\outter_n)_{k_n}}{c_{k_n}(X)}\leq\sup_k\frac{N(1,x^\outter_n)_{k}}{c_k(X)}=c_X(x^\outter_n).}$$
The first part of \eqref{eq: things work out with the right kn} can be rewritten as
$$g(n) g(n+J)\leq N(g(n),g(n+J))_{k_n},$$
and by Proposition \ref{facts about sequences} it is indeed true that there are at most $k_n$ terms of the sequence $N(g(n),g(n+J))$ strictly smaller than $g(n) g(n+J)$. 

Next, we tackle the second part of  \eqref{eq: things work out with the right kn}:
$$ g(n)+g(n+J)\geq c_{k_n}(X).$$
By Theorem \ref{thm:lattice path}, it suffices to find a convex lattice path $\Lambda_n$ that encloses $k_n+1$ lattice points and has $\Omega$-length   $g(n)+g(n+J)$:
\begin{equation}\label{eq:points and length} 
\textstyle{\mathcal{L}(\Lambda_n)= \frac{(g(n)+1)(g(n+J)+1)}{2} \,\,\,\,\mbox{and}\,\,\,\, \ell_\Omega(\Lambda_n)=g(n)+g(n+J).}
\end{equation}
We do this separately for each of the six cases under consideration in Appendix~\ref{ap:convex lattice paths}.

The convex lattice paths $\Lambda_n$ for each case (and sub-case) can be found in Figure~\ref{fig:paths}, while the formul\ae\ for $s_n$ and $t_n$ are provided in \eqref{eq:sn and tn}. Using the identities $(\vardiamondsuit)$ in Lemma \ref{lemma:identities} and induction we conclude that $s_n$ and $t_n$ are indeed integers.

Formul\ae\ for the number of lattice points $\mathcal{L}(\Lambda_n)$ enclosed by the path and its $\Omega$-length 
$\ell_\Omega(\Lambda_n)$ are provided in Table~\ref{tab:calL ell} for each case and sub-case. 
As discussed after Table~\ref{tab:calL ell}, it is a computational check that these give the correct numbers
that satisfy \eqref{eq:points and length}.
\end{proof}

Next, we show that at the inner corners, there are explicit ellipsoid embeddings 
realizing the purported value of the ellipsoid embedding function. We do this by 
exploring recursive families of ATFs, following a suggestion of Casals. The idea 
that the recurrence sequences involved in the coordinates of the corners of the 
infinite staircases may be related to the Markov-type equations that show up 
when performing ATF mutations was first mentioned to us by Smith and is studied in detail by 
Maw for symplectic del Pezzo surfaces \cite{maw-thesis}.
This procedure is explained nicely in Evans' lecture notes \cite[Example~5.2.4]{evans}.
We use a series of mutations first described by Vianna \cite[\S 3]{vianna}
on the compact manifolds corresponding to our negative weight expansions with $J=2$.  
The $J=3$ ATFs have base diagram a quadrilateral and do not
seem to have been explicitly used before, though Vianna has introduced quadrilateral based
ATFs  in \cite[Figs~7 and 8]{vianna}.  
It could be interesting to explore the number theory and exotic Lagrangian tori
that these produce.
In algebraic geometry (and looking at the dual lattice), one can also study a related operation
also called mutation, which is
a combinatorial operation arising from the theory of cluster algebras.  This is explored in
\cite{KW}; in particular see Example 1.2  and references therein.

\begin{proposition}\label{prop:inner}
Let $X$ be a convex toric domain whose negative weight expansion is 
$$
(3)\ , \ (3;1)\ , \ (3;1,1)\ , \ (3;1,1,1)\ , \ (3;1,1,1,1)\ , \mbox{or } (4;2,2).
$$
For each $(x^\inner_n,y^\inner_n)=\left( \frac{g(n+J)\left( g(n+1)+g(n+1+J) \right)}{\left(g(n)+g(n+J)\right) g(n+1)},  \frac{g(n+J)}{g(n)+g(n+J)}  \right)$, there is a symplectic embedding 
\begin{equation}\label{eq:inner embed}
E(1,x^\inner_n)\embeds y^\inner_n X,
\end{equation} 
which forces $c_X(x^\inner_n)\leq y^\inner_n$.
\end{proposition}

 \begin{proof}
We use Theorem~\ref{prop:mainproposition} and Proposition~\ref{prop:triangleintothing} to prove that there is an embedding 
\begin{equation}\label{eq:ctd inner}
E\left(\frac{g(n)+g(n+J)}{g(n+J)}, \frac{g(n+1)+g(n+1+J)}{g(n+1)} \right)\embeds X,
\end{equation}
which is equivalent to \eqref{eq:inner embed}. By definition of $c_X(a)$, this implies the desired inequality. The proof consists of applying successive mutations to base diagrams, beginning with a Delzant polygon.  This allows us to use 
Proposition~\ref{prop:triangleintothing} to find
ellipsoids embedded in compact manifolds.
Theorem~\ref{prop:mainproposition} then allows us to deduce that those ellipsoids must also be embedded in the 
corresponding convex toric domain.
Since two convex toric domains with the same negative weight expansions have identical ellipsoid embedding functions (see Remark \ref{rmk:dependsonly}), it suffices to exhibit the embeddings for one convex toric domain per negative weight expansion. 
We must take particular care with the negative weight expansion $(3;1,1,1,1)$, making use of Lemma~\ref{rem:polygon}.

We begin by producing ATFs
on the compact manifolds $M$ corresponding to our negative weight expansions.
The manifolds are
$$
\C P^2_3\ ; \  \C P^2_3\# \overline{\C P}^2_1\ ; \  \C P^2_3\# 2\overline{\C P}^2_1\ ; \  \C P^2_3\# 3\overline{\C P}^2_1\ ;
$$
$$
 \C P^2_3\# 4\overline{\C P}^2_1\ ; \ \mbox{and } \C P^1_2\times \C P^1_2.
$$
Except for $ \C P^2_3\# 4\overline{\C P}^2_1$, these manifolds may be endowed 
with toric actions.  The corresponding Delzant polygons are displayed in Figure~\ref{fig:delzant}.
Our first step is to apply mutations to the Delzant polygons to produce a base diagram
that is a triangle with two nodal rays when $J=2$ and a 
quadrilateral with three nodal rays when $J=3$.   For $\C P^2_3\# 4\overline{\C P}^2_1$, we use
Vianna's trick \cite[\S3.2]{vianna} to find an appropriate ATF on this manifold.  Specifically, we begin with 
the ATF on $\C P^2_3\# 3\overline{\C P}^2_1$ given in Figure~\ref{fig:3-1-1-1-pregame}(e).  This ATF
has a smooth toric corner at the origin where we may perform a toric blowup of symplectic size $1$, resulting
in an ATF on $\C P^2_3\# 4\overline{\C P}^2_1$.  These initial maneuvers are described in Appendix~\ref{ap:atfs}
and the results are shown in Figure~\ref{fig:pregame}.

\begin{center}
  \begin{figure}[ht]
\begin{overpic}[
scale=0.5,unit=1mm]{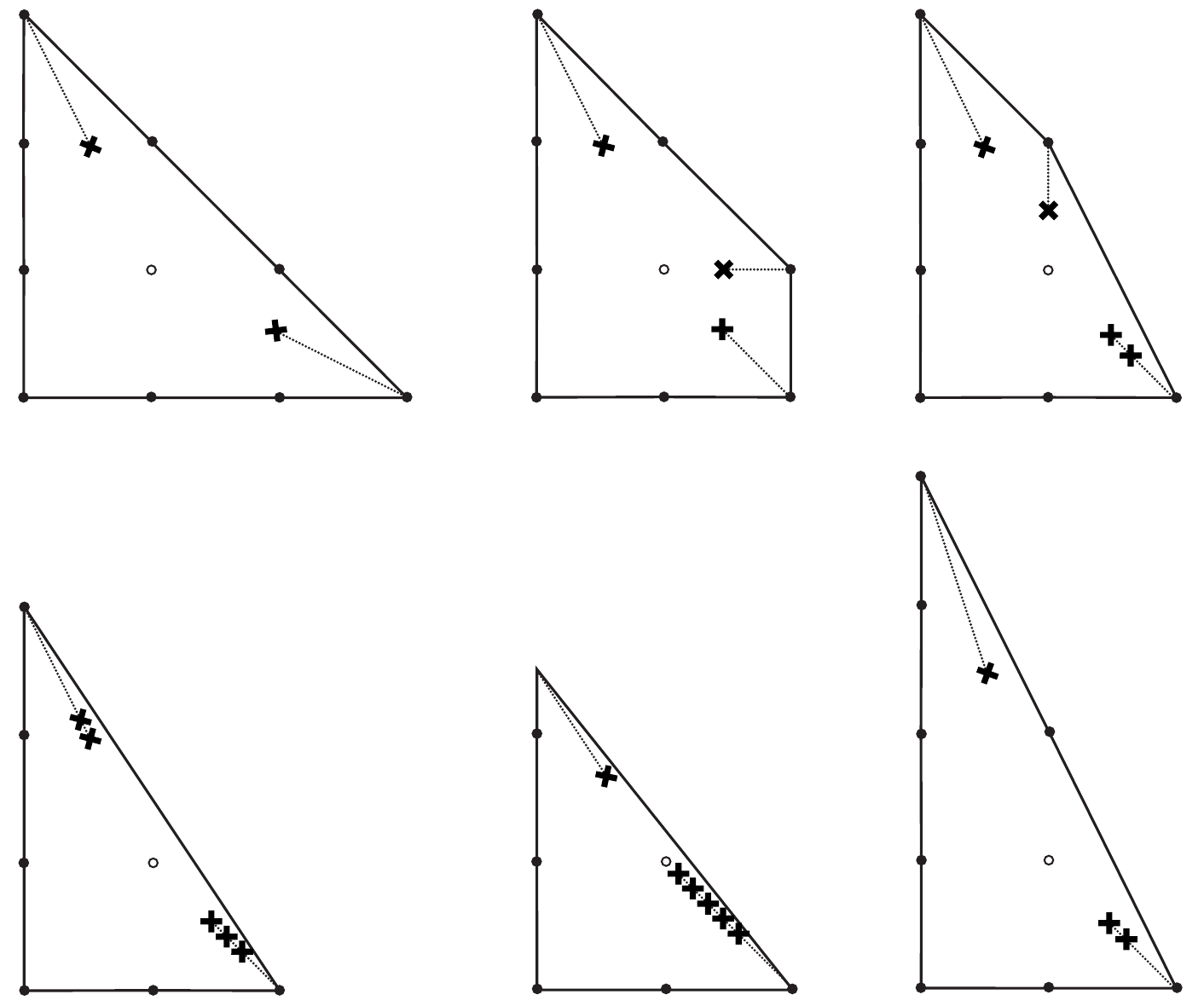}
   \put(18,55){\tiny{$\C P^2_3$}}
   \put(60,55){\tiny{$\C P^2_3\# \overline{\C P}^2_1$}}
   \put(97,55){\tiny{$\C P^2_3\# 2\overline{\C P}^2_1$}}
   \put(8,-4){\tiny{$\C P^2_3\# 3\overline{\C P}^2_1$}}
   \put(60,-4){\tiny{$\C P^2_3\# 4\overline{\C P}^2_1$}}
   \put(97,-4){\tiny{$\C P^1_2\times \C P^1_2$}}
\end{overpic}
\vskip 0.15in

\caption{The base diagrams for ATFs on our manifolds.  These
are a triangle with two nodal rays when $J=2$ and a 
quadrilateral with three nodal rays when $J=3$.}
\label{fig:pregame}
\end{figure}
\end{center}

We now want to show that for any $0<\varepsilon<1$,
\begin{equation}\label{eq:atf inner}
\textstyle{
(1-\varepsilon)\cdot E\left(\frac{g(n)+g(n+J)}{g(n+J)}, \frac{g(n+1)+g(n+1+J)}{g(n+1)} \right)\embeds M,
}
\end{equation}
where $M$ is the compact manifold from our list.
We achieve this by showing that the base diagram obtained at each additional mutation 
contains the triangle with vertices 
$$\textstyle{
(0,0),\ 
\left(\frac{g(n)+g(n+J)}{g(n+J)},0\right), \mbox{ and } \left(0,\frac{g(n+1)+g(n+1+J)}{g(n+1)} \right). 
}
$$
We will proceed by induction.   We will include the details for the first few steps in
two worked examples in the case $J=2$ in Example~\ref{ex:worked1} and $J=3$ in Example~\ref{ex:worked2}.

In Table~\ref{tab:sigma}, we record the additional data we will need for our recursive 
mutation procedure.

\renewcommand{\arraystretch}{1.5}\begin{table}[h]
\begin{tabular}{| c || c | c | } 
\hline

Negative weight expansion & $J$ & $\sigma_n$
\\ \hline \hline
$(3)$     & $2$ & $1$
\\  \hline
$(3;1,1,1)$  & 2 & $\Bigg\{\begin{array}{ll}2,& n \text{ odd}\\ 3,& n \text{ even} \end{array}$ 
\\   \hline
$(3;1,1,1,1)$   & $2$ & $\Bigg\{\begin{array}{ll}1,& n \text{ odd}\\ 5,& n \text{ even} \end{array}$  
\\  \hline
$(3;1)$   & $3$ & 1
\\  \hline
$(3;1,1)$  & $3$ &   $ \Bigg\{\begin{array}{ll} 2,& n\equiv 0 \mod 3 \\1, & n\equiv 1,2 \mod 3\end{array}$  
\\  \hline$(4;2,2)$ & $2$ & $ \Bigg\{\begin{array}{ll}1,& n \text{ odd}\\ 2,& n \text{ even} \end{array}$    
 \\   \hline
\end{tabular}
\caption{Additional data, by negative weight expansion.}\label{tab:sigma}
\end{table}

\newpage

We now treat separately the cases where $J=2$ and $J=3$, 
starting with $J=2$.  In this case, starting with base diagrams in Figure~\ref{fig:pregame} and
continuing to apply mutations, all further base diagrams will be triangles.
Following the notation in Figure~\ref{fig:trianglen}(a), 
we will repeatedly apply mutation at the vertex
on the positive $y$-axis, along the nodal
ray $v_n$ emanating from that vertex.
The induction hypothesis is that the triangle $\Delta_n$ has side lengths $a_n,b_n,c_n$, nodal rays $v_n$ and 
$u_n$, hypotenuse direction vector $w_n$ as shown in Figure~\ref{fig:trianglen}(a), and  matrix 
that takes $\Delta_n$ to $\Delta_{n+1}$ is $M_n$:
$$v_n=\begin{psmallmatrix} g(n+1)\\-g(n+3) \end{psmallmatrix},\,\,\, 
u_n=\begin{psmallmatrix} -g(n)\\g(n+2) \end{psmallmatrix},\,\,\,
w_n=\begin{psmallmatrix}\sigma_{n+1}g(n+1)^2\\ -\sigma_{n+2}g(n+2)^2 \end{psmallmatrix},$$
$$ a_n=\frac{g(n+1)+g(n+3)}{g(n+1)},\,\,\,b_n=\frac{g(n)+g(n+2)}{g(n+2)},$$
$$\mbox{and }M_n=\begin{pmatrix} -\sigma_{n+2}\,g(n+2)^2 &  -\sigma_{n+1}\,g(n+1)^2   \\ \sigma_{n+3}\,g(n+3)^2&2+\sigma_{n+2}\,g(n+2)^2\end{pmatrix},$$
where $\sigma_n$ is as in Table~\ref{tab:sigma}. 

\newpage 

\begin{center}
  \begin{figure}[ht]
\begin{overpic}[
scale=0.9,unit=1mm]{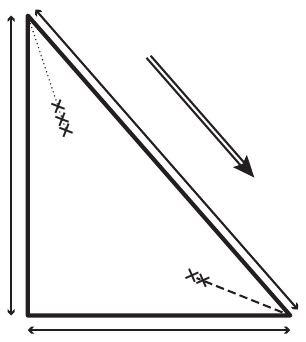}
   \put(11,28){{\small{$v_n$}}}
   \put(23,11){{\small{$u_n$}}}
   \put(31,36){{\small{$w_n$}}}
   \put(-3,25){\small{$a_n$}}
   \put(23,-2){\small{$b_n$}}
   \put(26,28){\small{$c_n$}}
   \put(22,-9){(a)}
\end{overpic}
\hskip 0.75in
\begin{overpic}[
scale=0.9,unit=1mm]{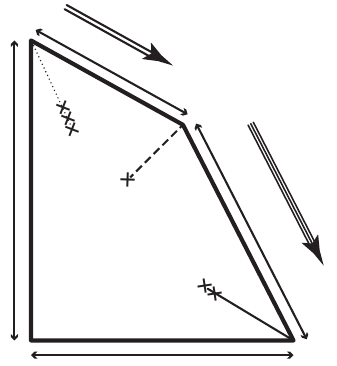}
   \put(25,13){{\small{$u_n$}}}
   \put(16,25){{\small{$v_n$}}}
   \put(11,32){{\small{$w_n$}}}
      \put(45,28){{\small{$s_n$}}}
      \put(18,52){{\small{$r_n$}}}
   \put(-3,25){\small{$a_n$}}
   \put(23,-2){\small{$b_n$}}
   \put(38,23){\small{$c_n$}}
   \put(15,46.5){\small{$d_n$}}
   \put(22,-9){(b)}
\end{overpic} \vskip 0.3in
\caption{The general base diagrams  (a)  $\Delta_n$ for the cases when $J=2$; and (b) $\square_n$ 
for the cases when $J=3$.  The mutation from $\Delta_n$ (respectively $\square_n$) to $\Delta_{n+1}$ 
(resp.\ $\square_{n+1}$) 
takes place at the vertex on the positive $y$-axis
along the nodal ray $v_n$ (resp.\ $w_n$).}
\label{fig:trianglen}
\end{figure}
\end{center}

The base case is immediate from Figure~\ref{fig:pregame}.
For the induction step, we must check first that the matrix $M_n$ is indeed performing the mutation from 
$\Delta_n$ to $\Delta_{n+1}$, that is: \begin{enumerate}
\item $M_n v_n=v_n$ , 
\item $M_nw_n=\begin{psmallmatrix} 0\\1 \end{psmallmatrix}$ , and
\item$\det(M_n)=1$.
\end{enumerate}
We must also check that this transformation gives rise to the new data of $\Delta_{n+1}$:
\begin{enumerate}  \setcounter{enumi}{3}
\item $w_{n+1}=M_n \begin{psmallmatrix} -1\\0 \end{psmallmatrix}$ ,
\item $v_{n+1}=M_n u_n$ ,
\item $u_{n+1}=-v_n$ ,
\item $a_{n+1}=a_n+c_n$ ,
\item $b_{n+1}=a_n\, \frac{1^\text{st} \text{ entry of } v_n }{2^\text{nd} \text{ entry of } v_n }$ , and
\item $c_{n+1}=b_n-b_{n+1}$.
\end{enumerate}

The proof of these uses the identities in Lemma \ref{lemma:identities}. Finally, we note that at each step,
the base diagram $\Delta_n$ is exactly the triangle with vertices $(0,0)$,  
$$
\textstyle{
(\frac{g(n)+g(n+2)}{g(n+2)},0)=(b_n,0), \mbox{ and }
(0,\frac{g(n+1)+g(n+3)}{g(n+1)})=(0,a_n),}
$$
which is what we wanted to prove.

Next we tackle the $J=3$ case. Here, the base diagram never becomes a triangle, instead it is always a quadrilateral. 
Following the notation in Figure~\ref{fig:trianglen}(b), 
we will repeatedly apply mutation at the vertex
on the positive $y$-axis, along the nodal
ray $w_n$ emanating from that vertex.
The formulas below give the relevant data of the base diagram $\square_n$:

$$u_n=\begin{psmallmatrix} -g(n)\\g(n+3) \end{psmallmatrix},\,\,\,
w_n=\begin{psmallmatrix} g(n+1)\\-g(n+4) \end{psmallmatrix},
$$

$$
s_n=\begin{psmallmatrix}\sigma_{n+1}g(n+1)^2\\ 1-\sigma_{n+1}g(n+1)g(n+4) \end{psmallmatrix},\,\,\,
r_n=\begin{psmallmatrix}\sigma_{n}g(n)g(n+3)-1\\ -\sigma_{n+3}g(n+3)^2 \end{psmallmatrix},$$

$$ a_n=\frac{g(n+1)+g(n+4)}{g(n+1)},\,\,\,b_n=\frac{g(n)+g(n+3)}{g(n+3)},$$

\noindent and

$$M_n=\begin{pmatrix} 1-\sigma_{n+1}\,g(n+1)g(n+4) &  -\sigma_{n+1}\,g(n+1)^2   \\ \sigma_{n+4}\,g(n+4)^2&1+\sigma_{n+1}\,g(n+1)g(n+4)\end{pmatrix},$$
where $\sigma_n$ is again as in Table~\ref{tab:sigma}. 

Performing a mutation on $\square_n$ uses the matrix $M_n$ and yields $\square_{n+1}$. The matrix $M_n$ satisfies
\begin{enumerate}
\item $M_n w_n=w_n$
\item $M_n r_n=\begin{psmallmatrix} 0\\1 \end{psmallmatrix}$ 
\item$\det(M_n)=1$
\end{enumerate}
and the data for $\square_{n+1}$ is obtained via
\begin{enumerate}  \setcounter{enumi}{3}
\item $r_{n+1}=M_n s_n$
\item $s_{n+1}=M_n \begin{psmallmatrix} -1\\0 \end{psmallmatrix}$
\item $v_{n+1}=M_n u_n$ and $w_{n+1}=M_n v_n$; thus $w_{n+1}=M_n M_{n-1}u_{n-1}$
\item $u_{n+1}=-w_n$
\item $a_{n+1}=a_n+d_n$
\item $d_{n+1}=c_n$
\item $b_{n+1}=-a_n\, \frac{1^\text{st} \text{ entry of } w_n }{2^\text{nd} \text{ entry of } w_n }$
\item $c_{n+1}=b_n-b_{n+1}$.
\end{enumerate}

The proof of these relations uses the identities in Lemma \ref{lemma:identities}. Finally, we note that the triangle with vertices $(0,0)$, 
$$
\textstyle{
(\frac{g(n)+g(n+3)}{g(n+3)},0)=(b_n,0), \mbox{ and } (0,\frac{g(n+1)+g(n+4)}{g(n+1)})=(0,a_n)
}
$$
fits in the base diagram $\square_n$ for
each $n$, which is what we wanted to prove.

The ATFs described above and Proposition~\ref{prop:triangleintothing} allow us to conclude that we have the 
desired  embeddings \eqref{eq:atf inner} with target the compact manifold $M$.
We now argue that there are also such embeddings with target a convex toric domain.
Suppose that $X$ is the convex toric domain with the same negative weight expansion as $M$.  By Theorem~\ref{prop:mainproposition},
we then have an embedding
$$
(1-\varepsilon)\cdot E\left(\frac{g(n)+g(n+J)}{g(n+J)}, \frac{g(n+1)+g(n+1+J)}{g(n+1)} \right)\embeds X
$$
for every $0<\varepsilon <1$.  Now note that
$$
(1-\varepsilon)\cdot\overline{E(a,b)} \bigsubset[2] \left(1-\frac{\varepsilon}{2}\right)\cdot E(a,b),
$$
so we may conclude that we have symplectic embeddings of the closed ellipsoids
$$
(1-\varepsilon)\cdot \overline{E\left(\frac{g(n)+g(n+J)}{g(n+J)}, \frac{g(n+1)+g(n+1+J)}{g(n+1)} \right)}\embeds X
$$
for every $0<\varepsilon <1$. We may now apply \cite[Cor. 1.6]{dan} to deduce
that there is a symplectic embedding of the form  \eqref{eq:ctd inner}, as desired. 
In fact, we may also apply Theorem~\ref{prop:mainproposition}  one more time to deduce that
$$
E\left(\frac{g(n)+g(n+J)}{g(n+J)}, \frac{g(n+1)+g(n+1+J)}{g(n+1)} \right)\embeds M.
$$
Thus we have shown that there are ellipsoid embeddings representing the purported interior corners.
\end{proof}

\noindent We now include two specific worked examples to illustrate how the ATF recursions operate.

\begin{example}\label{ex:worked1}
We first consider $M=\C P^2_3$, which should
produce the first steps of the Fibonacci staircase.  Using Table~\ref{table:recurrence} with $J=2$, we have
$g(n)=2,1,1,2,5,13$ for the terms $n=0,\dots,5$.  After initial terms, these
are the odd-index Fibonacci numbers. 
The first two mutations produce the sequence of triangles 
shown in Figure~\ref{fig:initCP2}.  We don't include all decorations
and labels for space considerations.  We do include the lengths of
the legs of each triangles, as these determine the proportions of the ellipsoids
that we can thus embed into $\C P^2$ and which show up at the first three interior corners of the staircase: $E(1,1)$, $E(1,4)$, and $E(1,\frac{25}{4})$.

\newpage 

\begin{center}
  \begin{figure}[ht]
\begin{overpic}[
scale=0.38,unit=1mm]{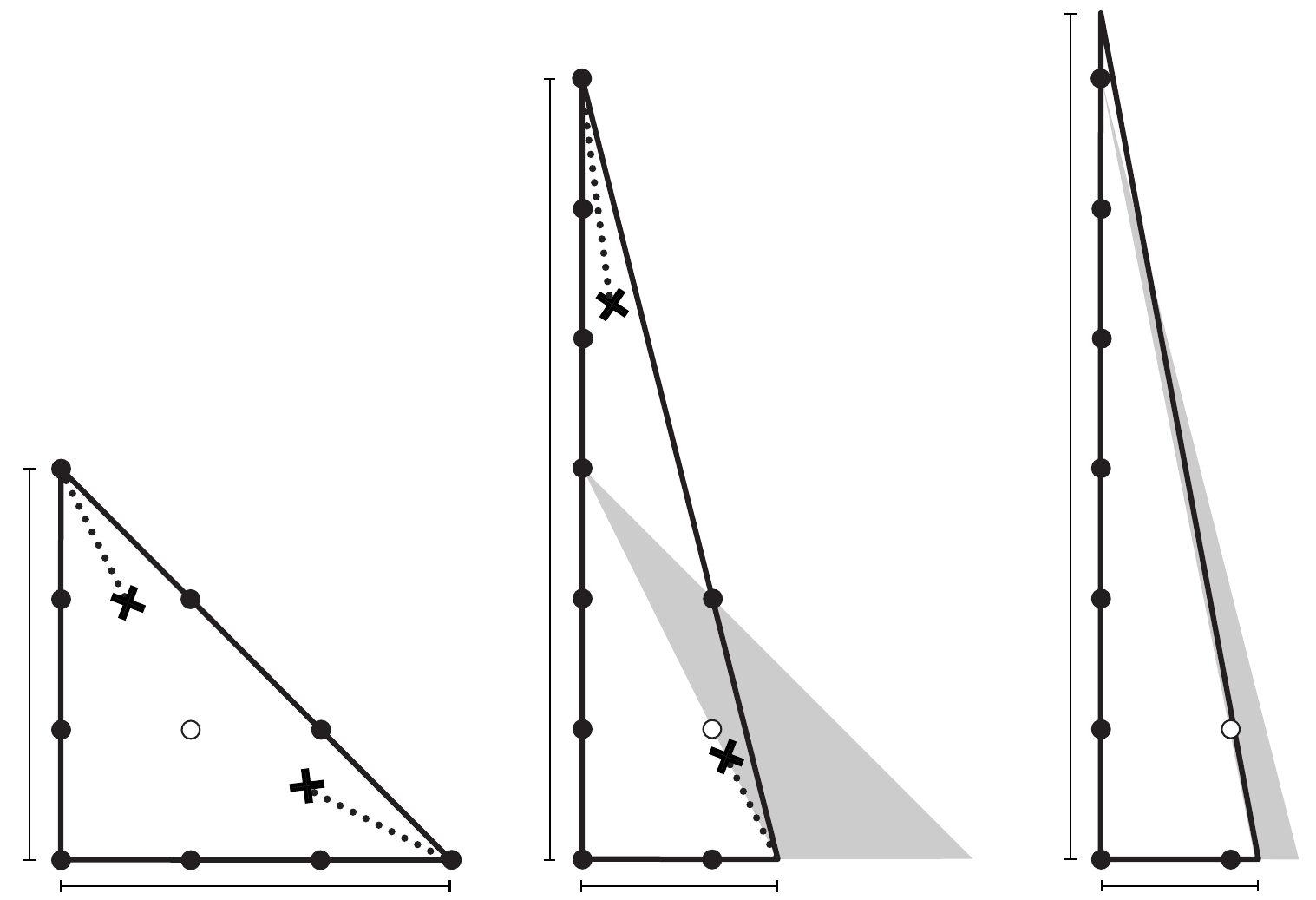}
   \put(-10,16){\small{$a_0=3$}}
   \put(14,-2.5){\small{$b_0=3$}}
   \put(29,33){\small{$a_1=6$}}
   \put(46,-3){\small{$b_1=\frac{3}{2}$}}
   \put(66,39){\small{$a_2=\frac{15}{2}$}}
   \put(82,-3){\small{$b_2=\frac{6}{5}$}}
   \put(18,-9){(a)}
   \put(50,-9){(b)}
   \put(88,-9){(c)}
\end{overpic} \vskip 0.35in
\caption{The first two mutations in the recursive procedure for $\C P^2_3$, starting in (a) from the initial ATF indicated in Figure~\ref{fig:pregame} for $\C P^2_3$.
From (a) to (b), we apply mutation matrix 
$\left(\begin{smallmatrix}
-1&-1\\ \phantom{-}4&\phantom{-}3
  \end{smallmatrix}\right)$ and from (b) to (c), mutation matrix $\left(\begin{smallmatrix}
-4&-1\\25&\phantom{-}6
  \end{smallmatrix}\right)$.
Note that the light gray portion in (b) is the portion of the triangle that has changed in the mutation from (a). That gray portion is not part of the new base diagram,
which is outlined in black.  Likewise, the light gray portion in
(c) is what has changed in mutation from (b).
We do not include nodal rays in (c) because their direction is
too close to the hypotenuse to be visible at this resolution.  Note also that the white lattice point is
in the interior of the polygon. }
\label{fig:initCP2}
\end{figure}
\end{center}
\end{example}

\begin{example}\label{ex:worked2}
We consider $M=\C P^2_3\# \overline{\C P}^2_1$.  Using Table~\ref{table:recurrence} with $J=3$, we have
$g(n)=2,1,1,2,5,13$ for $n=0,\dots,5$.
The first two mutations produce the sequence of quadrilaterals shown in Figure~\ref{fig:initHIRZ}.  We don't include all decorations
and labels for space considerations.  We do include the lengths of the sides sitting on the axes, as these determine the proportions of the ellipsoids
that we can thus embed into $\C P^2\# \overline{\C P}^2$ and which show up at the first three interior corners of the staircase: $E(1,\frac32)$, $E(1,\frac{10}{3})$, and $E(1,\frac{24}{5})$.

\newpage 

\begin{center}
  \begin{figure}[ht]
\begin{overpic}[
scale=0.44,unit=1mm]{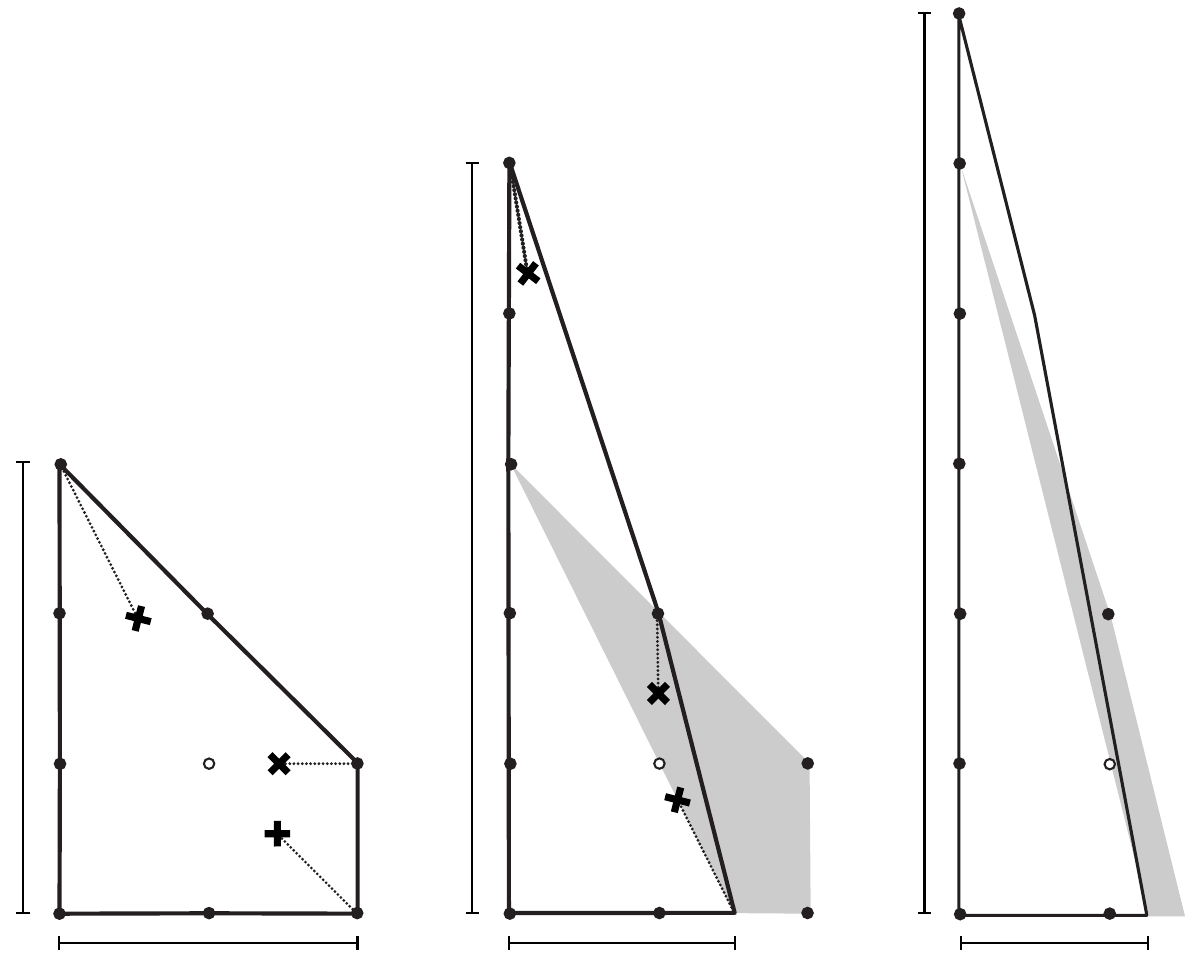}
   \put(-9,20){\small{$a_0=3$}}
   \put(13,-2.5){\small{$b_0=2$}}
   \put(24,30){\small{$a_1=5$}}
   \put(42,-3){\small{$b_1=\frac{3}{2}$}}
   \put(58,38){\small{$a_2=6$}}
   \put(74,-3){\small{$b_2=\frac{5}{4}$}}
   \put(13,-9){(a)}
   \put(46,-9){(b)}
   \put(78,-9){(c)}
\end{overpic} \vskip 0.3in
\caption{The first two mutations in the recursive procedure for 
$\C P^2_3\# \overline{\C P}^2_1$, starting in (a) from the initial ATF 
indicated in Figure~\ref{fig:pregame} for $\C P^2_3\# \overline{\C P}^2_1$.
From (a) to (b), we apply mutation matrix 
$\left(\begin{smallmatrix}
-1&-1\\ \phantom{-}4&\phantom{-}3
  \end{smallmatrix}\right)$ and from (b) to (c), mutation matrix $\left(\begin{smallmatrix}
-3&-1\\16&\phantom{-}5
  \end{smallmatrix}\right)$.
Note that the light gray portion in (b) is the portion of the triangle that has changed in the mutation from (a). That gray portion is not part of the new base diagram,
which is outlined in black.  Likewise, the light gray portion in
(c) is what has changed in mutation from (b).
We do not include nodal rays in (c) because their directions are
too close to the edges of the quadrilateral to be visible at this resolution.  Note also that the white lattice point is in the interior of the polygon. }
\label{fig:initHIRZ}
\end{figure}
\end{center}
\end{example}

We now have all of the ingredients in place to complete the proof that the Fano infinite staircases exist.

\begin{proof}[Proof of Theorem \ref{thm:main}]
We know from Propositions \ref{prop:outer} and \ref{prop:inner} that $y^\outter_n\leq c_X(x^\outter_n)$ and $c_X(x^\inner_n)\leq y^\inner_n$. 

Now we use Proposition \ref{prop:3things}. Since $x^\outter_n < x^\inner_n$ and $y^\outter_n = y^\inner_n$, because $c_X(a)$ is continuous and non-decreasing, it must be constant and equal to $ y^\outter_n$ between $x^\outter_n$ and $x^\inner_n$. Furthermore, since $x^\inner_n<x^\outter_{n+1}$ and the points $(0,0)$,  $(x^\inner_n,y^\inner_n)$ and $(x^\outter_{n+1},y^\outter_{n+1})$ are colinear, the scaling property of $c_X(a)$ implies that between each $x^\inner_n$ and $x^\outter_{n+1}$, the graph of $c_X(a)$ consists of a straight line segment (which extends through the origin). We thus have an infinite staircase in each of the cases studied.

Finally, by continuity and because  $x^\outter_n\to a_0$, $x^\inner_n\to a_0$, and $y^\outter_n=y^\inner_n\to\sqrt{\frac{a_0}{\vol}}$ as $n\to\infty$, we know that the infinite staircase accumulates from the left at
$\left(a_0,c_X(a_0)=\sqrt{\frac{a_0}{\vol}}\right)$, which completes the proof of the theorem.
\end{proof}

\begin{remark}
One must take care to interpret the base diagrams in Figure~\ref{fig:pregame} correctly.
These represent almost toric fibrations on smooth manifolds, not moment map images of toric orbifolds.
\end{remark}

\begin{remark}\label{rmk:the other four}
In this section we have have shown that the toric domains corresponding to the twelve reflexive polygons in Figure \ref{fig:twelve} have infinite staircases. However, there exist exactly sixteen reflexive polygons up to $AGL_2(\Z)$ equivalence: the twelve in 
Figure~\ref{fig:twelve} plus the four in Figure~\ref{fig:the others}, which have reduced blowup vectors 
$(3;1,1,1,1,1)$ and $(3;1,1,1,1,1,1)$.
 
Solving the quadratic equation \eqref{eqn:theequation} for these two negative weight expansions, we obtain that in the first case $a_0 = 1$ and in the second the quadratic equation has no real solutions at all.  By Theorem \ref{satisfies quadratic equation}, the latter case therefore cannot have an infinite staircase.  As for the former case, a straightforward modification of the method in \cite[\S 2.5]{cg ellipsoid} can be used to show that $c_X(a) = \frac{a+1}{4}$ for $a \ge 1$ sufficiently close to $1$.\footnote{\phantom{.}We omit the modified argument for brevity but include here one necessary ingredient, which is the cap function of the target. To rigorously compute it one uses the sequence subtraction formula \eqref{def:seqsum}: the sequence of ECH capacities of X is given by $S_1-S_2$, where $S_1$ is the sequence of ECH capacities of $E(3,3)$ and $S_2$ is the sequence of ECH capacities of $E(1,5)$. The result then simplifies to $\text{cap}_X(T)=\frac{T^2}{8}+\frac{T}{2}+\Gamma_r$ where $T\equiv r \hspace{1 mm} (\mbox{mod }4)$ and $\Gamma_0=1, \Gamma_1=\frac{3}{8},  \Gamma_2=\frac{1}{2}$ and $\Gamma_3=\frac{3}{8}$. The definition of the $\text{cap}_X$ function and a discussion of its significance can be found in \eqref{quasipolynomial} and at the beginning of Section \ref{sec:conj}.}   Thus, infinite staircases do not exist in either of these cases.

\begin{figure}[h!]
\centering
\includegraphics[width=0.75\textwidth]{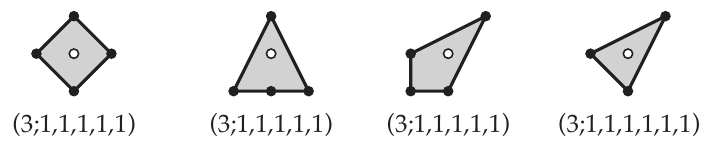}
\caption{The remaining four reflexive polygons.}
\label{fig:the others}
\end{figure}
\end{remark}

\section{Conjecture: why these may be the ``only'' infinite staircases}\label{sec:conj}

In this section we describe some evidence towards Conjecture \ref{conj:reflex}, which
speculates that the only rational convex toric domains that admit an infinite staircase are those whose moment polygon, 
up to scaling, is $AGL_2(\Z)$-equivalent to one in Figure~\ref{fig:twelve}.
In light of \cite{project9,usher}  about infinite staircases for irrational targets,  
it is crucial in the conjecture that  the toric domain be rational.

Let $X_\Omega$ be a rational convex toric domain with negative weight expansion $(b;b_1,b_2,\ldots,b_n)$. The 
ellipsoid embedding function of the scaling of a convex toric domain is a scaling of the ellipsoid embedding function 
of the original domain.  Thus, we may assume that the negative weight expansion of $X_\Omega$ is integral
and primitive: $\gcd(b,b_1,\dots,b_n)=1$. With this scaling, we say that $\Omega$ is {\bf primitive}.

By Theorem \ref{satisfies quadratic equation}, if the ellipsoid embedding function of $X_\Omega$ has an infinite staircase, then $c_{X_\Omega}(a_0)=\sqrt{ \frac{a_0}{\vol} }$. This implies that
$$E(1,a_0)\embeds \sqrt{\frac{a_0}{\vol}} X_\Omega,$$
which by Proposition \ref{prop:embedding} and conformality of ECH capacities is equivalent to an inequality of sequences of ECH capacities:

\begin{equation}\label{eq:ineqECH}
c_{ECH}\left(E\left(\sqrt{\frac{\vol}{a_0}},\sqrt{a_0{\vol}}\right)\right)\leq c_{ECH}(X_\Omega).
\end{equation}
To rewrite this inequality we introduce the \textbf{cap function} of a convex toric domain $X$, for $T\in\mathbb{N}$: 
\begin{equation}\label{quasipolynomial}
\text{cap}_{X}(T) \eqdef \# \lbrace k : c_k(X) \le T \rbrace.
\end{equation}
With $u=\sqrt{\frac{\vol}{a_0}}$ and $ v=\sqrt{a_0\vol}$, the inequality \eqref{eq:ineqECH} is equivalent to:
\begin{equation}\label{eq:ineqErhart}
\text{cap}_{E(u,v)}(T) \ge \text{cap}_{X_\Omega}(T), \mbox{ for all } T\in\mathbb{N}.
\end{equation}

We first look at the right hand side of inequality \eqref{eq:ineqErhart}.
For rational convex toric domains $X_\Omega$ with $\Omega$ primitive, it is possible to argue that $\text{cap}_{X_\Omega}(T)$
has quadratic term $\frac{1}{2\vol}T^2$ and linear term $\frac{\per}{2 \vol}T$. In \cite[Theorem 8]{wormleighton}, Wormleighton 
analyzes the constant term of $\text{cap}_{X_\Omega}(T)$ and shows that since $\Omega$ is a primitive lattice convex toric domain, the cap function of $X_\Omega$ is eventually equal to a quasipolynomial:
\begin{equation}\label{eq:capfcn}
\text{cap}_{X_\Omega}(T)=\frac{1}{2\vol}T^2+\frac{\per}{2 \vol}T+\Gamma_r,
\end{equation}
where $r\in\{0,\ldots,\vol-1\}$ is a congruence class of $T$ (mod $\vol$) and each $\Gamma_r$ is a constant.

Next, we look at the left hand side of inequality \eqref{eq:ineqErhart}.  For a general ellipsoid $E(u,v)$, the cap function equals the Ehrhart function of the triangle $\Delta_{u,v}$, with vertices $(0,0), (1/u,0),$ and $(0,1/v)$:
\begin{eqnarray*}
\text{cap}_{E(u,v)}(T) & = & \# \lbrace k : c_k(E(u,v)) \le T \rbrace)\\
& =& \# \lbrace k : N(u,v)_k \le T \rbrace\\
& =& \# \left( T\Delta_{u,v} \cap \mathbb{Z}^2 \right)\\
& =& \text{ehr}_{\Delta_{u,v}}(T).
\end{eqnarray*}

In the particular case of the ellipsoid in \eqref{eq:ineqErhart}, it is usually the case that $a_0$ is irrational; since we are just explaining heuristic evidence, we assume this for what follows. 

\begin{assumption}\label{ass:irrational}
The accumulation point $a_0$ is irrational.
\end{assumption}
\noindent The case when $a_0$ is rational must be treated separately, and we will not concern ourselves with it here.
With Assumption~\ref{ass:irrational} in place, 
we have that the ratio $\frac v u=a_0$ is irrational, and 
thus we satisfy the conditions of \cite[Lemma 2.1]{cgls}, which allows us to write:
\begin{equation*}\text{cap}_{E(u,v)}(T)=\text{ehr}_{\Delta(u,v)}(T)=\frac{1}{2 u v}T^2+\frac12\left(\frac1u+\frac1v\right)T+d(T),\end{equation*}
where $d(T)$ is asymptotically $o(T)$. It is straightforward to see that
$$\frac{1}{uv}=\frac{1}{\vol}$$
and since $a_0$ satisfies the quadratic equation \eqref{eqn:theequation} we can get
\begin{equation}\label{eq:perovervol}\frac{1}{u}+\frac{1}{v}=\sqrt{\frac{a_0}{\vol}}+\sqrt{\frac{1}{a_0\vol}}=\frac{\per}{\vol}.\end{equation}
Thus
\begin{equation}\label{eq:captriangle}\text{cap}_{E(u,v)}(T)=\frac{1}{2 \vol}T^2+\frac{\per}{2 \vol}T+d(T).\end{equation}

Now, modulo Assumption~\ref{ass:irrational}, if the ellipsoid embedding function of $X_\Omega$ has an infinite staircase, then \eqref{eq:ineqErhart} forces
\begin{equation}\label{eq:shortineqErhart}
d(T)\geq \Gamma_r, \mbox{ for all } T\in\mathbb{N}.
\end{equation}
The converse would  
allow us to rule out the existence of infinite staircases 
by finding one $T$ for which
\begin{equation}
\label{eqn:keyequation} 
d(T)< \Gamma_r.
\end{equation}
Our approach for finding such a $T$ is partially inspired by influential work of Hardy and Littlewood  \cite{HL, HLanother}.

We say that a number $x$ has {\bf continued fraction expansion}
$[a_0,a_1,a_2,\ldots]$ if
$$
x = a_0+ \frac{1}{a_1+\frac{1}{a_2+ \frac{1}{\ddots}}}.
$$
The numbers $a_i$ are called the {\bf partial quotients} of $x$.
Hardy and Littlewood  studied the function $\text{ehr}_{\Delta_{u,v}}(T)$, as $T$ varies 
among the positive real numbers.  
They show in \cite[Theorems A3 and A4]{HL} 
that if $\frac{u}{v}$ has bounded partial quotients, then the term $d(T)$ is $\mathcal{O}(\log(T))$.
Moreover, they show that this is optimal in the sense that there is a positive constant $K$
and examples $\frac{u}{v}$ with bounded partial quotients that satisfy     
\begin{equation}
\label{eqn:hardylittlewood}
d(T_i) > K\cdot \log(T_i)\quad \mbox{ and } \quad d(T_j) < -K\cdot \log(T_j),
\end{equation}
where the $T_i$ and $T_j$ are two increasing sequences of real numbers.
In particular, this implies that one cannot replace $\log(T)$ by any function with
strictly smaller growth.
We will say
that a function is {\bf optimally} $\pm \mathcal{O}(\log(T))$ to signify that its behavior is
$\mathcal{O}(\log(T))$ and it satisfies \eqref{eqn:hardylittlewood}.
In particular, if the $T_j$ could be assumed to be integers in \eqref{eqn:hardylittlewood}, then 
\eqref{eqn:keyequation} would eventually hold, since the $\Gamma_r$ are bounded.  
This would  
rule out infinite staircases for all rational convex toric domains,
underscoring the subtlety of this argument.

In \cite{cgls}, Cristofaro-Gardiner, Li, and Stanley studied the function  $\text{ehr}_{\Delta_{u,v}}(T)$ 
for integer values of $T$.  They show that this function can sometimes be a quasipolynomial when 
restricted to integral $T$, even when $\frac{u}{v}$ is irrational, making both sides of \eqref{eq:ineqErhart}
quasipolynomial.  Being a quasipolynomial over $\Z$ 
implies that in contrast to \eqref{eqn:hardylittlewood} above, $d(T)$ only takes on finitely many values.
They classify all cases where the Ehrhart function is a quasipolynomial as follows.

\begin{theorem}\cite[Theorem 1(i)]{cgls}
\label{thm:cglis}
Assume that $\frac{u}{v}$ is irrational.  Then $\text{ehr}_{\Delta_{u,v}}(T)$ is a quasipolynomial if and only if
\begin{equation}
\label{eqn:theconditions}
u+v = \alpha \quad \mbox{ and } \quad \frac{1}{u} + \frac{1}{v} = \beta  \quad \mbox{ for } \quad \beta\ , \alpha\beta\ \in \mathbb{N}.
\end{equation}
\end{theorem}   

\begin{remark}
In our context, following \eqref{eqn:perimetersequal} and \eqref{eq:perovervol}, we have 
$\alpha = \per$ and $\beta = \frac{\per}{\vol}$.
\end{remark}

In view of  \eqref{eqn:hardylittlewood} and Theorem~\ref{thm:cglis}, we conjecture the following.

\begin{conjecture}
\label{conj:theconjecture}
Let $\frac{u}{v}$ be a quadratic surd \footnote{\phantom{.}
To keep a strict analogy with  \cite[Theorems A3 and A4]{HL}, 
we should only require that $\frac{u}{v}$ has bounded partial quotients.  We have included the stronger 
hypothesis that $\frac{u}{v}$ is a quadratic surd because we can assume it in view of 
Theorem~\ref{satisfies quadratic equation}. It is plausible that it could help with the proof, since it implies 
that the continued fraction of $\frac{u}{v}$ is periodic.
} 
such that  $\frac{u}{v}+\frac{v}{u}\in\mathbb{Q}$ \footnote{\phantom{.}
The rationality
of $\frac{u}{v}+\frac{v}{u}$ for rational convex toric domains is a consequence of $\frac{u}{v}+\frac{v}{u}=\frac{\per^2}{\vol}-2$,
which is a rational function of the integral negative weight expansion.
}.
If  \eqref{eqn:theconditions} does not hold, then  $d(T)$ is optimally $\pm\mathcal{O}(\log(T))$
where the $T_i$ and $T_j$ are sequences of positive integers.
\end{conjecture}
  
We will say more about why Conjecture~\ref{conj:theconjecture} might hold below and now explain why
Conjecture ~\ref{conj:theconjecture} and Assumption~\ref{ass:irrational} imply Conjecture~\ref{conj:reflex}.
Under these assumptions, if the ellipsoid embedding 
function of $X_\Omega$ has an infinite staircase 
then $\frac{\per}{\vol},\frac{\per^2}{\vol}$ are in $\mathbb{N}$.
Consider now the convex toric domain $X_{\widetilde{\Omega}}$ corresponding to the scaled region $\widetilde{\Omega}=\frac{\per}{\vol} \Omega$.
Because $\frac{\text{per}}{\text{vol}}\in\mathbb{N}$, the region  $\widetilde{\Omega}$  is still a lattice polygon, with $\widetilde{\vol}$ and $\widetilde{\per}$ 
satisfying
$$\text{area} = \frac{\widetilde{\vol}}{2} = \frac{\per^2}{\vol}$$
and
$$\# \text{ of boundary lattice points} = \widetilde{\per}= \frac{\per^2}{\vol}. $$
\noindent Pick's Theorem states that
$$\text{area}=\# \text{ of interior lattice pts}+\frac{\# \text{ of boundary lattice pts}}{2}-1$$
and so we conclude that $\widetilde{\Omega}$ has exactly one interior lattice point: that is, it is a reflexive polygon.
This completes our argument towards Conjecture~\ref{conj:reflex}.

We close this section with some further exploration of Conjecture~\ref{conj:theconjecture}, showing how
in general it follows from Conjecture~\ref{conj:numbertheory}.
The general case is when  $\alpha \beta=\frac{u}{v} + \frac{v}{u}+2\in\mathbb{Q}$ is not an integer\footnote{\phantom{.}
By contrast, when $\Omega$ is a (scaling of a) reflexive polygon, we have that $\alpha\beta = \frac{\per^2}{\vol}$ is an integer.
The special case of Conjecture~\ref{conj:theconjecture} when $\alpha\beta = \frac{\per^2}{\vol}$ is an integer and 
$\Omega$ is not a reflexive polygon must be handled separately.
}, 
for $\alpha$ and $\beta$ defined by \eqref{eqn:theconditions}.   In this context, $\alpha\beta = \frac{\per^2}{\vol}$,
which is an invariant associated to the shape of the moment polygon $\Omega$ of the convex toric domain (up to scalings; 
see Remark~\ref{rem:invt}).
We start by examining $d(T)$.  Following  
the analysis in Cristofaro-Gardiner, Li, and Stanley of Ehrhart functions for triangles, 
by \cite[Eq. 2.1]{cgls}, $\text{ehr}_{\Delta_{u,v}}(T)$ can also be written as 
$$
\text{ehr}_{\Delta_{u,v}}(T)=\sum_{m=0}^{\lfloor{T/\per}\rfloor}\left(1+ \left\lfloor\frac{T-m\cdot \per}{v} \right\rfloor    +
\left\lfloor\frac{T-m\cdot\per}{u}\right\rfloor  \right).$$
In what follows, we denote by $\{x\}=x-\floor{x}$ the fractional part of a real number $x$. Let $T=n\cdot\per$ and rewrite $\text{ehr}_{\Delta_{u,v}}(T)$ as 
$$
\begin{array}{lr}
\text{ehr}_{\Delta_{u,v}}(n\cdot\per)  &\\
\phantom{BOO} =\sum^n_{k=0} \left(1+\left\lfloor\frac{k\cdot\per}{v}\right\rfloor +\left\lfloor\frac{k\cdot\per}{u}\right\rfloor  \right) & \text{ \footnotesize (1)}\\
\phantom{BOO} =2\sum^n_{k=0}  k + \sum^n_{k=0}\left(1+\left\lfloor\frac{k u}{v}\right\rfloor +\left\lfloor\frac{k v}{u}\right\rfloor  \right) & \text{\footnotesize (2)}\\
\phantom{BOO} =\left(2+\frac{u}{v}+\frac{v}{u} \right)\sum^n_{k=0}  k - \sum^n_{k=0} \left(-1+\left\{\frac{k u}{v}\right\} +\left\{\frac{k v}{u}\right\}  \right)
& \text{   {\footnotesize (3)}}\\
\phantom{BOO} =\frac{\per^2}{\vol}
\sum^n_{k=0}  k - \sum^n_{k=0} \left(-1+\left\{\frac{k u}{v}\right\} +\left\{\frac{k v}{u}\right\}  \right)
& \text{   {\footnotesize (4)}}\\
\phantom{BOO} =\frac{\per^2}{\vol}\frac{\left(\frac{T}{\per}\right)\left(\frac{T}{\per}+1\right)}{2}- \sum^n_{k=0} \left(\left(\left\{\frac{k u}{v}\right\}-\frac{1}{2}\right) +\left(\left\{\frac{k v}{u}\right\} -\frac{1}{2}\right) \right) &
\\
\phantom{BOO} =\frac{1}{2\vol}T^2+\frac{\per}{2\vol}T- \sum^n_{k=0} \left(\left(\left\{\frac{k u}{v}\right\}-\frac{1}{2}\right) +\left(\left\{\frac{k v}{u}\right\} -\frac{1}{2}\right) \right),&\\
\end{array}
$$
where at line (1), we reindex; at line (2) we use $\per=u+v$; at line (3), we note that $\floor{x}=x-\{x\}$; and at line (4), we use $\frac{u}{v}+\frac{v}{u}=\frac{\per^2}{\vol}-2$.

For $n\in\mathbb{N}$ and $\theta\in\mathbb{R}\setminus\mathbb{Q}$, we define as in \cite{bs} the function
$$
C_\theta(n):=\sum_{k=0}^n\left(\left\{ k\theta \right\}-\frac{1}{2}\right).
$$
Hardy and Littlewood established that the function $C_\theta(n)$ is optimally $\pm\mathcal{O}(\log(n))$ \cite[Theorem 9]{HLanother}.
This function has a long history and has been studied by 
Sylvester \cite{sylvester}, Lerch \cite{lerch}, 
Sierpi\'nski \cite{sierp}, Hardy and Littlewood \cite{HL,HLanother} (who highlighted its properties in their 1912 ICM address), 
Ostrowski \cite{ostr}, S\'os \cite{sos}, and Brown and Shiue \cite{bs}.

Comparing the final form of $\text{ehr}_{\Delta_{u,v}}(n\cdot\per)$ just above with  \eqref{eq:captriangle}, 
we conclude that for $T=n\cdot\per$ we have
\begin{equation}\label{eq:d(T)}d(T)=d(n\cdot\per)=- \left(C_{a_0}\left(n\right)+C_{1/a_0}\left(n\right) \right).\end{equation}
We conclude that, in the case when $\alpha\beta=a_0+\frac1{a_0}$ is not an integer, 
Conjecture~\ref{conj:theconjecture}  would now be a consequence of the following.

\begin{conjecture}\label{conj:numbertheory}
Let $a_0$ 
be a quadratic surd with $a_0+\frac1{a_0}\in\mathbb{Q}$. 
Then the function $C_{a_0}\left(n\right)+C_{1/a_0}\left(n\right)$ is bounded if and only if 
$a_0+\frac1{a_0}\in\mathbb{N}$.
In particular, when it is bounded, it is identically zero.
If it is unbounded, 
then it is unbounded above and below; and more specifically, it is optimally $\pm \mathcal{O}(\log(n))$.
\end{conjecture}

\begin{figure}[h]
\subfloat[]
{
\includegraphics[width=0.31\textwidth]{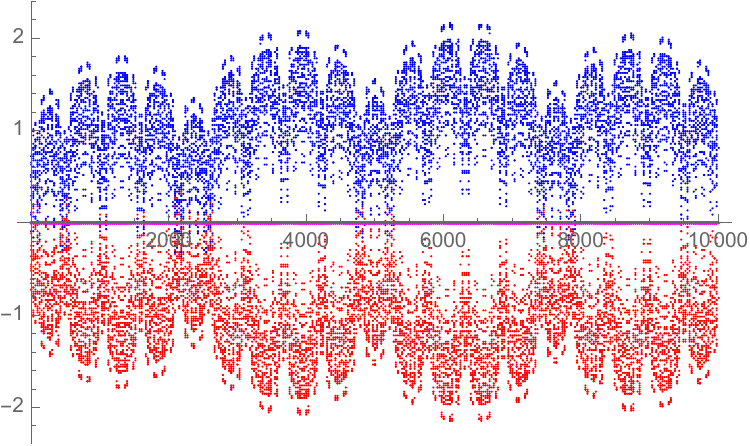} 
}
\subfloat[]
{
\includegraphics[width=0.29\textwidth]{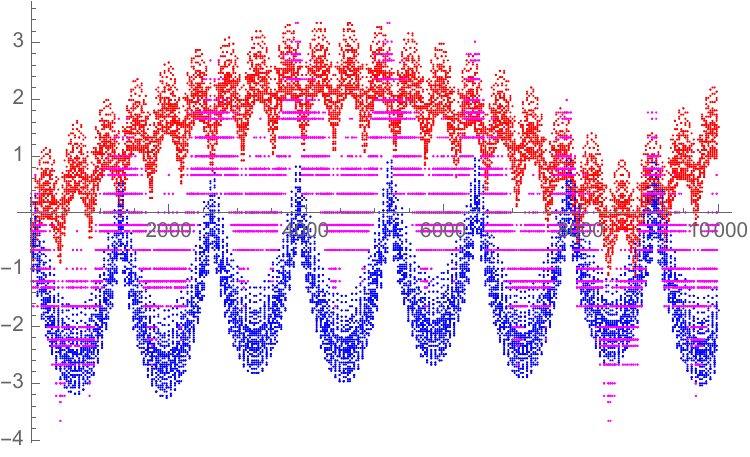} 
}
\subfloat[]
{
\includegraphics[width=0.31\textwidth]{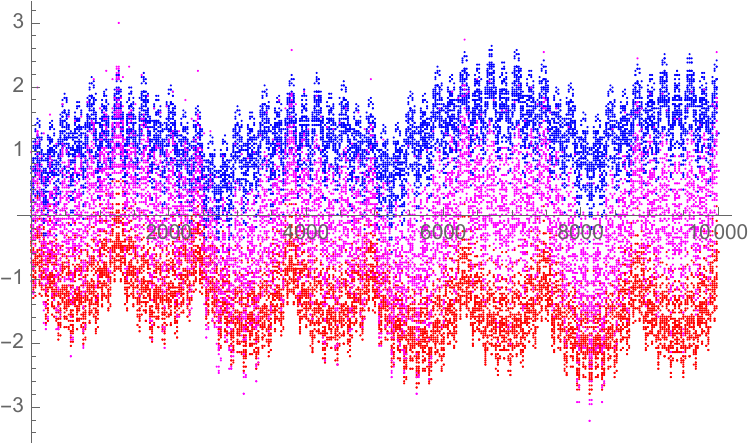}
}
\caption{Each figure shows the function {\color{blue} $C_{a_0}(n)$ in blue}, the function {\color{red} $C_{1/a_0}(n)$ in red},
and their sum  {\color{magenta}$C_{a_0}(n)+C_{1/a_0}(n)$ in magenta}.  In (a), $a_0+\frac{1}{a_0}$ is 5; in
(b) it is $\frac{61}{9}$; and in (c) it is $\frac{6+8\sqrt{3}}{3}$.
 }\label{fig:redblue}
\end{figure}

Figure \ref{fig:redblue}(a) and (b) shows some experimental evidence towards this conjecture.  
In (a), we have that $a_0+\frac{1}{a_0}$ is a natural number. In this case,
the sum $C_{a_0}(n)+C_{1/a_0}(n)$ is identically zero on the domain computed.  
In (b),  we have that  $a_0+\frac{1}{a_0}=\frac{p}{q}$ is not a natural number. Here,
the sum $C_{a_0}(n)+C_{1/a_0}(n)$ appears to be unbounded above and below. 
It takes on discrete values because it is always a multiple of $\frac{1}{q}$. 

We include Figure \ref{fig:redblue}(c) for contrast: here $a_0+\frac{1}{a_0}$ is irrational, 
more specifically, $a_0$ is the accumulation point $S_{2,0}$ for one of Usher's infinite staircases \cite{usher}. This
example is not in the scope of this manuscript because it does not correspond to a rational convex toric domain.
The sum  $C_{a_0}(n)+C_{1/a_0}(n)$ still appears to be unbounded above and below.  

Sarnak has suggested two possible approaches for resolving Conjecture~\ref{conj:numbertheory}:
an  elementary approach involving continued fractions, along the lines of \cite{bs, sos}; or
a more sophisticated approach  building on results of Hecke's \cite{hecke}.  This seems to us a good starting point for future work on this conjecture.

\newpage

\appendix
\section{Lattice Paths}\label{ap:convex lattice paths}

\noindent In this appendix, we compile the combinatorial data we need for Proposition~\ref{prop:outer}.

\begin{figure}[htbp]
\centering
\includegraphics[width=\textwidth]{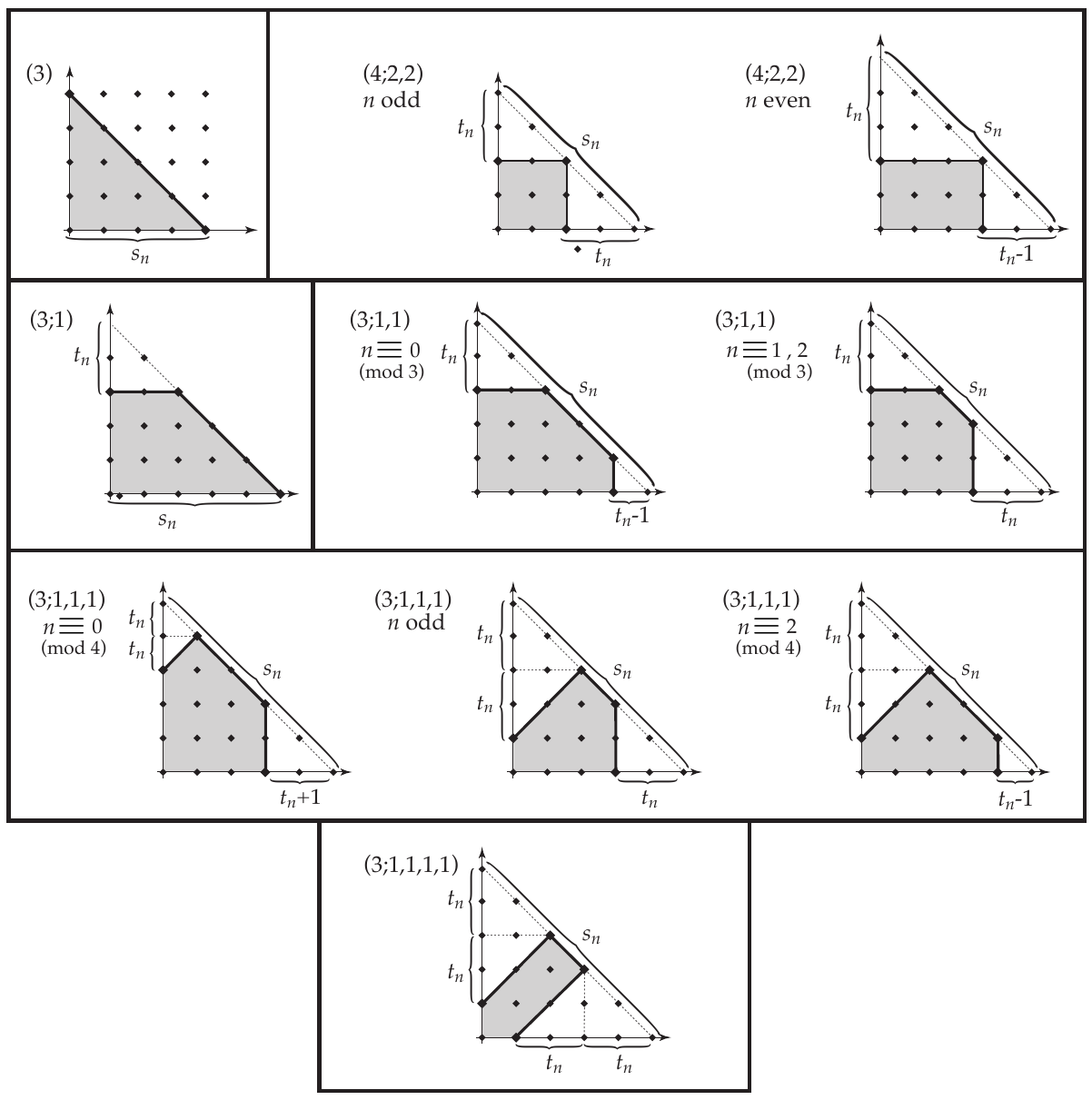}
\caption{The lattices paths $\Lambda_n$.  All triangles drawn with size $s_n$ or $t_n$ on one of their sides are right-angle equilateral triangles (in the ``length" that counts lattice points).}
\label{fig:paths}
\end{figure}

The sequences $s_n$ and $t_n$ that determine the paths $\Lambda_n$ 
have formul\ae\ in terms of their negative weight expansions.  In all six cases, the negative weight expansion is of the form
$(B;b,\dots,b)$.  In terms of those $B$ and $b$, we have
\begin{equation}\label{eq:sn and tn}
\textstyle{s_n = \frac{B\cdot (g(n)+g(n+J)) + c_n}{\vol} \mbox{ and } t_n = \frac{b\cdot (g(n)+g(n+J)) + d_n}{\vol},}
\end{equation}
where $c_n$ and $d_n$ are given in Table~\ref{tab:cn dn}.
We also have, in terms of $B$, $b$, and $k$ the number of $b$s,
\begin{equation}\label{eq:blowup Omega length}
\ell_{\Omega}(\Lambda_n) = B\cdot s_n + kb \cdot t_n + e_n.
\end{equation}
The $e_n$ are given explicitly in Table~\ref{tab:cn dn} and implicitly in Table~\ref{tab:calL ell}.

To prove Proposition~\ref{prop:outer}, these quantities  $\mathcal{L}(\Lambda_n)$ and  $\ell_\Omega(\Lambda_n)$ must
satisfy \eqref{eq:points and length}.  Using the definitions of $s_n$ and $t_n$, as well as the properties $\heartsuit$ and $\clubsuit$ of Lemma~\ref{lemma:identities}, one can
argue directly that $\mathcal{L}(\Lambda_n)= \frac{(g(n)+1)(g(n+J)+1)}{2}$.  We also have
\begin{eqnarray*}
\ell_\Omega(\Lambda) & = & B\cdot s_n + kb \cdot t_n + e_n \\
 & = & \frac{(B^2-k b^2)(g(n)+g(n+J))+Bc_n-bkd_n}{\vol}+e_n\\
 & = & g(n)+g(n+J) +\frac{Bc_n-bkd_n +\vol\cdot e_n}{\vol}.
\end{eqnarray*}
Checking that $Bc_n-kbd_n+\vol\cdot  e_n = 0$ for all cases and all moduli of $n$ then guarantees $\ell_\Omega(\Lambda)=g(n)+g(n+J)$, as desired.

\renewcommand{\arraystretch}{1.2}

\begin{center}
{
\begin{table}[h]
\begin{tabular}{| c ||  c | c |  c | c |}
\hline
$(B;b,\dots,b)$            & $\vol$          & $c_n$ & $d_n$ & $e_n$  \\ \hline \hline
$(3)$          & $9$                   & $0$ & $\nexists d_n$ & 0   \\  \hline
$(4;2,2)$        & $8$                     &  $0$ &  $\begin{array}{ll}
4 & n\mbox{ even}\\
0 & n\mbox{ odd}\end{array}$  & $\begin{array}{ll}
2 & n\mbox{ even}\\
0 & n\mbox{ odd}\end{array}$ \\  \hline
$(3;1)$      & $8$                       & 
$\begin{array}{rl}
& \mbox{mod } 6 \\
2  & n\equiv 0  \\
-1  &n\equiv 1 \\
1   &n\equiv 2 \\
-2   &n\equiv 3 \\
1   &n\equiv 4 \\
-1  &n\equiv 5 \\ \end{array}$       &
$\begin{array}{rl}
& \mbox{mod } 6 \\
6  & n\equiv 0  \\
-3   &n\equiv 1 \\
3   &n\equiv 2 \\
-6   &n\equiv 3 \\
3   &n\equiv 4 \\
-3   &n\equiv 5 \\ \end{array}$       
& 0
\\  \hline
$(3;1,1)$    & $7$ & $\begin{array}{rl}
& \mbox{mod } 6 \\
1  & n\equiv 0 \\
-2  &n\equiv 1 \\
2   &n\equiv 2 \\
-1   &n\equiv 3 \\
2   &n\equiv 4 \\
-2  &n\equiv 5 \\ \end{array}$       &
$\begin{array}{rl}
& \mbox{mod } 6 \\
5  & n\equiv 0  \\
-3   &n\equiv 1 \\
3   &n\equiv 2 \\
2   &n\equiv 3 \\
3   &n\equiv 4 \\
-3   &n\equiv 5 \\ \end{array}$       
&  $\begin{array}{rl}
& \mbox{mod } 3 \\
0  & n\equiv 0  \\
1  &n\equiv 1,2 \\ \end{array}$ 
\\  \hline
$(3;1,1,1)$ & $6$ & $\begin{array}{rl}
& \mbox{mod } 4 \\
0  & n\equiv 0  \\
3  &n\equiv 1 \\
0   &n\equiv 2 \\
-3   &n\equiv 3 \\ \end{array}$       &
$\begin{array}{rl}
& \mbox{mod } 4 \\
-2  & n\equiv 0   \\
3   &n\equiv 1 \\
2   &n\equiv 2 \\
-3   &n\equiv 3 \\ \end{array}$     
&  $\begin{array}{rl}
& \mbox{mod } 4 \\
-1  & n\equiv 0   \\
0  & n\equiv 1,3 \\
1  &n\equiv 2 \\ \end{array}$   
\\  \hline
$(3;1,1,1,1)$  & $5$  & $\begin{array}{rl}
& \mbox{mod } 4 \\
4  & n\equiv 0  \\
0  &n\equiv 1 \\
-4   &n\equiv 2 \\
0   &n\equiv 3 \\ \end{array}$       &
$\begin{array}{rl}
& \mbox{mod } 4 \\
3 & n\equiv 0  \\
0   &n\equiv 1 \\
-3   &n\equiv 2 \\
0   &n\equiv 3 \\ \end{array}$       & 0
\\  \hline
\end{tabular}\caption{The constants $c_n$ and $d_n$ used in the formul\ae\ in \eqref{eq:sn and tn}, by 
negative weight expansion; and the constants $e_n$ for  \eqref{eq:blowup Omega length}.}\label{tab:cn dn}
\end{table}
}
\end{center}


\begin{landscape}

\renewcommand{\arraystretch}{1.4}

\begin{center}
{
\begin{table}[h]
\begin{tabular}{| c ||  c | c |}
\hline
$(B;b,\dots,b)$                      & $\mathcal{L}(\Lambda_n)$       &     $\ell_\Omega(\Lambda_n)$   \\ \hline \hline
$(3)$                       & $\frac{(s_n+1)(s_n+2)}{2}$    &    $3 s_n$ \\ \hline
$(4;2,2)$                     &
 $ \begin{array}{ll}
\frac{(s_n+1)(s_n+2)-t_n(t_n+1)-t_n(t_n-1)}{2} & n\mbox{ even}\\
\frac{(s_n+1)(s_n+2)-2\cdot t_n(t_n+1)}{2}  & n\mbox{ odd}\end{array}$  
   &    $ \begin{array}{ll}
2 (s_n-(t_n-1)) + 2 (s_n-t_n) & n\mbox{ even}\\
2 (s_n-t_n) + 2 (s_n-t_n) & n\mbox{ odd}\end{array}$  \\ \hline   
$(3;1)$                       &  $\frac{(s_n+1)(s_n+2)-t_n(t_n+1)}{2}$       
&   $2(t_n)+3(s_n-t_n)$      \\ \hline    
$(3;1,1)$             &
 $ \begin{array}{ll}
 &\mbox{mod } 3\\
\frac{(s_n+1)(s_n+2)-2\cdot t_n(t_n+1)}{2}  & n\equiv 0   \\
\frac{(s_n+1)(s_n+2)-t_n(t_n+1)-t_n(t_n-1)}{2} & n\equiv 1,2 
\end{array}$  
   &    $\begin{array}{ll}
 &\mbox{mod } 3\\
2 t_n+3 (s_n-2t_n) + 2 t_n & n\equiv 0  \\
2t_n+3(s_n-2t_n+1)+2(t_n-1) &  n\equiv 1,2  \end{array}$        \\ \hline    
$(3;1,1,1)$                   &    $\begin{array}{ll}
& \mbox{mod } 4 \\
\frac{(s_n+1)(s_n+2)-2t_n(t_n+1)-(t_n+1)(t_n+2)}{2}  & n\equiv 0   \\
   \frac{(s_n+1)(s_n+2)-3 t_n (t_n+1)}{2}  & n\equiv 1,3   \\ 
\frac{(s_n+1)(s_n+2)-2t_n(t_n+1)-(t_n-1)t_n}{2} & n\equiv 2 
\end{array}$     &  $\begin{array}{ll}
& \mbox{mod } 4 \\
t_n+3(s_n-2t_n-1)+2(t_n+1) & n\equiv 0 \\
   t_n+3(s_n-2t_n)+2t_n  & n\equiv 1,3   \\ 
t_n+3(s_n-2t_n+1)+2(t_n-1) & n\equiv 2 
\end{array}$         \\ \hline   
$(3;1,1,1,1)$                       &     $\frac{(s_n+1)(s_n+2)-4 t_n(t_n+1)}{2}$ &    $t_n+3(s_n-2t_n)+t_n$                \\ \hline     
\end{tabular}
\caption{The quantities $\mathcal{L}(\Lambda_n)$ and  $\ell_\Omega(\Lambda_n)$, by negative weight expansion.  
The first, $\mathcal{L}(\Lambda_n)$, counts lattice points enclosed by $\Lambda_n$.  The second, $\ell_\Omega(\Lambda_n)$,
is a notion of length of the path, defined in \eqref{eq:Omega length}; the constant term in each expression
is $e_n$  in \eqref{eq:blowup Omega length}.}\label{tab:calL ell}
\end{table}
}
\end{center}
\end{landscape}

\newpage

\section{ATFs}\label{ap:atfs}

In this appendix, we describe the initial ATF maneuvering required to produce ATFs on the manifolds
$$
\C P^2_3\ ; \  \C P^2_3\# \overline{\C P}^2_1\ ; \  \C P^2_3\# 2\overline{\C P}^2_1\ ; \  \C P^2_3\# 3\overline{\C P}^2_1\ ;
$$
$$
 \C P^2_3\# 4\overline{\C P}^2_1\ ; \ \mbox{and } \C P^1_2\times \C P^1_2
$$
that have base diagram a triangle with two nodal rays when $J=2$ and a quadrilateral with three nodal
rays when $J=3$.

For $\C P^2_3$, $J=2$ and the moment image is already a triangle.  We must simply apply nodal
trades to add nodal rays at the two corners not at the origin.  See Figure~\ref{fig:3-pregame}.

\begin{center}
  \begin{figure}[ht]
\begin{overpic}[
scale=0.5,unit=1mm]{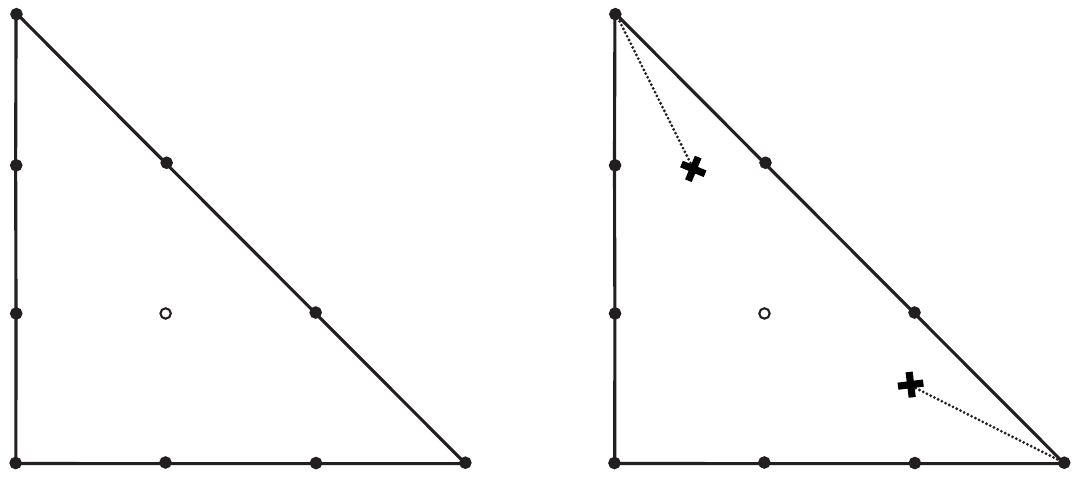}
   \put(20,-3){(a)}
   \put(70,-3){(b)}
\end{overpic}
\caption{In (a), we see the Delzant polygon for $\C P^2_3$.  From 
(a) to (b), we apply two nodal trades to add two singular fibers, creating a new 
ATF on $\C P^2_3$.}
\label{fig:3-pregame}
\end{figure}
\end{center}


For $ \C P^2_3\# \overline{\C P}^2_1$, $J=3$ and the moment image is already a quadrilateral.  We must simply apply nodal
trades to add nodal rays at the three corners not at the origin.  See Figure~\ref{fig:3-1-pregame}.

\begin{center}
  \begin{figure}[ht]
\begin{overpic}[
scale=0.43,unit=1mm]{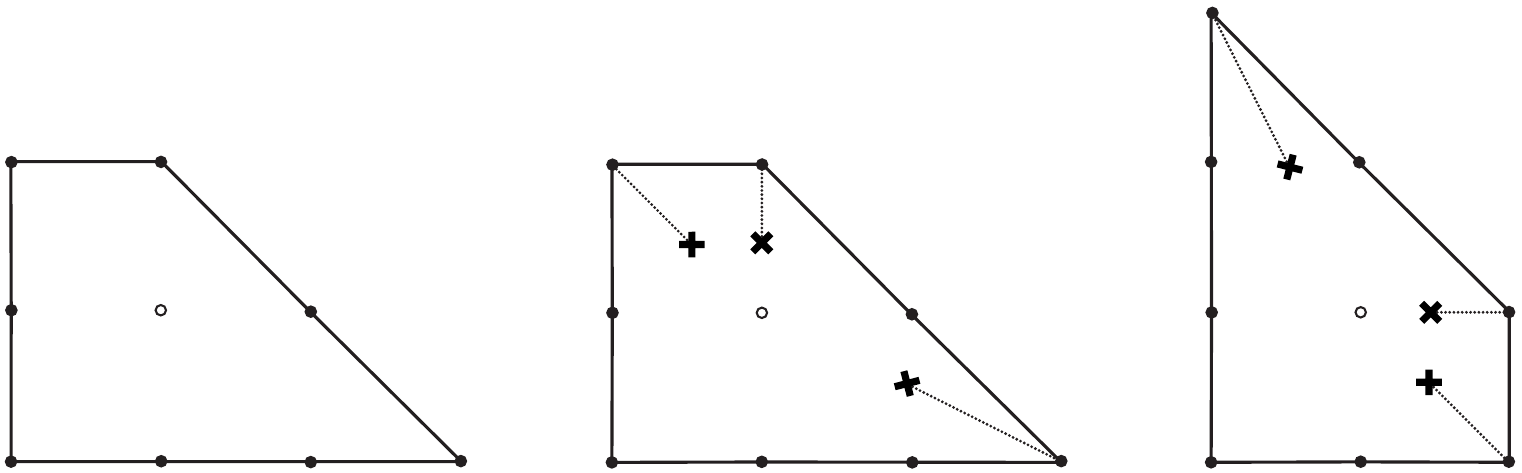}
   \put(14,-3.5){(a)}
   \put(57,-3.5){(b)}
   \put(98,-3.5){(c)}
   \put(42,24){$\mathbf v$}
   \put(111,-2){$\mathbf w$}
\end{overpic}

\vskip 0.1in

\caption{In (a), we see the Delzant polygon for $ \C P^2_3\# \overline{\C P}^2_1$.  From 
(a) to (b), we apply three nodal trades to add three singular fibers.  This represents a
new ATF on $ \C P^2_3\# \overline{\C P}^2_1$
 From (b) to (c), we apply
a mutation changing anchor vertex $\mathbf v$ in (b) to new anchor $\mathbf w$ in (c),
creating yet a third 
ATF on $ \C P^2_3\# \overline{\C P}^2_1$. The mutation matrix is $\left(\begin{smallmatrix}
0&-1\\1&2
\end{smallmatrix}\right)$. }  
\label{fig:3-1-pregame}
\end{figure}
\end{center}

\newpage

For $ \C P^2_3\# 2\overline{\C P}^2_1$, $J=3$ and the moment image is a pentagon.  
There is a sequence of  ATF moves that achieves a quadrilateral.
See Figure~\ref{fig:3-1-1-pregame}.

\begin{center}
  \begin{figure}[ht]
\begin{overpic}[
scale=0.5,unit=1mm]{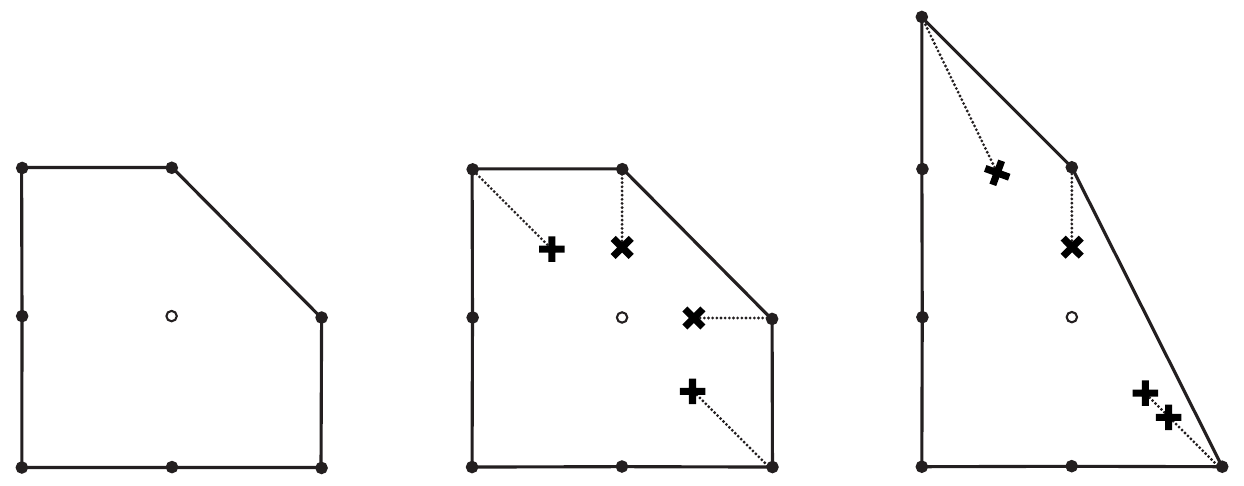}
   \put(13,-5){(a)}
   \put(50,-5){(b)}
   \put(89,-5){(c)}
   \put(37,27){$\mathbf v$}
   \put(105,-1.5){$\mathbf w$}

\end{overpic}
\vskip 0.15in

\caption{In (a), we see the Delzant polygon for $ \C P^2_3\# 2\overline{\C P}^2_1$.  From 
(a) to (b), we apply four nodal trades to add four singular fibers, creating a new 
ATF on $\C P^2_3\# 2\overline{\C P}^2_1$.  From (b) to (c), we apply
a mutation changing anchor vertex $\mathbf v$ in (b) to new anchor $\mathbf w$ in (c), 
with resulting base diagram a quadrilateral with three nodal rays, as desired. The mutation matrix is $\left(\begin{smallmatrix}
0&-1\\1&2
\end{smallmatrix}\right)$. 
In (c), two of the nodal rays have a single singular fiber and the third at $\mathbf w$ 
has two singular fibers.
This is yet a third 
ATF on $\C P^2_3\# 2\overline{\C P}^2_1$.
}
\label{fig:3-1-1-pregame}
\end{figure}
\end{center}

\newpage

For $ \C P^2_3\# 3\overline{\C P}^2_1$, $J=2$ and the moment image is a hexagon.  
There is a sequence of  ATF moves that achieves a triangle.
See Figure~\ref{fig:3-1-1-1-pregame}.

\begin{center}
  \begin{figure}[ht]
\begin{overpic}[
scale=0.5,unit=1mm]{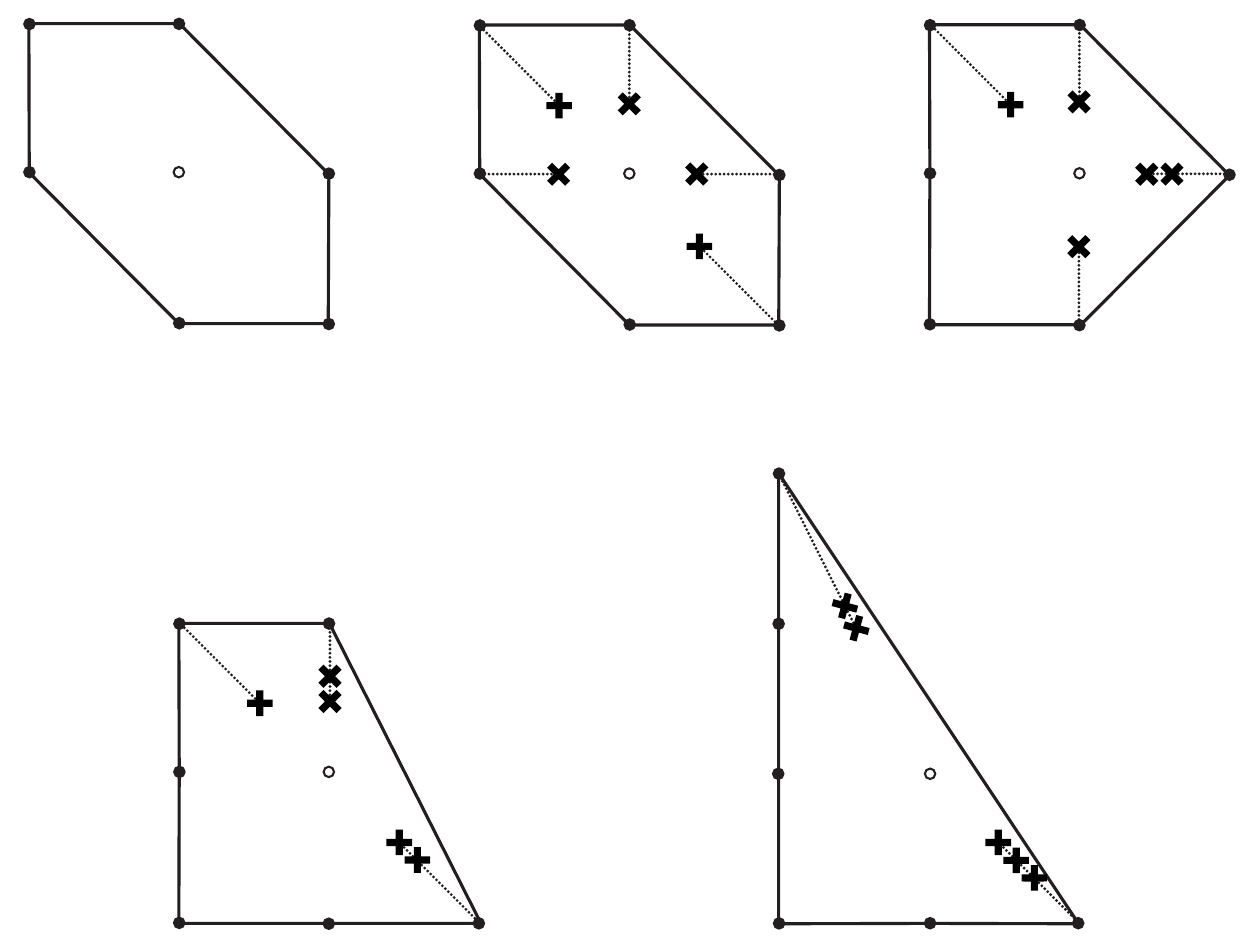}
   \put(18,47){(a)}
   \put(55,47){(b)}
   \put(85,47){(c)}
   \put(24,-3){(d)}
   \put(75,-3){(e)}
   \put(35,65){$\mathbf v_1$}
   \put(105,65){$\mathbf w_1$}
   \put(92,50){$\mathbf v_2$}
   \put(27,30){$\mathbf w_2$}
   \put(12,30){$\mathbf v_3$}
   \put(92,-){$\mathbf w_3$}
\end{overpic}

\vskip 0.1in
\caption{In (a), we see the Delzant polygon for $ \C P^2_3\# 3\overline{\C P}^2_1$.  From 
(a) to (b), we apply five nodal trades to add five singular fibers, creating a new 
ATF on $\C P^2_3\# 3\overline{\C P}^2_1$.  From (b) to (c), we apply
a mutation with anchor node changing from $\mathbf{v_1}$ to $\mathbf{w_1}$ and 
with resulting base diagram a pentagon with four nodal rays. The mutation matrix is $\left(\begin{smallmatrix}
1&1\\0&1
\end{smallmatrix}\right)$.  From (c) to (d), we
apply another mutation, now with anchor node changing from $\mathbf{v_2}$ to $\mathbf{w_2}$ 
and with resulting base diagram a quadrilateral with three nodal rays. The mutation matrix is $\left(\begin{smallmatrix}
1&0\\-1&1
\end{smallmatrix}\right)$.   Finally,
from (d) to (e), we perform a third mutation with anchor node changing from $\mathbf{v_3}$ to $\mathbf{w_3}$ 
and with resulting base diagram the desired triangle
with two nodal rays. The mutation matrix is $\left(\begin{smallmatrix}
0&-1\\1&2
\end{smallmatrix}\right)$. 
In (e), one of the nodal rays has two singular fibers and the other has three singular fibers.
}
\label{fig:3-1-1-1-pregame}
\end{figure}
\end{center}

\newpage

For $\C P^2_3\# 4\overline{\C P}^2_1$, $J=2$ but the manifold is not toric.  We begin by using Vianna's
trick \cite[\S3.2]{vianna} to find an appropriate ATF on this manifold.  Specifically, we begin with 
the ATF on $\C P^2_3\# 3\overline{\C P}^2_1$ given in Figure~\ref{fig:3-1-1-1-pregame}(e).  This ATF
has a smooth toric corner at the origin where we may perform a toric blowup of symplectic size $1$.  
In terms of the base diagram, this corresponds to chopping off a $1\times 1$ triangle at the origin.
This results in a quadrilateral with two nodal rays representing an ATF on $\C P^2_3\# 4\overline{\C P}^2_1$, 
shown in Figure~\ref{fig:3-1-1-1-1-pregame}(b).
There is then a sequence of  ATF moves that achieves a triangle with two nodal rays.
See Figure~\ref{fig:3-1-1-1-1-pregame}.


\begin{center}
  \begin{figure}[ht]
\begin{overpic}[
scale=0.48,unit=1mm]{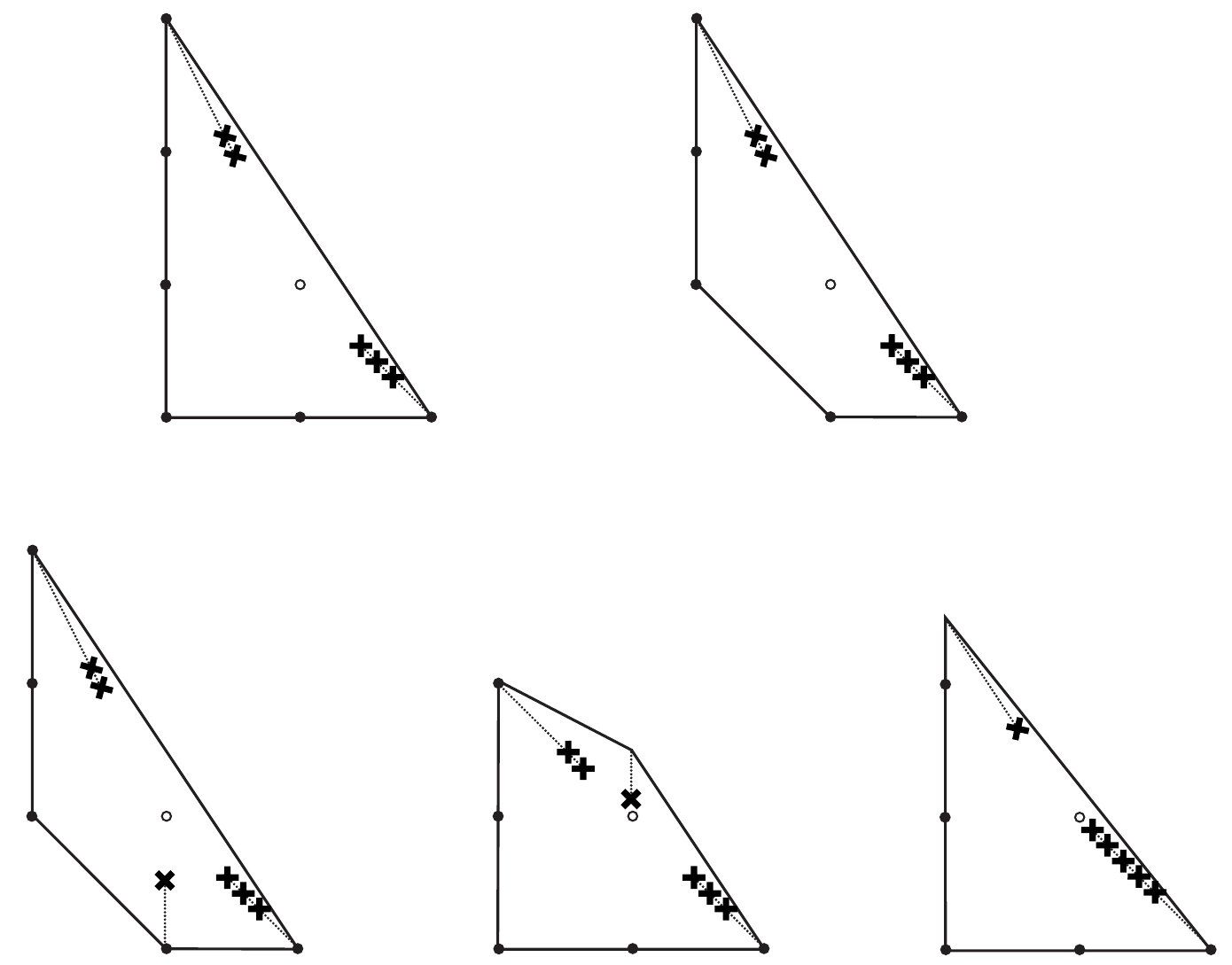}
   \put(25,44){(a)}
   \put(78,44){(b)}
   \put(15,-5){(c)}
   \put(57,-5){(d)}
   \put(100,-5){(e)}
   \put(11,46){$\mathbf O$}
   \put(71,49){$\mathbf{v_1}$}
   \put(10,0.5){$\mathbf{v_1}$}
   \put(59,20){$\mathbf{w_1}$}
   \put(42,28){$\mathbf{v_2}$}
   \put(110,-3){$\mathbf{w_2}$}
\end{overpic}

\vskip 0.2in
\caption{In (a), we see the base diagram for the ATF on $ \C P^2_3\# 3\overline{\C P}^2_1$ 
illustrated in \ref{fig:3-1-1-1-pregame}(e).  
From (a) to (b), we apply a toric blowup of size $1$ at the point 
labeled $\mathbf O$, resulting in an
almost toric fibration on $\C P^2_3\# 4\overline{\C P}^2_1$.  From (b) to (c), we apply
one nodal trade at the point labeled $\mathbf{v_1}$. From (c) to (d), we
apply a mutation with anchor node changing from $\mathbf{v_1}$ to $\mathbf{w_1}$ and  
with resulting base diagram a quadrilateral with three nodal rays. The mutation matrix is $\left(\begin{smallmatrix}
1&0\\1&1
\end{smallmatrix}\right)$.  Finally,
from (d) to (e), we perform a second mutation with anchor node changing from $\mathbf{v_2}$ to $\mathbf{w_2}$ and 
with resulting base diagram the desired triangle
with two nodal rays. The mutation matrix is $\left(\begin{smallmatrix}
3&2\\-2&-1
\end{smallmatrix}\right)$. 
In (e), one of the nodal rays has one singular fiber and the other has five singular fibers.
}
\label{fig:3-1-1-1-1-pregame}
\end{figure}
\end{center}

\newpage

For $\C P^1_2\times \C P^1_2$, $J=2$ and the moment image is a quadrilateral. There is a sequence of 
ATF moves that achieves a triangle, shown in Figure~\ref{fig:4-2-2-pregame}

\begin{center}
  \begin{figure}[ht]
\begin{overpic}[
scale=0.5,unit=1mm]{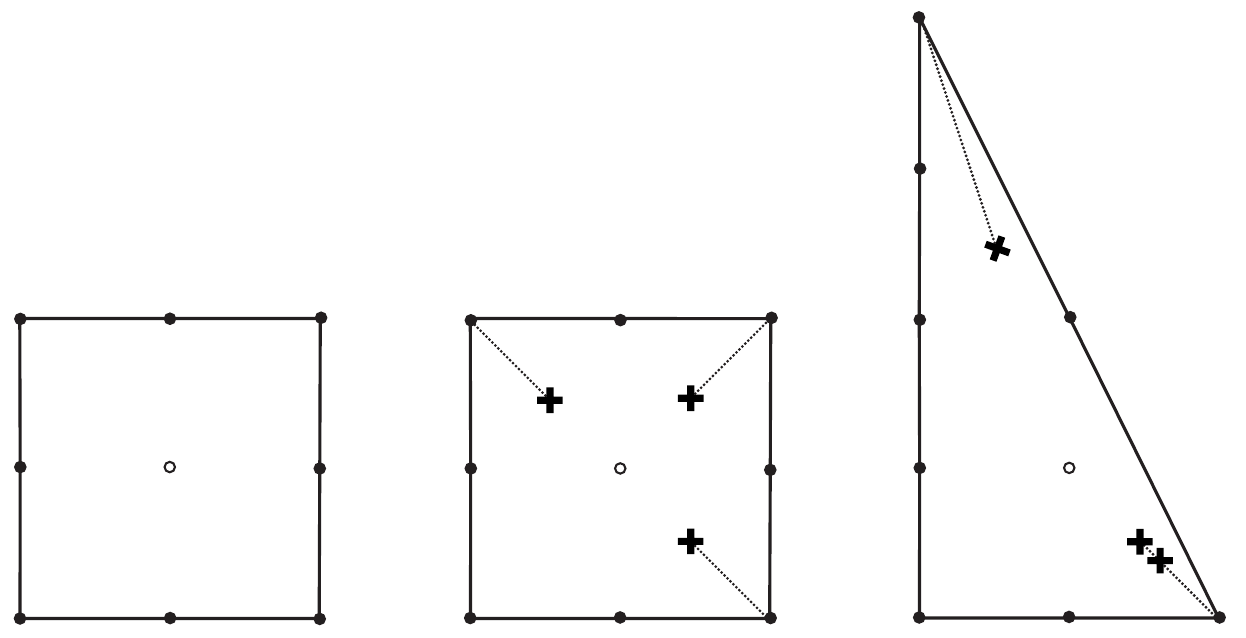}
   \put(12,-5){(a)}
   \put(50,-5){(b)}
   \put(89,-5){(c)}
   \put(36,27){$\mathbf{v}$}
   \put(104,-2){$\mathbf{w}$}
\end{overpic}

\vskip 0.2in

\caption{In (a), we see the Delzant polygon for $\C P^1_2\times \C P^1_2$.  From 
(a) to (b), we have applied three nodal trades to add three singular fibers, creating a new 
ATF on $\C P^1_2\times \C P^1_2$.  Finally, from (b) to (c), we apply
a mutation with anchor node changing from $\mathbf{v}$ to $\mathbf{w}$ and   
with resulting base diagram a triangle with two nodal rays, as desired. The mutation matrix is $\left(\begin{smallmatrix}
0&-1\\1&2  \end{smallmatrix}\right)$.  In
(c), one of the nodal rays has a single singular fiber and the other has two singular fibers.}
\label{fig:4-2-2-pregame}
\end{figure}
\end{center}

\newpage

\section{Behind the scenes}\label{appendix}

In this section, we give an account of how we found the six negative weight expansions that appear in
Theorem \ref{thm:main}. At the beginning of this project,
the ellipsoid embedding functions for the ball, polydisk, and ellipsoid $E(2,3)$
were known to have infinite staircases. These correspond to the blowup 
vectors $(1)$, $(4;2,2)$, and $(3;1,1,1)$.

By Theorem~\ref{satisfies quadratic equation}, we knew where the accumulation 
point would occur for any domain, if an infinite staircase were to exist. 
We wrote Mathematica code that generates an approximation of the graph 
of $c_X(a)$ for a given $X$, and started by trying a number of different integer 
negative weight expansions. By chance, we first tried $(3;1)$ and found an infinite staircase.
The vector $(3;1,1)$ admitted one too, and we were off, trying to prove that there
was always an infinite staircase.  The actual answer, of course, has turned out
to be more subtle.
The code we used in our early searches is included below and the 
notebook is available \cite{code}.

The idea behind the code is that the ellipsoid embedding function can 
be computed as the supremum of ratios of ECH capacities, as in 
equation \eqref{def:csup}. We compute a large (but finite!) number of ECH 
capacities of the domain $X$ using the sequence subtraction operation of 
Definition \ref{def:seqsum}. We also compute a large but finite number of 
ECH capacities of $E(1,a_i)$, for equally spaced values $a_i$ within a given 
range. Next, for each $a_i$ we find the maximum of the ratios of the 
computed ECH capacities, obtaining a list of points $(a_i,\tilde{c}_X(a_i))$. 
Our code then approximates the graph by connecting the dots. 
Figure~\ref{fig:4pics} illustrates four examples.

The graph of $\tilde{c}_X(a)$ is an approximation of the graph of the embedding function $c_X(a)$ in two senses: 
we are only using a finite number of points in the domain, and the computed 
values $\tilde{c}_X(a_i)$ are not completely accurate because we have restricted to
a finite list of ECH capacities. Nonetheless, the approximation does allow us to 
visually rule out certain domains from the possibility of having an infinite staircase.
For example, in the bottom left graph in Figure~\ref{fig:4pics}(c), there clearly exists 
an obstruction where the infinite staircase would have to accumulate, so that blowup
vector must not admit an infinite staircase. In cases where it is more ambiguous, 
we change the range of points $a_i$, ask the code to compute more ECH capacities, and hence zoom in on the graph to probe further. 
In other cases, more careful analysis may be needed.
The bottom right graph in Figure~\ref{fig:4pics}(d) provides such a case.
At the accumulation point itself, there is no obstruction and the capacity function
equals the volume lower bound.  As illustrated in the figure, there is an
obstruction to the left of the accumulation point which prevents 
an ascending staircase approaching the accumulation point from 
the left.  On the other side of the accumulation point, 
we may use the fact that the blowup
vector of the ellipsoid $E(3,4)$ is $(4;1,1,1,1)$, as indicated in 
Figure~\ref{fig:4-1-1-1-1-is-ellipsoid}, to calculate the capacity function.  
The first named author has
provided a detailed analysis \cite[\S 2.5]{cg ellipsoid} 
to show that for an interval 
after the accumulation point,
the embedding function for this blowup vector
is $\frac{a+3}{4}$, so there is no staircase to the right of the
accumulation point either.

\newpage

\begin{center}
  \begin{figure}[ht]
\includegraphics[width=1.5in]{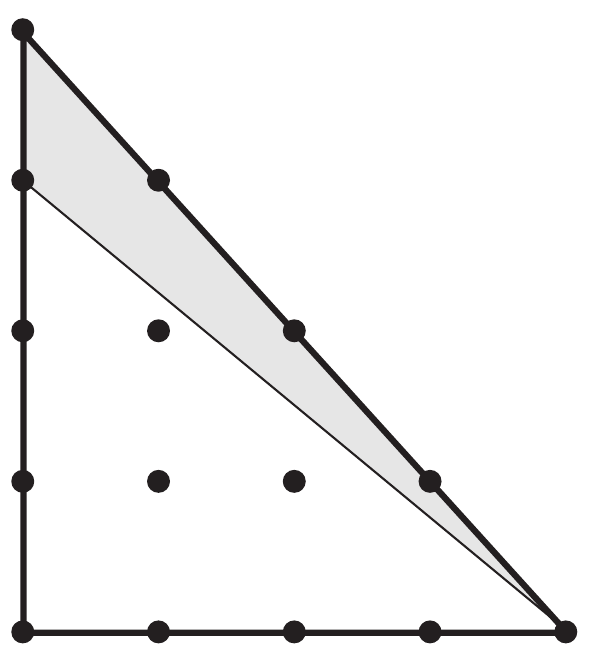}
\caption{This figure shows the triangle corresponding to a ball
with capacity $4$.  The shaded region corresponds to an ellipsoid
$E(1,4)$ or equivalently four balls of capacity $1$.  Thus, the 
remaining ellipsoid $E(3,4)$ has blowup vector $(4;1,1,1,1)$.}
\label{fig:4-1-1-1-1-is-ellipsoid}
\end{figure}
\end{center}

Whenever this zooming in process suggested that there indeed exists an infinite staircase for that domain, the next step was to find the coordinates of the inner and outer corners of the staircase. Recall that in the ball case these are ratios of certain Fibonacci numbers, so we were interested in obtaining a recurrence sequence from the numerators and denominators of these coordinates. We used the Mathematica function Rationalize to approximate the values $\tilde{c}_X(a_i)$ by fractions with small denominators and then fed the integer sequences obtained into OEIS, the Online Encyclopedia of Integer Sequences, sometimes unearthing unexpected connections\footnote{\phantom{.}In one instance, the integer sequence that came up on the OEIS search engine was sequence A007826: numbered stops on the Market-Frankford rapid transit (SEPTA) railway line in Philadelphia, PA USA. This constitutes possibly the first ever application of symplectic geometry to mass transit.}. Eventually we switched to using the function FindLinearRecurrence on Mathematica to find the linear recurrence for the sequences found.

\subsection*{The Mathematica Code} $\phantom{boo}$

\vskip 0.2in

\noindent List of first $\left(\frac{  (\floor{\frac{k a}{b}}+1)(k+1)  }   {2}-1\right)$ ECH capacities of  $E(a,b)$, we usually set $k=100$:

\lstset{language=Mathematica} 
\begin{lstlisting}
ECHellipsoid[a_, b_, k_] := 
  Module[{l = Floor[(k + 1) Floor[1 + k a/b]/2] - 1}, 
   Take[Sort[Flatten[Table[N[m a + n b], {m, 0, k}, 
    {n, 0, k}]]], l]];
\end{lstlisting}

\noindent List of first $\left(\frac{(k-1)^2}{2}\right)$ ECH capacities of the ball $E(1,1)$, usually we set $k=100$:

\begin{lstlisting}
ECHball[k_] := 
  Take[Sort[Flatten[N[Array[Array[k - # &, #] &, k]]]], 
   Floor[(k - 1)^2/2]];
\end{lstlisting}

\noindent Sequence subtraction operation:

\begin{lstlisting}
auxlist[list1_, list2_, i_] := 
  Block[{a = list1, b = list2, l},
    l = Min[Length[a], Length[b]]; 
    Array[a[[i + #]] - b[[#]] &, l - i]];

minus[list1_, list2_] := 
  Block[{a = list1, b = list2, l}, 
    l = Min[Length[a], Length[b]];
    Array[Min[auxlist[a, b, # - 1]] &, l]];
\end{lstlisting}

\noindent Ellipsoid embedding function $c_X(a)$ where ECHlist is (the beginning of) the sequence of ECH capacities of $X$:

\begin{lstlisting}
c[a_, ECHlist_] := Module[{p = Length[ECHlist], k, l, nn},
   k = Floor[(Sqrt[a^2 + 6 a + 1 + 8 a p] - 1 - a)/2];
   nn = ECHellipsoid[1, a, k];
   l = Min[p, Length[nn]];
   Max[Array[nn[[# + 1]]/ECHlist[[# + 1]] &, l - 1]]];
\end{lstlisting}

\noindent Creates a list of points $(a_i,c_X(a_i))$ where ECHlist is (the beginning of) the sequence of ECH capacities of $X$:

\begin{lstlisting}
clist[amin_, amax_, astep_, ECHlist_] := 
  Block[{l = Floor[(amax - amin)/astep]}, 
    Array[
     {amin + (# - 1) astep, c[amin + (# - 1) astep, ECHlist]} 
     &, l]];
\end{lstlisting}

\noindent Plots the volume curve, plus a vertical line at the location of the potential accumulation point:

\begin{lstlisting}
constraint[amin_, amax_, vol_, acc_] := 
  Plot[Sqrt[t/vol], {t, amin, amax}, PlotStyle -> Red, 
   Epilog -> {InfiniteLine[{acc, 0}, {0, 1}]}];
\end{lstlisting}

\begin{example} 
Let $X$ have negative weight expansion $(4;2,1)$. Make changes as necessary. 
Accpoint  gives the location of the potential accumulation point, so that one can choose the range of points to plot. 
The output of this example is in Figure \ref{fig:4pics}.

\begin{lstlisting}
ECHcap = Block[{seq = ECHball[100]}, 
    minus[minus[4*seq, 2*seq], seq]]; 
    
per = 3*4 - 2 - 1;    (* per = 3 b - sum b_i *)
vol = 4^2 - 2^2 - 1^2; (* vol = b^2 - sum b_i^2 *)
accpoint =  
    x /. NSolve[x^2 + (2 - per^2/vol ) x + 1 == 0, x][[2]]  

5.17022

aamin = 1;
aamax = 6;
aastep = 0.01;
Show[{constraint[aamin, aamax, vol, accpoint], 
  ListPlot[clist[aamin, aamax, aastep, ECHcap], 
    Joined -> True]}]
\end{lstlisting}
\end{example}


\newpage




\begin{thebibliography}{MMM}

\bibitem{audin}
M.\ Audin, {\em Torus actions on symplectic manifolds. Second revised edition.} 
{\bf Progress in Mathematics} {\bf 93}.  Birkh\"auser Verlag, Basel, 2004, MR2091310, Zbl 0726.57029.


\bibitem{br}
M. Beck and S. Robins, ``Computing the continuous discretely. Integer-point enumeration in polyhedra'' 
second edition, 
{\bf Undergraduate Texts in Mathematics}, Springer, New York, 2015, MR3410115, Zbl 1339.52002.

\bibitem{project9}
M. Bertozzi, T. Holm, E. Maw, G. Mwakyoma, D. McDuff, A.R. Pires, and M. Weiler, ``Infinite staircases for Hirzebruch surfaces.'' {\em Research directions in symplectic and contact geometry and topology}, 
47--157, Assoc. Women Math. Ser., {\bf 27}, Springer,
MR4417715, Zbl 1512.53080.

\bibitem{biran} P. Biran, ``Symplectic packing in dimension 4", {\em GAFA}, {\bf 7} (1997), 420--437,
MR1466333, Zbl 0892.53022.

\bibitem{bs}
T.C. Brown and P.J.-S. Shiue, ``Sums of fractional parts of integer multiples of an irrational'', 
{\em J. Number Theory} {\bf 50} (1995), 181--192, MR1316813, Zbl 0824.11041.

\bibitem{bho}
O. Buse, R. Hind, and E. Opshtein, ``Packing stability for symplectic four-manifolds", 
{\em Trans. Amer. Math. Soc.} {\bf 368} (2016), 8209--8222, MR3546797, Zbl 1354.53081.


\bibitem{CV}
R.\ Casals and R.\ Vianna, ``Full ellipsoid embeddings and toric m1utations." 
{\em Selecta Math.} (N.S.) {\bf 28} (2022), no.3, Paper No.\ 61, 62 pp.
MR4414137, Zbl 07523718.


\bibitem{chls}
K.\ Cieliebak, H.\ Hofer, J.\ Latschev, F.\ Schlenk.
``Quantitative symplectic geometry."
{\em Dynamics, ergodic theory, and geometry}, 1--44, {\bf Math.\ Sci.\ Res.\ Inst.\ Publ.}, {\bf 54}, 
Cambridge Univ. Press, Cambridge, 2007.  MR2369441, Zbl 1143.53341.


\bibitem{dan}
D.\ Cristofaro-Gardiner, ``Symplectic embeddings from concave toric domains into convex ones'', with an 
Appendix by the author and K.Choi, {\em J. Differential Geom.} 112 (2019), 199--232, MR3960266, Zbl 1440.53084.


\bibitem{cg ellipsoid}
D.\ Cristofaro-Gardiner, ``Special eccentricities of four-dimensional ellipsoids."
{\em Algebr. Geom. Topol.} {\bf 22} (2022), no. 5, 2267--2291, MR4503337, Zbl 1515.53084.

\bibitem{cgfs polydisk}
D.\ Cristofaro-Gardiner, D. Frenkel, and F. Schlenk,
``Symplectic embeddings of four-dimensional ellipsoids into integral polydiscs." 
{\em Alg.\ Geom.\ Top.}\  {\bf 17} (2017) 1189--1260, MR3623687, Zbl 1362.53085.

\bibitem{cghind}
D. Cristofaro-Gardiner and R. Hind, ``Symplectic embeddings of products", {\em Comm. Math. Helv.}, {\bf 93} (2018), 1--32, MR3777123, Zbl 1398.53083.

\bibitem{ghost}
D. Cristofaro-Gardiner, R. Hind, D. McDuff, ``The ghost stairs stabilize to sharp symplectic embedding obstructions", 
{\em J. Topology}, {\bf 11} (2018), 309--378,  MR3789827, Zbl 1394.53080.

\bibitem{code}
D. Cristofaro-Gardiner, T. S. Holm, A. R. Pires, and A. Mandini, Mathematica code, 
posted with the
\texttt{arXiv} version of this manuscript, \texttt{arXiv:2004.13062}.

\bibitem{asymptotics}D. Cristofaro-Gardiner, M. Hutchings, and V. Ramos, 
``The asymptotics of ECH capacities'', {\em Invent. Math.} {\bf 199} (2015), no. 1, 187--214,
MR3294959, Zbl 1315.53091.


\bibitem{cgk}
D. Cristofaro-Gardiner and A. Kleinman, ``Ehrhart polynomials and symplectic embeddings of ellipsoids."  
{\em J. Lond. Math. Soc. } (2) {\bf 101} (2020), no. 3, 1090--1111,
MR4111936, Zbl 1452.53066.

\bibitem{cgls}
D. Cristofaro-Gardiner, T. Li, and R. Stanley, ``New examples of period collapse."
Preprint \texttt{arXiv:1509.01887}.


\bibitem{cgs}
D. Cristofaro-Gardiner and N. Savale, ``Subleading asymptotics of ECH capacities." 
Selecta Math. (N.S.) 26 (2020), no. 5, Paper No. 65, 34 pp.
MR4157012, Zbl 1451.53114.


\bibitem{delzant}
T.\ Delzant,
``Hamiltoniens p\'eriodiques et images convexes de l'application moment." 
[Periodic Hamiltonians and convex images of the momentum mapping]
{\em Bull.\ Soc.\ Math.\ France} {\bf 116} (1988), no. 3, 315--339,
MR0984900, Zbl 0676.58029.


\bibitem{ev} M. Entov and M. Verbitsky, ``Unobstructed symplectic packing for tori and hyperkahler manifolds", {\em J. Topol. Anal.} {\bf 08} (2016) 589--626,
MR3545014, Zbl 1353.53055.


\bibitem{evans}
J. Evans, {\em Lagrangian torus fibrations}.  Unpublished lecture notes   \texttt{http://jde27.uk/misc/ltf.pdf}.
dated 8 February 2019.
Retrieved 22 January 2020.

\bibitem{many}
C. Farley, T. Holm, N. Magill, J. Schroder, M. Weiler, Z. Wang and E. Zabelina, ``Four-periodic infinite staircases for four-dimensional polydisks", Preprint \texttt{arXiv:2210.15069}.



\bibitem{frenkelmuller}
D. Frenkel and D. M\"uller, ``Symplectic embeddings of 4-dimensional ellipsoids into cubes." 
{\em J. Symplectic Geom.} {\bf 13} (2015), no. 4, 765--847,
MR3480057, Zbl 1339.53082.




\bibitem{gu}
J. Gutt and M. Usher, ``Symplectically knotted codimension-zero embeddings of domains in $\mathbb{R}^4$", 
{\em Duke Math J.}, {\bf 168.12} (2019) 2299--2363,
MR3999447, Zbl 1481.53106.

\bibitem{HL}
G.H.\ Hardy and J.E.\ Littlewood,
``Some problems of Diophantine approximation: The lattice-points of a right-angled triangle."
{\em Proc.\ London Math.\ Soc.\ } {\em Ser.\ 2} {\bf 20} {\em no.\ 1378}, (1920) 15--36,
MR1577360, JFM 48.0187.01.

\bibitem{HLanother}
G.H.\ Hardy and J.E.\ Littlewood,
``Some problems of Diophantine approximation: The lattice-points of a right-angled triangle (second memoir)."
{\em Abh.\ Math.\ Sem.\ Hamburg} {\bf 1} (1922), 212--249,
MR3069402, JFM 48.0197.07.



\bibitem{hausmann knutson}
J.-C$\ell$.\ Hausmann and A.\ Knutson,
``A limit of toric symplectic forms that has no periodic Hamiltonians."
{\em GAFA}
{\bf 10} (2000) 556--562,
MR1779612, Zbl 0984.53031.

\bibitem{hecke}
E. Hecke, ``{\"U}ber analytische Funktionen und die Verteilung von Zahlen mod. eins.", {\em  Abh. Math. Sem. Hamburg} {\bf 1} (1922) 54-76, MR3069388, JFM 48.0197.03.


\bibitem{hind}
R. Hind, ``Some optimal embeddings of symplectic ellipsoids", {\em J. Topology}, {\bf 8} (2015), 871--883,
MR3394319, Zbl 1325.53110.


\bibitem{hindkerman}
R. Hind and E. Kerman, ``New obstructions to symplectic embeddings", 
{\em Invent. Math.}, {\bf 196} (2014), 383--452,
MR3193752, Zbl 1296.53160.

\bibitem{hk} T. Holm and L. Kessler, ``Circle actions on symplectic 4-manifolds", {\em Comm. Ana. and Geo.}, {\bf 27}
No.\ 2, 421--464, MR4003013, Zbl 1489.53113.

\bibitem{h} M. Hutchings, ``Recent progress on symplectic embedding problems.''
{\em Proc. Natl. Acad. Sci. USA} {\bf 108} (2011), no. 20, 8093--8099, MR2806644, Zbl 1256.53054.

\bibitem{qech} M. Hutchings, ``Quantitative embedded contact homology.''
{\em J. Differential Geom.} {\bf 88} (2011), no. 2, 231--266,
MR2838266, Zbl 1238.53061.



\bibitem{h2} M. Hutchings, ``Fun with ECH capacities, II'', 

https://floerhomology.wordpress.com/2011/10/12/fwec2/.

\bibitem{h3} M. Hutchings, ``Lecture notes on embedded contact homology." 
{\em Contact and symplectic topology}, 389--484, {\em Bolyai Soc. Math. Stud.}, {\bf 26}, J\'anos Bolyai Math. Soc., Budapest, 2014.  MR3220947, Zbl 1432.53126.

\bibitem{kapovich-millson}
M. Kapovich and J. Millson, 
``The symplectic geometry of polygons in Euclidean space."
{\em J. Differential Geom.\ } {\bf 44} (1996), no.\ 3, 479--513, MR1431002, Zbl 0889.58017.


\bibitem{kkp} Y. Karshon, L. Kessler, and M. Pinsonnault, ``A compact symplectic four-manifold admits only finitely many inequivalent toric actions", {\em J. Symp.\ Geom.}, {\bf 5} (2007) no.\ 2, 139--166, MR2377250, Zbl 1344.53065.

\bibitem{KL:noncpt toric}
Y.\ Karshon and E.\ Lerman, ``Non-Compact Symplectic Toric Manifolds." {\em SIGMA}.	
{\em Symmetry Integrability Geom. Methods Appl.} {\bf 11} (2015), Paper 055, 37 pp.
MR3371718, Zbl 1328.53103.

\bibitem{KW}
A.\ Kasprzyk and B.\ Wormleighton,
``Quasi-period collapse for duals to Fano polygons: an explanation arising from algebraic geometry."
Preprint \texttt{arXiv:1810.12472}.

\bibitem{lms} J. Latschev, D. McDuff, and F. Schlenk, ``The Gromov width of $4$-dimensional tori", Geom. Topol. {\bf 17.5} (2013) 2813--2853, MR3190299, Zbl 1277.57024.

\bibitem{lerch}
M.\ Lerch, ``Question 1547.'' {\em L'Interm\'ediaire des math\'ematiciens} {\bf 11} (1904), 144--145.

\bibitem{lerman-tolman}
E.\ Lerman and S.\ Tolman,
``Hamiltonian torus actions on symplectic orbifolds and toric varieties."
{\em Trans. Amer. Math. Soc.} {\bf 349} (1997), no. 10, 4201--4230, MR1401525, Zbl 0897.58016.

\bibitem{leung symington}
N.C.\ Leung and M.\ Symington, ``Almost toric symplectic four-manifolds."
{\em J.\ Symp.\ Geom.}\ {\bf 8} (2010), no. 2, 143--187, MR2670163, Zbl 1197.53103.

\bibitem{li} T.J. Li, ``Smoothly embedded spheres in symplectic 4-manifolds", {\em Proc. Amer. Math. Soc.} {\bf 127} (1999), no. 2, 609--613, MR1459135, Zbl 0911.57018. 

\bibitem{m} N. Magill, ``Unobstructed embeddings in Hirzebruch surfaces", arXiv:2204.12460.

\bibitem{mm} N. Magill and D. McDuff, ``Staircase symmetries in Hirzebruch surfaces", {\em Algebr. Geom. Topol.} {\bf 23} (2023), no. 9, 4235--4307, MR4670996, Zbl 07775433.

\bibitem{mmw} N. Magill, D. McDuff and M. Weiler, ``Staircase Patterns in Hirzebruch surfaces", arXiv:2203.06453.

\bibitem{mtalk} N. Magill, D. McDuff, and M. Weiler ``Recursive patterns in Hirzebruch surfaces", talk at Western Hemisphere Virtual Symplectic Seminar, July 2021, available online.

\bibitem{inprep} N. Magill, A.R. Pires, and M. Weiler ``A classification of infinite staircases for Hirzebruch surfaces'', arXiv:	arXiv:2308.08065.

\bibitem{maw-thesis}
E.\ Maw,  {\em Symplectic topology of some surface singularities.}  PhD Thesis.

\bibitem{m1} D. McDuff, ``The structure of rational and ruled symplectic 4-manifolds", J. Amer. Math. Soc {\bf 3} (1990), 679 - 712, MR1049697, Zbl 0723.53019.

\bibitem{mcd} D. McDuff, ``Symplectic embeddings of 4-dimensional ellipsoids.''
{\em J. Topol.} {\bf 2} (2009), no. 1, 1--22, MR2499436, Zbl 1166.53051.

\bibitem{m2} D. McDuff, ``The Hofer conjecture on symplectic embeddings of ellipsoids.''
{\em J. Diff. Geom.} {\bf 88} (2011) 519--532, MR2844441, Zbl 1239.53109.

\bibitem{m3} D. McDuff, ``A remark on the stabilized symplectic embedding problem for ellipsoids", 
{\em Eur. J. Math}, {\bf 4} (2018) 356--371, MR3782228, Zbl 1393.53080.

\bibitem{m4} D. McDuff, ``Hamiltonian $S^1$- manifolds are uniruled", Duke Math. J. 146  (3), 449 - 507, MR2484280, Zbl 1183.53080.

\bibitem{mcduffschlenk}
D.McDuff, F.Schlenk, ``The embedding capacity of 4-dimensional symplectic ellipsoids.'' 
{\em Ann. of Math. (2)}, {\bf 175} (2012), no. 3, 1191--1282, MR2912705, Zbl 1254.53111.

\bibitem{ostr} 
A. Ostrowski, ``Bemerkungen zur theorie der Diophantischen approximationen'', {\em Abh.\ Math.\ Sem.\ Univ.\ Hamburg} {\bf 1} (1922), 249--250, MR3069403, JFM 48.0185.02.

\bibitem{kyler1}
K. Siegel, ``Computing higher symplectic capacities I", {\em Int. Math. Res. Not.} 2022, no. 16, 12402---12461, MR4466005, Zbl 07573380.

\bibitem{kyler2}
K. Siegel, ``Higher symplectic capacities", Preprint \texttt{arXiv:1902.01490}.

\bibitem{sierp}
W. Sierpi\'nski, ``Un th\'eor\`eme sur les nombres irrationnels'', {\em Krakau. Anz.} {\bf [2]}, (1909),
725-727, JFM 40.0220.04.


\bibitem{sos}
V. T. S\'os, ``On the Theory of Diophantine Approximations I (on a problem of A. Ostrowski)'', {\em Acta
Math. Sci. Hungar. } {\bf 8} (1957), 461--472, MR0093510, Zbl 0080.03503.

\bibitem{lightcone}
M.Steele, ``The Cauchy-Schwarz master class. An introduction to the art of mathematical inequalities.'' 
{\em AMS/MAA Problem Books Series.} 
Cambridge University Press, Cambridge, 2004. x+306 pp, MR2062704, Zbl 1060.26023.

\bibitem{sylvester}
J.J. Sylvester, ``Sur la fonction $E(x)$'', {\em Comptes Rendus}, {\bf 50} (1860), 732--734 ( {\em Collected math. papers}, {\bf 2}, 179--180).

\bibitem{symington}
M. Symington, ``Four dimensions from two in symplectic topology", Topology and geometry of manifolds (Athens, GA, 2001), {\em Proc. Sympos. Pure Math.} {\bf 71} 153--208,  Amer. Math. Soc., Providence, RI, 2003, MR2024634, Zbl 1049.57016.


\bibitem{usher} M. Usher, ``Infinite staircases in the symplectic embedding problem for four-dimensional ellipsoids into polydisks''
{\em Algebr. Geom. Topol.} {\bf 19} (2019), no. 4, 1935--2022,  MR3995022, Zbl 1429.53096.


\bibitem{vianna}
R. Vianna,  ``Infinitely many monotone lagrangian tori in del Pezzo surfaces."
{\em Selecta Math. (N.S.)} {\bf 23} (2017), no. 3, 1955--1996, MR3663599, Zbl 1377.53101.

\bibitem{wormleighton} Ben Wormleighton, ``Algebraic capacities" 
{\em Selecta Math. (N.S.)} {\bf 28}, (2022) no.1, 1--45, MR4344145, Zbl 07437394.

\bibitem{zung}
N.T.\ Zung, ``Symplectic topology of integrable Hamiltonian systems. II. Topological classification", {\em Compositio
Math.}\ {\bf 138} (2003) 125--156, MR2018823, Zbl 1127.53308.

\end{thebibliography}
\end{document}